\documentclass[paper=a4, fontsize=11pt]{scrartcl} 

\usepackage[T1]{fontenc} 
\usepackage{fourier} 
\usepackage[english]{babel} 
\usepackage{amsmath,amsfonts,amsthm} 
\usepackage{amsmath, calligra, mathrsfs}
\usepackage{amssymb}
\usepackage{mathtools}
\usepackage{hyperref}
\usepackage{mathdots}
\usepackage{array}
\usepackage[mathscr]{euscript}
\usepackage{verbatim}
\usepackage[dvipsnames]{xcolor}
\usepackage{mdframed} 
\usepackage{soulutf8}
\usepackage{stmaryrd}
\usepackage{tikz-cd}
\usepackage{hyperref}
\usepackage{lipsum} 

\usepackage{sectsty} 
\allsectionsfont{\centering \normalfont\scshape} 

\usepackage{fancyhdr} 
\usepackage{xurl}
\newcommand*{\sheafhom}{\mathcal{H}\kern -.5pt om}
\numberwithin{equation}{section} 
\numberwithin{figure}{section} 
\numberwithin{table}{section} 

\newtheorem{thm}{Theorem}[section]
\newtheorem{cor}[thm]{Corollary}
\newtheorem{prop}[thm]{Proposition}

\newtheorem{lem}[thm]{Lemma}

\theoremstyle{definition}
\newtheorem{defn}[thm]{Definition}

\newtheorem{exmp}[thm]{Example}

\theoremstyle{remark}
\newtheorem{rem}[thm]{Remark}

\DeclareMathOperator{\lk}{lk}

\newcommand{\horrule}[1]{\rule{\linewidth}{#1}} 

\title{	
	\normalfont \normalsize 
	\textsc{} \\ [25pt] 
	\horrule{0.5pt} \\[0.4cm] 
	\huge Gamma positivity, PL homeomorphism types, and orthogonal polynomials

	\horrule{2pt} \\[0.5cm] 
}

\author{Soohyun Park \\ \href{mailto:lalaland.cappelletti@gmail.com}{lalaland.cappelletti@gmail.com} } 

\date{\normalsize March 26, 2026} 

\begin{document}
	
	\maketitle

	\begin{abstract}
		
		\noindent Using preservations of piecewise linear (PL) homeomorphism types under edge contractions (the link condition) as a topological proxy for flagness, we give a quantitative description of the effect flagness on on gamma positivity of simplicial spheres. In particular, we show that the link condition has a trivial effect on the $g$-vectors (and thus gamma vectors) of high-dimensional simplicial spheres with nonnegative gamma vectors in many cases. Note that this reflects a dichotomy between quantitative behavior arising from $g_1$ components (e.g. measuring ``net number of edge subdivisions'' from the boundary of a cross polytope) that are linear in the dimension and those that are superlinear in the dimension. When the link condition is nontrivial, we show that it gives a lower bound for growth rates of $g$-vector components. This lower bound increases as the number of edges and the distance of the $M$-vector condition on $g$-vectors of simplicial spheres from equality decrease. These lower bounds translate to ones on top gamma vector components and give lower bounds on gamma vector growth rates when the gamma vector components are dominant terms in the $g$-vector components with the same index (e.g. $g$-vectors with components increasing quickly compared to the dimension). Finally, we show that the same results apply to positivity properties generalizing gamma positivity arising from connections between orthogonal polynomials and lattice paths. In the course of doing this, we describe gamma vector components in terms of monomer/dimer covers and point out connections between repeated (stellar) edge subdivisions (Tchebyshev subdivisions) and dimer covers. 
	\end{abstract}
	
	\section*{Introduction}

	Our main objective is to understand underlying reasons behind initial assumptions made in combinatorial positivity properties of polynomials and their effects on quantitative behavior. The positivity property we are concerned with is gamma positivity. This is a positivity property of (reciprocal/palindromic) polynomials lying between unimodality and real-rootedness that occurs in many different contexts ranging from permutation statistics to signatures of toric varieties (\cite{Athgam}, \cite{LR}). Many of these examples are tied together by flag simplicial spheres \cite{Athgam}. A natural question to ask what a natural motivation for the flag and sphere assumptions would be. As mentioned in \cite{Gal} and \cite{LR}, the main geometric reasons (e.g. related to $CAT(0)$ complexes and positivity properties of (restrictions of) conormal bundles on toric varieties) are associated to properties of the top component. If we look for more combinatorial reasons related to the entire flag property, we are often led to specific cases (e.g. those considered in \cite{Athgam}) or general connections to balanced simplicial complexes (e.g. \cite{CCV}). \\
	
	In order to have a better grasp of how we can think about these objects objects quantitatively, we considered averaging and asymptotic properties. This does not seem to have been studied in many settings in the context of gamma vectors apart from work on symmetric edge polytopes in \cite{DJKKV}. More specifically, we focused on measuring how important the flag property is to nonnegativity of gamma vectors of simplicial spheres. To this end, we replaced the flag property by the more general link condition $\lk_\Delta(a) \cap \lk_\Delta(b) = \lk_\Delta(e)$ for edges $e \in \Delta$ of a simplicial complex $\Delta$. This is a common combinatorial application of the flag property which ends up being equivalent to preservation of piecewise linear (PL) homeomorphism types of compact PL manifolds without boundary \cite{Nev}. We will mainly consider PL simplicial spheres. \\
	
	Given a simplicial sphere, the characterization of possible $g$-vectors of simplicial spheres (Theorem \ref{gconjsphere}) and approximations of pseudopowers arising from Macaulay representations (Definition \ref{macrepower}) discussed in Section \ref{gennonnegrealiz} (e.g. applications of results in \cite{El} and \cite{Bol}) enable us to study the effect of the link condition on the structure of simplicial spheres with nonnegative gamma vectors in Section \ref{linkcondimp}. The results involved also apply to generalizations of gamma positivity coming from connections between orthogonal polynomials and lattice paths (see Section \ref{genorthlat}). This fits into lattice path-related properties of gamma vector transformations (Proposition \ref{gamtgrowth}), recent gamma positivity-related work (e.g. \cite{BV} and \cite{EHR}), and positivity and structural results of the gamma polynomial from inverted Chebyshev expansions (Theorem \ref{gamchebinv}, Corollary 1.35 on p. 25 of \cite{Pflag} and Corollary 1.10 on p. 9 of \cite{Pcheb}). The main results are listed below. \\

	 The ``averaged'' property mentioned earlier is due to an application of a $g$-vector version of a local-global formula involved in the proof of the upper bound conjecture and the generalized lower bound theorem for balanced simplicial complexes. This results in global inequalities of $g$-vector components storing ``average local data'' from links over contracted edges. Since the global/averaged version of the $g(\widetilde{\Delta}_e) \ge 0$ conditions are often trivial, we will list a global/average version of the $M$-vector condition below. \\

	 \begin{prop} (Section \ref{linkcondimp}) \\
	 	Let $\Delta$ be a $(d - 1)$-dimensional simplicial sphere satisfying the link condition. The global/averaged $M$-vector condition on the $g(\widetilde{\Delta})$ is a lower bound on $\frac{g_{k + 2}(\Delta)}{g_{k + 1}(\Delta)}$ in terms of $\frac{g_{k + 1}(\Delta)}{g_k(\Delta)}$, the number of edges $f_1(\Delta)^{ \frac{1}{k} }$, and $\frac{g_k(\Delta)^{ \frac{k + 1}{k} }}{g_{k + 1}(\Delta)}$ (Proposition \ref{ratlowbd} and Example \ref{simpratlowbd}). The last parameter measures how far the $M$-vector condition $g_{k + 1}(\Delta) \le g_k(\Delta)^{ \langle k \rangle }$ is form equality (Part 2c) of Theorem \ref{gconjsphere} and Corollary \ref{macpowasymp}). Decreasing these parameters increases the resulting lower bound of $\frac{g_{k + 2}(\Delta)}{g_{k + 1}(\Delta)}$. \\
	 \end{prop}
	 
	 Returning to the ``averaging'' perspective, we obtain the following results on asymptotic behavior for high-dimensional simplicial spheres $\Delta$ satisfying the link condition such that $\gamma(\Delta) \ge 0$. While the results below are for \emph{individual} link condition inequaliites, we can also obtain analogous results for \emph{families} of link condition inequalities. \\

	 \begin{thm} (Section \ref{linkcondimp}) \label{linkvsgamsphere} \\
	 	Let $\Delta$ be a $(d - 1)$-dimensional simplicial sphere satisfying the link condition such that $\gamma(\Delta) \ge 0$. Given an edge $e \in \Delta$, let $\widetilde{\Delta}_e$ be the contraction of $e$. Fix $1 \le k \le \frac{d}{2} - 2$ in Part 2 and Part 3. \\

	 	\begin{enumerate}
	 		
	 		\item \textbf{(Triviality conditions and linear vs. superlinear $g_1(\Delta)$)} \\
	 			\begin{enumerate}
	 				\item The global/averaged $g(\widetilde{\Delta}_e) \ge 0$ condition for edges $e \in \Delta$ condition follows from the $M$-vector condition for arbitrary simplicial spheres $\Delta$ unless $g_1(\Delta) < d + 10$. \\
	 				
	 				\item The global/averaged $M$-vector condition on the $g(\widetilde{\Delta}_e)$ is trivial if $g_1(\Delta)$ is linear in $d$. \\
	 				
	 				\item If $g_1(\Delta)$ is superlinear with respect to $d$ and the rate of increase of the $g$-vector components with respect to single index changes is small compared to the number of edges, the global/averaged $M$-vector condition on the $g(\widetilde{\Delta}_e)$ is trivial. This indicates that the $g$-vector (and thus the gamma vector) is not affected by the link condition in large dimensions in such cases. The barycentric subdivision example suggests a closer connection to $f$-vectors in such settings. Note that the $f$-vector realization condition on a nonnegative gamma vector is asymptotically slightly stronger than that of realization by a simplicial sphere. \\
	 				
	 				\item Additional instances of triviality occur in global/averaged $M$-vector conditions on the $g(\widetilde{\Delta}_e)$ in ``generic'' cases among asymptotic equivalences between rate of increase of the $g$-vector components with respect to index changes.
	 			\end{enumerate}
	 		
	 		\pagebreak 
	 		
	 		\item \textbf{(Interpretation in nontrivial instances)} \\
	 			\begin{enumerate}
	 				\item ``Special/maximal'' cases among asymptotic equivalences between rate of increase of the $g$-vector components with respect to index changes can give rise to nontrivial global/averaged $M$-vector conditions on the $g(\widetilde{\Delta}_e)$. \\
	 				
	 				\item When global/averaged $M$-vector conditions on the $g(\widetilde{\Delta}_e)$ are nontrivial, the resulting limiting inequalities take the form of lower bounds between 0 and 1 for the quotient of $\frac{g_{k + 2}(\Delta)}{g_{k + 1}(\Delta)}$ by $\frac{g_{k + 1}(\Delta)}{g_k(\Delta)}$ (Corollary \ref{nontrivlinkgam} and Corollary \ref{constlinknontriv}). Since we often expect $\frac{g_{\ell + 1}(\Delta)}{g_\ell(\Delta)}$ to decrease as $\ell$ increases (Remark \ref{indexratchange}), this lower bound can be thought of as a ``control'' on this decrease. \\
	 				
	 				\item While there is initially a wide range of possible outputs, the (non)triviality depends on the comparison of ``average over edges $e \in \Delta$'' of differences $g_k(\Delta) - g_{k - 1}(\lk_\Delta(e))$ with a modified averaging over edges of ratios $\frac{g_k(\Delta)^{ \frac{k + 1}{k} }}{g_{k + 1}(\Delta)}$ measuring how far the $M$-vector inequality $g_{k + 1}(\Delta) \le g_k(\Delta)^{ \langle k \rangle}$ (Part 2c) of Theorem \ref{gconjsphere}) is from equality. This can also be phrased in terms of the size of $f_1(\Delta)^k$ compared to $g_{k + 1}(\Delta)$. Decreases in the former and increases in the latter yield larger lower bounds of this quotient of $\frac{g_{k + 2}(\Delta)}{g_{k + 1}(\Delta)}$ by $\frac{g_{k + 1}(\Delta)}{g_k(\Delta)}$. \\
	 			\end{enumerate}
	 		
	 		\item \textbf{(Gamma vector implications) \\}
	 		
	 		For gamma vectors $\gamma(\Delta)$, lower bounds on $\frac{g_{k + 2}(\Delta)}{g_{k + 1}(\Delta)}$ can be phrased as lower bounds on $\gamma_{k + 2}(\Delta)$ since this term only appears in $g_{k + 2}(\Delta)$ (see Proposition \ref{gamtohgvect}). However, the asymptotic effect on high-dimensional spheres satisfying the link condition with $\gamma(\Delta) \ge 0$ is the largest when the gamma vector components involved are asymptotically dominant in the $g$-vector components of the same index. For example, this occurs if $\gamma_{\ell + 1}(\Delta) \ge (d + 1)g_\ell(\Delta)$ (e.g. when $g_{\ell + 1}(\Delta) \ge 2(d + 1) g_\ell(\Delta)$). While the results discussed here involve \emph{individual} inequalities associated to constant indices $k$, the underlying tools (Proposition \ref{linkinitasymp}, Corollary \ref{asympgratinit}, Corollary \ref{coefflinkasymp}) apply to the behavior of \emph{collections} of inequalities associated to varying indices. \\ 
	 	\end{enumerate}
	 \end{thm}
	
	 The results involved apply to generalized positivity properties arising from connections between orthogonal polynomials and lattice paths. \\

	 \begin{prop} \textbf{(New connections and generalizations from orthogonal polynomials) \\} (Section \ref{genorthlat}) \\
	 	\begin{enumerate}
	 		
	 		\item Viewing the inverted Chebyshev expansion from this perspective, we have an interpretation of the gamma vector components in terms of monomer/dimer covers with colored monomers (Proposition \ref{gamorthpov}). In Remark \ref{dimtchebsubdiv}, we put this in the context of connections between dimer covers and repeated (stellar) edge subdivisions. This follows from compatibility of Tchebyshev subdivisions from \cite{Het} and \cite{HetB} with modified $f$-polynomials, which were the main connection between the inverted Chebyshev expansion and positivity/geometric properties in the results mentioned above. \\

	 		\item The combinatorial properties of lattice paths give orthogonal polynomial generalizations of the implication of unimodality by gamma positivity (Proposition \ref{formalgpos}) and realizability properties of nonnegative gamma vectors analogous to those in Section \ref{gamposout} for simplicial spheres (Proposition \ref{geninvgamsphere}) and Cohen--Macaulay simplicial complexes (Remark \ref{cmorthgen}). \\

	 		\item The results of Theorem \ref{linkvsgamsphere} apply to the generalized positivity properties mentioned above. This is outlined in Remark \ref{linkapproxorth}.
	 	\end{enumerate}
	 \end{prop}

	\section{Background} \label{backsec}

	We will start by defining the main objects of interest. \\
	
	\begin{defn} (Proposition 2.1.1 on p. 272 of \cite{Gal}) \\ \label{gamvect}
		\vspace{-3mm} 
		\begin{enumerate}
			\item Given a polynomial of degree $d$, the function $\gamma(t)$ is defined by \[ h(t) = (1 + t)^d \gamma \left( \frac{t}{(1 + t)^2} \right). \] The \textbf{gamma vector} is defined to be the coefficients of $\gamma$. \\
			 
			\item When $h$ is a reciprocal/palindromic polynomial, $\gamma$ is a polynomial and the coefficients $\gamma_i$ of $t^i$ in $\gamma(t)$ are the unique coefficients such that \[ h(t) = \sum_{i = 0}^{ \lfloor \frac{d}{2} \rfloor } \gamma_i x^i (1 + x)^{d - 2i}. \] We will mainly focus on the case where $d$ is even (Remark \ref{oddcheb}). \\
		\end{enumerate}
	\end{defn}
	
	The connection to orthogonal polynomials comes from the fact that the gamma polynomial can be expressed as a sort of inverted Chebyshev polynomial expansion. We will give further context and explore implications on nonnegative expansions by other orthogonal polynomials in Section \ref{genorthlat}. This will involve results generalizing those of gamma vectors on positivity properties and realizations by simplicial spheres and Cohen-Macaulay simplicial complexes. A key input will be the connection between orthogonal polynomials and lattice paths (e.g. see \cite{V1} and \cite{V2}). \\

	\begin{thm} (Theorem 1.7 on p. 8 of \cite{Pcheb}) \\ \label{gamchebinv} 
		Let $h = h_0 + h_1 t + \ldots + h_{d - 1} t^{d - 1} + h_d t^d$ be a reciprocal polynomial of even degree $d$ satisfying the relations $h_k = h_{d - k}$ for $0 \le k \le \frac{d}{2}$. Then, the gamma polynomial associated to $h(t)$ is equal to \[ \gamma(u) = u^{ \frac{d}{2} } g \left( \frac{1}{u} - 2 \right), \] where \[ g(u) \coloneq h_{\frac{d}{2}} + 2 \sum_{j = 1}^{ \frac{d}{2} } h_{ \frac{d}{2} - j } T_j \left( \frac{u}{2} \right) \] and $T_j(x)$ is the $j^{\text{th}}$ Chebyshev polynomial of the first kind. \\
		
		Equivalently, we have \[ u^{ \frac{d}{2} } \gamma \left( \frac{1}{u} \right) = g(u - 2) \] and \[ (2w)^{ \frac{d}{2} } \gamma \left( \frac{1}{2w} \right) = g(2w - 2). \]
	\end{thm}
	
	\begin{rem} \label{oddcheb}
		If $d$ is odd, we can repeat a version of this expansion for the polynomial $r(t)$ such that $h(t) = (1 + t) r(t)$ (Proposition 3.4 on p. 11 of \cite{AGSMMS}). This implies that $\gamma(u) = u^{\frac{d - 1}{2}} p \left( \frac{1}{u} - 2 \right)$, where $p$ is the counterpart of the polynomial $g$ in Theorem \ref{gamchebinv} with $r$ replacing $h$. This is consistent with the preservation of the gamma vector after multiplication by $1 + t$ (p. 17 of \cite{BV}). Geometrically, we can think of taking suspensions (p. 516 of \cite{NP}). \\
	\end{rem}

	Nonnegativity of the gamma vector of a reciprocal polynomial is a positivity property lying between unimodality and real-rootedness that appears in many different combinatorial and geometric contexts \cite{Athgam}. This ranges from permutation statistics to signatures of toric varieties \cite{LR} and the main connection to geometric contexts focuses on the top component. These examples are often special cases of triangulations of certain simplicial spheres. The Dehn--Sommerville relations $h_i = h_{d - i}$ (Theorem 5.4.2 on p. 238 and Theorem 5.2.16 on p. 229 of \cite{BH}) then lead back to the focus on reciprocal polynomials. \\

	Before specifying the connection to simplicial spheres, we will first define what properties of the triangulations will be considered. \\

	\begin{defn} (p. 208 and p. 213 of \cite{BH}, p. 80 and Theorem 1.3 on p. 77 of \cite{St}) \\ \label{fhvect}
		Let $\Delta$ be $(d - 1)$-dimensional simplicial complex.
		\begin{enumerate}
			\item The \textbf{$f$-vector} $f(\Delta) \coloneq (f_0, f_1, \ldots f_{d - 1})$ records the number of faces of $\Delta$ of each dimension. This is the convention on p. 55 of \cite{St} and p. 213 of \cite{BH} involved in $f$-vector realizations. However, note that $f_{-1}(\Delta) = 1$ is included in some realization-related contexts (e.g. $h(\Delta) = f(\Gamma)$ for some balanced $\Gamma$ if $\Delta$ is flag and Cohen--Macaulay by \cite{CCV} and p. 33 of \cite{BFS}). \\
			
			\item The \textbf{$h$-vector} is obtained from the $f$-vector by an invertible linear transformation given by \[ h_j = \sum_{i = 0}^j (-1)^{j - i} \binom{d - i}{j - i} f_{i - 1}. \] Alternatively, we have \[ \sum_{i = 0}^d h_i x^{d - i} = \sum_{i = 0}^d f_{i - 1}(x - 1)^{d - i}. \]
			
			The $h$-vector gives the numerator of the Hilbert series $\frac{h(t)}{(1 - t)^d}$ of the Stanley--Reisner ring $k[\Delta]$ of $\Delta$ (p. 212 -- 213 of \cite{BH}). If $\Delta$ gives the normal fan structure of a rational simple polytope, the individual components can be interpreted as even degree Betti numbers $h^{2i}(X_\Delta)$ from the intersection cohomology of the toric variety $X_\Delta$ (Theorem 3.1 and Corollary 3.2 on p. 196 -- 197 of \cite{Stg}). \\

			\item The \textbf{$g$-vector} is defined as \[ g \coloneq (h_0, h_1 - h_0, h_2 - h_1, \ldots, h_{ \left\lfloor \frac{d}{2} \right\rfloor } - h_{ \left\lfloor \frac{d}{2} \right\rfloor - 1 }). \]

			\color{black}
		\end{enumerate}
	\end{defn}

	We will often compare global properties of these vectors with those of local parts which are defined below. \\
	
	\begin{defn} (Definition 5.3.4 on p. 232 of \cite{BH}, p. 60 of \cite{St}) \\
		The \textbf{link} $\lk_\Delta(F)$ of a simplicial complex $\Delta$ over a face $F \in \Delta$ is \[ \lk_\Delta(F) \coloneq \{ G \in \Delta : G \cup F \in \Delta, G \cap F = \emptyset \}. \]
	\end{defn}

	\begin{rem} \textbf{(Piecewise linear (PL) assumption)} \label{plassump} \\
		By ``simplicial sphere'', we will mean a piecewise linear (PL) simplicial sphere where the links are spheres (see p. 4 of \cite{KN} and p. 26 of \cite{Hud}). \\
	\end{rem}
	
	\color{black}
	
	In our context, we will take the input polynomial for the gamma vector (Definition \ref{gamvect}) to be the $h$-polynomial of a simplicial sphere. The source of the connection to reciprocal polynomials is given by the Dehn--Sommerville relations in the first part of the result below. It gives the criteria for an $h$-vector to be realized by a simplicial sphere. \\

	\begin{thm} (Theorem 1.1 on p. 76 and 56 of \cite{St}, p. 255 of \cite{BH}, p. 5 of \cite{Ad}, Conjecture on p. 561 of \cite{McM}, p. 1184 of \cite{Gr}) \\ \label{gconjsphere}
		Let $h = (h_0, \ldots, h_d) \in \mathbb{Z}^{d + 1}$. The following are equivalent:
		\begin{enumerate}
			\item There is a $(d - 1)$-dimensional simplicial sphere $\Delta$ such that $h(\Delta) = h$. 
			
			\item The vector $h$ satisfies the following properties:
				\begin{enumerate}
					\item $h_0 = 1$
					
					\item $h_i = h_{d - i}$ (Dehn--Sommerville relations) (Theorem 5.4.2 on p. 238 and Theorem 5.2.16 on p. 229 of \cite{BH}) 
					
					\item The $g$-vector $g \coloneq (h_0, h_1 - h_0, h_2 - h_1, \ldots, h_{ \left\lfloor \frac{d}{2} \right\rfloor } - h_{ \left\lfloor \frac{d}{2} \right\rfloor - 1 })$ satisfies the inequalities $0 \le g_{i + 1} \le g_i^{\langle i \rangle}$ for all $0 \le i \le  \left\lfloor \frac{d}{2} \right\rfloor - 1 $ (i.e. $g$ is the $h$-vector of some Cohen-Macaulay simplicial complex) (Theorem 2.2 on p. 56 of \cite{St} and Theorem \ref{hvectcmshell}). \\
				\end{enumerate}
		\end{enumerate}
	\end{thm}

	Recall that nonnegativity of the gamma vector of a reciprocal polynomial implies unimodality (p. 78 of \cite{Pet}). In the context of nonnegative gamma vectors realized by simplicial spheres, our main focus will be on Part 2c). The terms involved in its statement are defined below. They have close connections to other ``realization properties'' of $f$-vectors/$h$-vectors and linear generation of Cohen--Macaulay algebras (e.g. from toric varieties with degrees halved, see Corollary 2.5 on p. 57 of \cite{St} and Theorem 3.1 on p. 196 of \cite{Stg}). \\

	\begin{defn} (p. 161 and Lemma 4.2.6 on p. 159 -- 160 of \cite{BH}, p. 262 -- 263 of \cite{Zie}, p. 55 of \cite{St}, p. 5 of \cite{BB}, Chapter 5 and Chapter 6 of \cite{Bol}) \\ \label{macrepower}
		Let $a$ and $k$ be positive integers.
		\begin{enumerate}
			\item The \textbf{$k^{\text{th}}$ Macaulay representation of $a$} is the \textbf{unique} expansion of the form \[ a = \binom{n_k}{k} + \binom{n_{k - 1}}{k - 1} + \ldots + \binom{n_j}{j}, \] where $n_k > n_{k - 1} > \cdots > n_j \ge j \ge 1$. In this expansion, $n_k$ is the largest integer $q$ such that $\binom{q}{k} \le a$. The top indices of the binomial coefficients can be interpreted as \[ \{\underbrace{1, \ldots, 1}_\text{$j - 1$ times}, n_j + 1 \ldots, n_k + 1 \} \] being the $(a + 1)^{\text{st}}$ $k$-tuple of positive integers with respect to reverse lexicographical order. \\

			\item Given the $k^{\text{th}}$ Macaulay representation of $a$ from Part 1, its \textbf{$k^{\text{th}}$ pseudopower} is defined as \[ a^{\langle k \rangle} \coloneq \binom{ n_k + 1 }{k + 1} + \binom{n_{k - 1} + 1}{k} + \ldots + \binom{n_j + 1}{j + 1} \]
			
			using the same $n_k > n_{k - 1} > \cdots > n_j \ge j \ge 1$. \\

			By the reasoning in Part 1, $a^{\langle k \rangle}$ counts the number of $(k + 1)$-tuples of positive integers that are strictly smaller than the $(k + 1)$-tuple \[ \{ 1, \underbrace{2, \ldots, 2}_\text{$j - 1$ times}, n_j + 2, \ldots, n_k + 2 \} \] with respect to reverse lexicographical order. \\
			
			Some examples where the expansion $n^{\langle k \rangle}$ occurs include realizations of $f$-vectors/$h$-vectors by simplicial complexes and linear generation of Cohen--Macaulay graded algebras in Corollary 2.5 on p. 57 of \cite{St} (e.g. from toric varieties with degrees halved). \\
		\end{enumerate}
	\end{defn}

	Apart from realizations of simplicial spheres (Theorem \ref{gconjsphere}), these expansions are involved in $f$-vectors and $h$-vectors attained by general classes of simplicial complexes. In the first statement, we note that there may a difference in indexing conventions (e.g. including $-1$ index corresponding to index $0$ term of an $h$-vector or a gamma vector). \\

	\begin{thm} (Theorem 2.1 on p. 55 of \cite{St}, p. 209 of \cite{BH}) \\ \label{kkfvect}
		$(f_0, f_1, \ldots, f_{d - 1}) \in \mathbb{Z}^d$ is the $f$-vector of some $(d - 1)$-dimensional simplicial complex if and only if \[ 0 < f_{i + 1} \le f_i^{\langle i + 1 \rangle} \] for all $0 \le i \le d - 2$. \\
	\end{thm}

	\begin{thm} (Theorem 5.1.15 on p. 219 -- 220 of \cite{BH}, $M$-vector condition from Theorem 3.3 on p. 59 and Corollary 2.4 of \cite{St}) \\ \label{hvectcmshell}
		Let $s = (h_0, \ldots, h_d)$ be a sequence of integers. The following conditions are equivalent:
			\begin{enumerate}
				\item $h_0 = 1$ and $0 \le h_{i + 1} \le h_i^{\langle i \rangle}$ for all $ 1 \le i \le d - 1$. 
				
				\item $s$ is the $h$-vector of a shellable simplicial complex.
				
				\item $s$ is the $h$-vector of a Cohen--Macaulay simplicial complex.
			\end{enumerate}
	\end{thm}

	\section{Realizability of nonnegative gamma vectors} \label{gamposout}

	\subsection{General Cohen-Macaulay simplicial complexes and simplicial spheres} \label{gennonnegrealiz}

	In this subsection, we will describe the construction of nonnegative gamma vectors realized by simplicial spheres and Cohen--Macaulay simplicial complexes in general and show that sufficiently large ones can be extended component by component (Corollary \ref{sphereposgam} and Remark \ref{matcoeffsize}). The key input comes from estimates of pseudopowers from Macaulay representations of $g$-vectors and $h$-vectors (Corollary \ref{macpowasymp}) and the $g$-conjecture for simplicial spheres (Theorem \ref{gconjsphere}). If $\gamma_0 = 1$ and $\gamma_1, \ldots, \gamma_{\frac{d}{2}} \ge 0$ that grow sufficiently quickly relative to the dimension with $\gamma_j$ while satisfying designated order $\frac{j}{j - 1}$ upper bounds in $\gamma_1, \ldots, \gamma_{j - 1}$, they give rise to nonnegative gamma vectors $(\gamma_0, \gamma_1, \ldots, \gamma_{\frac{d}{2}})$ realized by simplicial spheres. \\
	
	This is based on the existence of suitable nonnegative upper bounds for $\gamma_{i + 1}$ that increase as $\gamma_1, \ldots, \gamma_i$ increase. It also implies the same statement for attainment of nonnegative gamma vectors by Cohen--Macaulay simplicial complexes. Given sufficiently large initial entries, this can be extended to the entire gamma vector. Alternatively, replacing $g$-vectors with $h$-vectors and adjusting the linear transformation matrix gives analogous results for Cohen--Macaulay simplicial complexes in general that give parameters we can work with for required lower bounds, which may differ due to this change.  \\
	
	We also obtain constraints on special cases of interest. An initial consequence of Macaulay representation estimates is that simplicial spheres with $\gamma(\Delta) \ge 0$ and $g_1(\Delta)$ linear in $d$ have $g_i(\Delta)$ ``of order $i$'' with respect to $d$ and analogous properties hold for $h$-vectors of Cohen--Macaulay simplicial complexes (Proposition \ref{gamposling1h1}). These include boundaries of cross polytopes and simplicial spheres PL homeomorphic to them that result from taking ``not too many net edge subdivisions''. Aside from quantitative estimates, the method above also restricts possible realization properties. \\
	
	Noting that gamma vectors of some naturally occuring simplicial spheres are $f$-vectors of auxiliary simplicial complexes (e.g. barycentric subdivisions by Theorem 1.1 on p. 1365 and Corollary 5.6 on p. 1376 of \cite{NPT}), we find that the required growth condition for a nonnegative gamma vector of a simplicial sphere to be the $f$-vector of some simplicial complex (which has been shown in barycentric subdivisions) asymptotically slightly stronger than those of realization by the $h$-vector of a simplicial sphere (Corollary \ref{posgamfcomp}). Given an index $i$, both the required upper bounds in each setting have the same order in the gamma vector components $1, \gamma_1(\Delta), \ldots, \gamma_i(\Delta)$. However, only the $f$-vector condition the difference involves ``constant term changes'' in top order parts only involving the \emph{previous} entries $1, \gamma_1(\Delta), \ldots, \gamma_{i - 1}(\Delta)$ (i.e. that do \emph{not} involve $\gamma_i(\Delta)$). Finally, we give explicit polynomial bounds along with recursive and close relations between $g$-vector and gamma vector components of simplicial spheres in Corollary \ref{spheregambds}. \\

	Our starting point is an explicit expression of the $h$-vector and $g$-vector of a simplicial complex as a nonnegative linear combination of gamma vector components with coefficients from \cite{NPT}. \\

	\begin{prop} (Observation 6.1 on p. 1377 of \cite{NPT}) \\ \label{gamtohgvect}
		Let $h$ be the first $\lfloor \frac{d}{2} \rfloor$ coefficients of a reciprocal polynomial of degree $d$.
		\begin{enumerate}
			\item We have $A \gamma = (h_0, \ldots, h_{\lfloor \frac{d}{2} \rfloor})$, where \[ A = \left[ \binom{d - 2j}{i - j} \right]_{0 \le i, j \le \frac{d}{2}}. \]
			
			\item We have $B \gamma = (g_0, \ldots, g_{\lfloor \frac{d}{2} \rfloor})$, where \[ B = \left[ \binom{d - 2j}{i - j} - \binom{d - 2j}{i - j - 1} \right]_{0 \le i, j \le \frac{d}{2}}. \]

		\end{enumerate}
		
		The matrices $A$ and $B$ are totally nonnegative by a Lindstr\"om--Gessel--Viennot argument (\cite{GV}, \cite{Lind}). If $\gamma(\Delta) \ge 0$, this immediately implies that $h_i(\Delta) \ge \binom{d}{i}$ for all $0 \le i \le \frac{d}{2}$. \\
	\end{prop}
	
	\begin{rem} \textbf{(More general perspectives) \\} \label{orthgen}
		The total nonnegativity and lower bound properties both fit into the more general context of positivity and inverses of expansions by orthogonal polynomials which we will further explore in Section \ref{genorthlat}. \\ 
	\end{rem}

	We first give some restrictions on the growth rates and sizes of the coefficients involved. \\
	
	\begin{lem} \label{gamtgrowth}
		Suppose that $0 \le r, s \le \frac{d}{2}$ with $r - s \ge 1$. 
		
		\begin{enumerate}
			\item In Proposition \ref{gamtohgvect}, we have \[ 1 + \frac{2}{d} \le \frac{a_{r, s}}{a_{r - 1, s}} \le d \] and \[ \frac{1}{3} < \frac{b_{r, s}}{b_{r - 1, s}} < d + 1. \]
			
			Note that $\frac{1}{d}$ is a lower bound of the latter if $d \ge 3$. \\
			
			\item Part 1 implies the bounds \[ \left( 1 + \frac{2}{d} \right)^{r - s} \le a_{r, s} \le d^{r - s} \] and \[ \frac{1}{3^{r - s}} = 3^{-(r - s)} < b_{r, s} < (d + 1)^{r - s}. \]
		\end{enumerate}
		
	\end{lem}
	
	\begin{proof}
		\begin{enumerate}
			\item For the first pair of bounds, the term $\frac{a_{r, s}}{a_{r - 1, s}}$ can be rewritten as \[ \frac{\binom{p}{q}}{\binom{p}{q - 1}} = \frac{p - q + 1}{q} = \frac{p + 1}{q} - 1, \] where $p = d - 2s$ and $q = r - s$. In our setting, we have $p \ge 2q$ since $0 \le r \le \frac{d}{2}$. Writing \[ \frac{d - 2s + 1}{r - s} - 1 = \frac{d - 2r + 1}{r - s} + 2 - 1 = \frac{d - 2r + 1}{r - s} + 1, \] we can find an upper bound by maximizing the numerator (minimizing $r$) and minimizing the denominator (setting $r - s = 1$). The resulting upper bound is $d$. To find a lower bound, we minimize the numerator (maximizing $r$) and maximize the denominator. Then, we have a lower bound of $1 + \frac{2}{d}$. \\
			
			To show that \[ \frac{1}{3} < \frac{b_{r, s}}{b_{r - 1, s}} < d + 1, \] we consider ratios of the form \[ \frac{\binom{p}{q} - \binom{p}{q - 1}}{\binom{p}{q - 1} - \binom{p}{q - 2}} \] with $p = d - 2s$ and $q = r - s$ for some $1 \le r, s \le \frac{d}{2}$ with $r - s \ge 1$. As mentioned above, $p \ge 2q$ since $0 \le r \le \frac{d}{2}$. \\
			
			Since \[ \frac{\binom{p}{q}}{\binom{p}{q - 1}} = \frac{p - q + 1}{q} \] and \[ \frac{\binom{p}{q - 1}}{\binom{p}{q - 2}} = \frac{p - q + 2}{q - 1}, \]
			
			we have
			
			\begin{align*}
				\frac{\binom{p}{q} - \binom{p}{q - 1}}{\binom{p}{q - 1} - \binom{p}{q - 2}} &= \frac{ \binom{p}{q} }{ \binom{p}{q - 1} } \frac{ \frac{p - 2q + 1}{p - q + 1} }{ \frac{p - 2q + 3}{p - q + 2} } \\
				&= \frac{p - q + 1}{q} \cdot \frac{p - 2q + 1}{p - q + 1} \cdot \frac{p - q + 2}{p - 2q + 3} \\
				&= \frac{p - 2q + 1}{p - 2q + 3} \cdot \frac{p - q + 2}{q} \\
				&= \frac{(p - 2q) + 1}{(p - 2q) + 3} \left( \frac{p + 2}{q} - 1 \right) \\
				&< \frac{p + 2}{q} - 1.
			\end{align*}
			
			Since $p = d - 2s \le d$ and $q = r - s$ with $1 \le r, s \le \frac{d}{2}$ and $r - s \ge 1$, we have that \[ \frac{\binom{p}{q} - \binom{p}{q - 1}}{\binom{p}{q - 1} - \binom{p}{q - 2}} < d + 1. \]
			
			Next, we consider the lower bound. Since $p - 2q \ge 0$, we have that \[ \frac{(p - 2q) + 1}{(p - 2q) + 3} \ge \frac{1}{3}. \]
			
			The condition $p \ge 2q$ also implies that \[ \frac{p + 2}{q} - 1 > 1. \] Thus we have \[ \frac{\binom{p}{q} - \binom{p}{q - 1}}{\binom{p}{q - 1} - \binom{p}{q - 2}} > \frac{1}{3}. \]
			
			\item This follows from applying Part 1 to $a_{r, r} = 1$ and $b_{r, r} = 1$. \\
		\end{enumerate}
	\end{proof}

	As a consequence of this property, nonnegativity of gamma vectors gives restrictions on ratios of $g$-vector components. \\
	
	\begin{cor} \label{gamgratrest}
		Suppose that we are in the setting of Proposition \ref{gamtgrowth} and $\gamma \ge 0$. Then, we have the following restrictions on ratios of consecutive terms:
		
		\begin{enumerate}
			\item We have $g_{\ell + 1} > \frac{1}{3} g_\ell$. If $g_\ell \ne 0$, this is equivalent to \[ \frac{g_\ell}{g_{\ell + 1}} < 3. \]
			
			\item If $g_\ell \ne 0$ and \[ \frac{g_{\ell + 1}}{g_\ell} \ge d + 1, \] then the additional increase is induced by $\gamma_{\ell + 1}$ and $\gamma_{\ell + 1} \ge g_{\ell + 1} - (d + 1) g_\ell$. \\
		\end{enumerate}
	\end{cor}
	
	\begin{proof}
		\begin{enumerate}
			\item Recall that \[ g_\ell = b_{\ell, 0} + b_{\ell, 1} \gamma_1 + \ldots + b_{\ell, \ell - 1} \gamma_{\ell - 1} + \gamma_\ell \] and \[ g_{\ell + 1} = b_{\ell + 1, 0} + b_{\ell + 1, 1} \gamma_1 + \ldots + b_{\ell + 1, \ell - 1} \gamma_{\ell - 1} + b_{\ell + 1, \ell} \gamma_\ell + \gamma_{\ell + 1}. \]

			Since $\gamma \ge 0$, we have $\gamma_{\ell + 1} \ge 0$ and \[ g_{\ell + 1} \ge b_{\ell + 1, 0} + b_{\ell + 1, 1} \gamma_1 + \ldots + b_{\ell + 1, \ell - 1} \gamma_{\ell - 1} + b_{\ell + 1, \ell} \gamma_\ell.  \]
			
			\color{black} 
			By Lemma \ref{gamtgrowth}, we have \[ \frac{b_{\ell + 1, s}}{b_{\ell, s}} > \frac{1}{3} \Longrightarrow b_{\ell + 1, s} > \frac{1}{3} b_{\ell, s} \] for all $0 \le \ell \le \lfloor \frac{d}{2} \rfloor - 1$ and $0 \le s \le \ell$.  \\
			
			Substituting this into the inequality above, we have 
			\begin{align*}
				g_{\ell + 1} &\ge b_{\ell + 1, 0} + b_{\ell + 1, 1} \gamma_1 + \ldots + b_{\ell + 1, \ell - 1} \gamma_{\ell - 1} + b_{\ell + 1, \ell} \gamma_\ell \\
				&> \frac{1}{3} b_{\ell, 0} + \frac{1}{3} b_{\ell, 1} \gamma_1 + \ldots + \frac{1}{3} b_{\ell, \ell - 1} \gamma_{\ell - 1} + \frac{1}{3} b_{\ell, \ell} \gamma_\ell \\
				&= \frac{1}{3} (b_{\ell, 0} + b_{\ell, 1} \gamma_1 + \ldots + b_{\ell, \ell - 1} \gamma_{\ell - 1} + b_{\ell, \ell} \gamma_\ell) \\
				&= \frac{1}{3} (b_{\ell, 0} + b_{\ell, 1} \gamma_1 + \ldots + b_{\ell, \ell - 1} \gamma_{\ell - 1} + \gamma_\ell) \\
				&= \frac{1}{3} g_\ell.
			\end{align*}
			
			Since $\gamma \ge 0$, we have $g \ge 0$. If $g_\ell \ne 0$, then we have $g_\ell > 0 \Longrightarrow g_{\ell + 1} > 0$ and the inequality above is equivalent to \[ \frac{g_\ell}{g_{\ell + 1}} < 3. \]
			
			\item By Lemma \ref{gamtgrowth}, we have \[ \frac{b_{\ell + 1 , s}}{b_{\ell, s}} < d + 1 \] for all $0 \le \ell \le \lfloor \frac{d}{2} \rfloor - 1$ by Lemma \ref{gamtgrowth}. \\
			
			Then, the inequality above implies that \[ b_{\ell + 1, 0} + b_{\ell + 1, 1} \gamma_1 + \ldots + b_{\ell + 1, \ell} \gamma_\ell < (d + 1) g_\ell \le g_{\ell + 1}. \]
			
			The second inequality is implied by the assumption $\gamma \ge 0$. This implies that the additional increase comes from $\gamma_{\ell + 1}$ and $\gamma_{\ell + 1} \ge g_{\ell + 1} - (d + 1) g_\ell$. \\
			
		\end{enumerate}
	\end{proof}

	\color{black}

	We will now relate these properties back to those of simplicial spheres and Cohen--Macaulay simplicial complexes in general. To this end, we will first consider properties of the Macaulay representation involved in characterizations of $g$-vectors and $h$-vectors of simplicial spheres and Cohen--Macaulay simplicial complexes respectively (Theorem \ref{gconjsphere} and Theorem \ref{hvectcmshell}). The first property follows from the description of the construction of the Macaulay representation in Definition \ref{macrepower} and considering the binomial coefficient as a function of the top index (see Lemma 5.6 on p. 2116 of \cite{El} and p. 38 of \cite{Bol}). \\
	
	\begin{lem} \label{macrepbds}
		Given positive integers $a$ and $k$, let \[ a = \binom{n_k}{k} + \binom{n_{k - 1}}{k - 1} + \ldots + \binom{n_j}{j}, \] where $n_k > n_{k - 1} > \cdots > n_j \ge j \ge 1$ be the $k^{\text{th}}$ Macaulay representation of $a$ (Definition \ref{macrepower}). \\
		
		Then, we have \[ n_k \le x < n_k + 1, \] where $x$ is the unique real number such that $a = \binom{x}{k}$. \\
	\end{lem}
	
	\begin{lem} (Theorem 5.9 on p. 2117 -- 2118 of \cite{El}) \label{kmacpwrbds} \\
		Recall from Definition \ref{macrepower} that \[ a^{\langle k \rangle} \coloneq \binom{ n_k + 1 }{k + 1} + \binom{n_{k - 1} + 1}{k} + \ldots + \binom{n_j + 1}{j + 1}. \] We have that \[ \binom{ n_k + 1 }{k + 1} \le a^{\langle k \rangle} \le \binom{x + 1}{k + 1}, \] where $x$ is the same number from Lemma \ref{macrepbds}. \\
	\end{lem}

	We now apply these estimates to give asymptotic properties and explicit upper bounds on the pseudopowers. \\
	
	\begin{defn} (special case of Part 2 of Definition \ref{asympdef}) \\
		Given functions $F$ and $G$ on positive real numbers, we write $F \sim G$ if \[ \lim_{x \to +\infty} \frac{F(x)}{G(x)} = 1. \] \\ 
	\end{defn}
	
	\begin{cor} \label{macpowasymp}
		Suppose that we are in the setting of Lemma \ref{macrepbds} and Lemma \ref{macrepower}. Given a positive integer $m$, let \[ C_m = \frac{(m!)^{\frac{1}{m}}}{m + 1}. \] 
		\begin{enumerate}
			\item If $k$ is fixed, then $a^{\langle k \rangle} \sim C_k a^{ \frac{k + 1}{k} } = C_k a^{ 1 + \frac{1}{k} }$ as $a \to +\infty$. We have $C_m < 1$ and $C_m \to \frac{1}{e}$ as $m \to +\infty$. 
			
			\item For any positive integers $a$ and $k$, we have $a^{\langle k \rangle} \le a^{\frac{k + 1}{k}} = a^{1 + \frac{1}{k}}$. \\
		\end{enumerate}
	\end{cor}
	
	\begin{proof}
		\begin{enumerate}
			\item By Lemma \ref{macrepbds} and Lemma \ref{macpowasymp}, we have $a \sim \frac{x^k}{k!}$ and $a^{\langle k \rangle} \sim \frac{x^{k + 1}}{(k + 1)!}$. This implies that \[ \frac{a^{\langle k \rangle}}{a} \sim \frac{x}{k + 1}. \]
			
			Since $x \sim (k! a)^{\frac{1}{k}}$, we have that \[ \frac{a^{\langle k \rangle}}{a} \sim \frac{(k! a)^{\frac{1}{k}}}{k + 1} \Longrightarrow a^{\langle k \rangle} \sim \frac{(k!)^{\frac{1}{k}}}{k + 1} a^{\frac{k + 1}{k}} = C_k a^{\frac{k + 1}{k}}. \]
			
			Since $(m!)^{\frac{1}{m}} < m < m + 1$, we have that $C_m < 1$. The claim that $C_m \to \frac{1}{e}$ as $m \to +\infty$ follows from Stirling's approximation $N! \sim \sqrt{2\pi N} \left( \frac{N}{e} \right)^N$. We may also use sharper estimates for arbitrary $N$ involving upper and lower bounds of $N!$ from constant multiples of $\sqrt{N} \cdot \left( \frac{N}{e} \right)^N = e^{-N} \cdot N^{N + \frac{1}{2}}$ (see p. 287 of \cite{ElM1}). \\
			
			\item Using the notation of Lemma \ref{macrepbds}, Lemma \ref{macpowasymp} implies that \[ \frac{a^{\langle k \rangle}}{a} \le \frac{\binom{x + 1}{k + 1}}{\binom{x}{k}} = \frac{x + 1}{k + 1}. \]
			
			Note that a term by term comparison of products in $\binom{x}{k} = \frac{x(x - 1) \cdots (x - k + 1)}{k!}$ implies that \[ \frac{x^k}{k^k} \le \binom{x}{k} = a \Longrightarrow x \le k a^{ \frac{1}{k} } \] since $x > 0$, $a > 0$, and $k > 0$. \\
			
			Substituting this into the bound \[ \frac{a^{\langle k \rangle}}{a} \le \frac{x + 1}{k + 1}, \]
			
			we have that 
			
			\begin{align*}
				\frac{a^{\langle k \rangle}}{a} &\le \frac{k a^{ \frac{1}{k}} + 1}{k + 1} \\
				&\le \frac{k a^{ \frac{1}{k}} + a^{\frac{1}{k}}}{k + 1} \\&= a^{\frac{1}{k}} \\
				\Longrightarrow a^{\langle k \rangle} &\le a^{\frac{k + 1}{k}}
			\end{align*}
			
			since $a \ge 1$ and $k > 0$.

		\end{enumerate}
	\end{proof}

	As an initial consequence, we can give a simple restriction on the growth rate of the components $g_\ell(\Delta)$ of the $g$-vector of a simplicial sphere $\Delta$. \\
	
	\begin{cor} \label{sphereginitbd}
		Given a $(d - 1)$-dimensional simplicial sphere $\Delta$, we have that $g_\ell(\Delta) \le g_1(\Delta)^\ell$ for all $0 \le \ell \le \frac{d}{2}$. \\
	\end{cor}
	
	\begin{proof}
		We have equality in the $\ell = 0$ and $\ell = 1$ cases. For general $\ell$, we will use induction. \\
		
		Suppose that $g_{\ell - 1}(\Delta) \le g_1(\Delta)^{\ell - 1}$. Since $\Delta$ is a simplicial sphere, Theorem \ref{gconjsphere} implies that $g_\ell(\Delta) \le g_{\ell - 1}(\Delta)^{\langle \ell - 1 \rangle}$. By Corollary \ref{macpowasymp}, we have $g_{\ell - 1}(\Delta)^{\langle \ell - 1 \rangle} \le g_{\ell - 1}(\Delta)^{ \frac{\ell}{\ell - 1} }$. Combining this with our induction assumption $g_{\ell - 1}(\Delta) \le g_1(\Delta)^{\ell - 1}$, we have $g_\ell(\Delta) \le (g_1(\Delta)^{\ell - 1})^{ \frac{\ell}{\ell - 1} } = g_1(\Delta)^\ell$. 
	\end{proof}

	\begin{rem} \textbf{(Comparison between $g_1(\Delta)$ and $d$) \\} \label{g1dcomp}
		
		While there is a lower bound $g_1(\Delta) \ge d - 1$ for $(d - 1)$-dimensional simplicial spheres such that $\gamma(\Delta) \ge 0$ (attained by the boundary of the $d$-dimensional cross polytope), the parameter $g_1(\Delta)$ can grow arbitrarily quickly with respect to $d$. In other words, $d$ can grow arbitarily slowly with respect to $g_1(\Delta)$ even when it is not a constant. We can see this by considering flag simplicial spheres, such that $\gamma(\Delta) \ge 0$ and $\gamma(\lk_\Delta(e)) \ge 0$ for each edge $e \in \Delta$. This is because flagness is preserved by edge subdivision and each edge subdivision increases the number of vertices by 1. \\
	\end{rem}
	
	We can also show that can show that simplicial spheres where $\gamma(\Delta) \ge 0$ and $g_1(\Delta)$ is linear in $d$ have $g_i(\Delta)$ is bounded above and below by a degree $i$ polynomial. In some sense, this indicates that $g_i(\Delta)$ is ``of order $i$'' with respect to $d$ in such a setting. Examples where this occurs include boundaries of cross polytopes where $g_1(\Delta) = d - 1$, which are minimal elements among flag PL simplicial spheres of a given dimension with respect to edge subdivisions \cite{LN}. For a simplicial sphere PL homeomorphic to the boundary of a cross polytope, this is a case where the ``net number of edge subdivisions'' (Corollary [10:2d] on p. 302 of \cite{Alex}, Corollary 4.1 on p. 75 of \cite{LN}) does not increase too quickly with respect to the dimension. Similar results hold for $h$-vectors of Cohen--Macaulay simplicial complexes. \\
	
	\begin{prop} \textbf{(Linear net number of edge subdivisions and order with respect to $d$) \\} \label{gamposling1h1}
		Let $\Delta$ be a $(d - 1)$-dimensional simplicial complex such that $\gamma(\Delta) \ge 0$.  	
		\begin{enumerate}
			\item Suppose that $\Delta$ is a simplicial sphere. If $g_1(\Delta)$ is linear in $d$, then each $g_i(\Delta)$ is bounded above and below by an order $i$ polynomial in $d$ with constant coefficients. \\ 
			
			\item Suppose that $\Delta$ is a Cohen--Macaulay simplicial complex. If $h_1(\Delta)$ is linear in $d$, then $h_i(\Delta)$ is bounded above and below by an order $i$ polynomial in $d$ with constant coefficients. \\ 
		\end{enumerate}
	\end{prop}

	\begin{rem}
		Although the bounds by degree $i$ polynomials suggest that $g_i(\Delta)$ and $h_i(\Delta)$ are ``of order $i$ with respect to $d$'' in this case, we need to be careful about the dependence of the coefficients in $i$. This can cause a large variance in things like ratios between terms from consecutive indices (see Part 2 of Example \ref{constexmp}). A more direct possible issue with the polynomial bound in terms of $g_1(\Delta)$ due to the index comparison will be discussed in Theorem \ref{avghbd} and Part 4 of Corollary \ref{spheregambds}. \\
	\end{rem}

	\begin{proof}
		\begin{enumerate}
			\item The lower bounds of $g_i(\Delta)$ by degree $i$ polynomials in $d$ is implied by $\gamma(\Delta) \ge 0$. This is because $g_i \ge b_{i, 0} = \binom{d}{i} - \binom{d}{i - 1}$ for all $0 \le i \le \frac{d}{2}$ by Proposition \ref{gamtohgvect}. Thus, we will focus on the upper bound. \\
			
			Suppose that $i = 2$. The $M$-vector condition on the $g$-vector from Theorem \ref{gconjsphere} implies that $g_2(\Delta) \le g_1(\Delta)^{\langle 1 \rangle} = \binom{g_1(\Delta) + 1}{2}$. Since $g_1(\Delta)$ is linear in $d$, the upper bound is a degree 2 polynomial in $d$. \\
			
			Suppose that $g_i(\Delta)$ is bounded above by a degree $i$ polynomial in $d$. Then, Theorem \ref{gconjsphere} and Corollary \ref{macpowasymp} imply that $g_{i + 1}(\Delta) \le g_i(\Delta)^{\langle i \rangle} \le g_i(\Delta)^{ \frac{i + 1}{i} }$. Since $g_i(\Delta)$ is bounded above by a degree $i$ polynomial in $d$, this implies that $g_{i + 1}(\Delta)$ is bounded above by a constant multiple of $d^{i + 1}$. \\

			\item This follows from the same reasoning as Part 1 after replacing $g_i(\Delta)$ by $h_i(\Delta)$, using $a_{i, 0}$ instead of $b_{i, 0}$ in Proposition \ref{gamtohgvect}, and applying Theorem \ref{hvectcmshell} instead of Theorem \ref{gconjsphere}. Note that $h_i(\Delta) = 1 + g_1(\Delta) + \ldots + g_i(\Delta)$. \\
		\end{enumerate}
	\end{proof}
	
	\color{black}

	Combining the estimates of Corollary \ref{macpowasymp} with Theorem \ref{gconjsphere}, we determine when nonnegative gamma vectors are realized by simplicial spheres and show that nonnegative gamma vector inputs that increase sufficiently quickly can be recursively extended. \\

	\begin{cor} \textbf{(Nonnegative gamma vectors realized by simplicial spheres) \\} \label{sphereposgam}
		Suppose that $\gamma_0 = 1$ (which is necessary for $h_0 = 1$). Fix a positive integer $d \ge 2$.
		\begin{enumerate}
			\item Given a choice of $\gamma_1, \ldots, \gamma_i \ge 0$, there is a choice of $\gamma_{i + 1} \ge 0$ compatible with the degree $i$ and $i + 1$ inequality in the $M$-vector condition on $g$ for realization by a $(d - 1)$-dimensional simplicial sphere if and only if \[ g_i^{\langle i \rangle} = (b_{i,0} + b_{i,1} \gamma_1 + \ldots + b_{i, i - 1} \gamma_{i - 1} + \gamma_i)^{\langle i \rangle} \ge b_{i + 1,0} + b_{i + 1,1} \gamma_1 + \ldots + b_{i + 1, i - 1} \gamma_{i - 1} + b_{i + 1, i} \gamma_i. \]
			
			\item If the condition in Part 1 is satisfied, the required bound on $\gamma_{i + 1}$ is
			\begin{align*}
				\gamma_{i + 1} &\le g_i^{\langle i \rangle} - (b_{i + 1,0} + b_{i + 1,1} \gamma_1 + \ldots + b_{i + 1, i - 1} \gamma_{i - 1} + b_{i + 1, i} \gamma_i) \\
				&= (b_{i,0} + b_{i,1} \gamma_1 + \ldots + b_{i, i - 1} \gamma_{i - 1} + \gamma_i)^{\langle i \rangle} - (b_{i + 1,0} + b_{i + 1,1} \gamma_1 + \ldots + b_{i + 1, i - 1} \gamma_{i - 1} + b_{i + 1, i} \gamma_i).
			\end{align*}

			\item Suppose that $\gamma_0 = 1$ and $\gamma_k$ satisfies the inequality in Part 2 for all $1 \le k \le i$, and results in integer entries of $(h_0, \ldots, h_{ \lfloor \frac{d}{2} \rfloor })$. If $\gamma_k \gg 0$ for all $1 \le k \le i$, there is a $\gamma_{i + 1} \ge 0$ compatible with the inequalities in Part 2. This indicates that nonnegative entries $(\gamma_0, \gamma_1, \ldots, \gamma_i)$ that increase sufficiently quickly (relative to $d$) with entries $\gamma_j$ while satisfying designated order $\frac{j}{j - 1}$ upper bounds with respect to $\gamma_1, \ldots, \gamma_{j - 1}$ yield nonnegative upper bounds for $\gamma_{i + 1}$ we can choose from. With sufficiently large inputs, this extends to the entire gamma vector. \\
		\end{enumerate}
	\end{cor}
	
	\begin{proof}
		
		Since Part 1 and Part 2 follow from Theorem \ref{gconjsphere}, we will focus on Part 3. The key point is that there is a suitable choice of $\gamma_{i + 1} \ge 0$ if $\gamma_1, \ldots, \gamma_i \ge 0$ are sufficiently large. Intuitively, Corollary \ref{macpowasymp} indicates that the left hand side is of order $\frac{i + 1}{i} = 1 + \frac{1}{i}$ in $1, \gamma_1, \ldots, \gamma_i$ and the right hand side is linear in these terms. If they are large, we would expect the gap between the two sides to get larger. We will combine Corollary \ref{macpowasymp} and Lemma \ref{gamtgrowth} to find an approximate threshold. \\
		
		Suppose that $\gamma_k > 0$ for some $1 \le k \le i$. This would mean that both sides are strictly positive. Let $\alpha = b_{i,0} + b_{i,1} \gamma_1 + \ldots + b_{i, i - 1} \gamma_{i - 1} + \gamma_i$ and \[ m = \frac{b_{i + 1,0} + b_{i + 1,1} \gamma_1 + \ldots + b_{i + 1, i - 1} \gamma_{i - 1} + b_{i + 1, i} \gamma_i}{b_{i,0} + b_{i,1} \gamma_1 + \ldots + b_{i, i - 1} \gamma_{i - 1} + \gamma_i}. \] 
		
		We would like to find when \[ \alpha^{\langle i \rangle} \ge m \alpha. \] Since $\frac{1}{3} < m < d + 1$ by Lemma \ref{gamtgrowth}, it suffices to show that \[ \alpha^{\langle i \rangle} \ge (d + 1) \alpha \] for $\alpha \gg 0$. This is equivalent to showing that \[ \frac{\alpha^{\langle i \rangle}}{\alpha} \ge d + 1 \] for $\alpha \gg 0$. \\
		
		By Corollary \ref{macpowasymp}, we have that \[ \alpha^{\langle i \rangle } \sim C_i \alpha^{ \frac{i + 1}{i} } \Longrightarrow \frac{\alpha^{\langle i \rangle}}{\alpha} \sim C_i \alpha^{\frac{1}{i}}, \] where \[ C_i = \frac{(i!)^{\frac{1}{i}}}{i + 1}. \] 
		
		Since $\lim_{\alpha \to +\infty} \alpha^{\frac{1}{i}} = +\infty$, we eventually have $\frac{\alpha^{\langle i \rangle}}{\alpha} \ge d + 1$ for $\alpha \gg 0$. The term $\alpha$ is a linear combination of $1, \gamma_1, \ldots, \gamma_i$ with coefficients in $b_{i, j}$ for $0 \le j \le i$ and the induced approximate lower bound on $\alpha$ is $\left( \frac{d + 1}{i + 1} \right)^i \cdot i!$. In particular, we expect the ratio $\frac{\alpha^{\langle i \rangle}}{\alpha}$ to increase as $\alpha$ increases. This means that we can increase $\gamma_1, \ldots, \gamma_i$ so that $\gamma_{i + 1}$ can be large enough for the index $i + 2$ difference in Part 2 to be nonnegative. Repeating this for successive indices yields the desired conclusion. \\
		
	\end{proof}
	
	\color{black}

	\begin{rem} \label{matcoeffsize} ~\\
		\begin{enumerate}
			\item \textbf{(Comparison with Cohen--Macaulay case)} \\
			
			Since simplicial spheres are Cohen--Macaulay, analogous results hold for nonnegative gamma vectors associated to $h$-vectors of Cohen--Macaulay simplicial complexes (equivalently those of shellable simplicial complexes by Theorem 5.1.15 on p. 219 -- 220 of \cite{BH}). We can also see this by directly applying Theorem \ref{hvectcmshell} and replacing the $b_{r, s}$ by $a_{r, s}$. While the actual threshold size may vary due to using $a_{r, s}$ instead of $b_{r, s}$ in Proposition \ref{gamtohgvect}, the common upper bound in Lemma \ref{gamtgrowth} gives an approximate common one we can use. \\
			
			\item \textbf{(Counterparts for varying dimensions $d$) \\}
			
			When we consider $d \to \infty$, the same statements above hold when each $\gamma_i$ is very large compared to $d$. The proof of Corollary \ref{sphereposgam} implies that it suffices to have $\gamma_i \gg M d^d$ (e.g. $\gamma_i \ge d!$) for some $M \gg 0$ for each $i$. In the proof of Part 2 of Corollary \ref{sphereposgam}, the lower bound on the linear combination of $1, \gamma_1, \ldots, \gamma_i$ involved (with coefficients in $b_{i, j}$ with $0 \le j \le i$) is $\left( \frac{d + 1}{i + 1} \right)^i \cdot i!$. Since we add a new $\gamma_i(\Delta)$ component on each turn, it may be possible to ``load'' this increase on the new component. For example, we may use this to keep the $\gamma_1(\Delta)$ linear in $d$. This is sort of the opposite of balancing increases across different components of the gamma vector. \\
		\end{enumerate}
	\end{rem}

	Given that there are natural examples where simplicial spheres with nonnegative gamma vectors have those given by $f$-vectors (e.g. barycentric subdivisions by Theorem 1.1 on p. 1365 and Corollary 5.6 on p. 1376 of \cite{NPT}), we compare realizability of nonnegative gamma vectors by simplicial spheres with the condition of being an $f$-vector.  \\
	
	\pagebreak 
	
	\begin{cor} \textbf{(Comparison with $f$-vector condition) \\} \label{posgamfcomp}
		
		Given a fixed dimension $d - 1$, the existence conditions for a nonnegative gamma vector to be associated to an $h$-vector realized by a Cohen--Macaulay simplicial complex (equivalently that of a shellable one) or even a simplicial sphere are asymptotically slightly weaker than those required for the gamma vector itself to be the $f$-vector of an auxiliary simplicial complex. \\
		
		The same steps indicate that gamma vector entries of index that grow sufficiently quickly with respect to $d$ with the index $i + 1$ terms respecting an order $\frac{i + 1}{i}$ bound in the terms of index $0, 1, \ldots, i$ are equal to the $f$-vector of some simplicial complex. The resulting upper bounds for these conditions still have the same maximal order terms in $1, \gamma_1(\Delta), \ldots, \gamma_i(\Delta)$. However, only the $f$-vector condition involves smaller ``constant terms'' for the top order terms that only involve the \emph{previous} gamma vector entries $1, \gamma_1(\Delta), \ldots, \gamma_{i - 1}(\Delta)$ (i.e. that do \emph{not} involve $\gamma_i(\Delta)$). \\
	\end{cor}
	
	\begin{rem}
		Comments analogous to those of Remark \ref{matcoeffsize} hold for varying dimensions. \\
	\end{rem}
	
	\begin{proof}
		
		Taking index changes into account, the $f$-vector realizability conditions of Theorem \ref{kkfvect} state that we need $\gamma_{i + 2}(\Delta) \le \gamma_{i + 1}(\Delta)^{\langle i + 1 \rangle}$. Substituting in $i - 1$ in place of $i$, we have $\gamma_{i + 1}(\Delta) \le \gamma_i(\Delta)^{\langle i \rangle}$. The steps of the proof of Corollary \ref{sphereposgam} applying the estimates of Corollary \ref{macpowasymp} imply the controlled growth statement analogous to these results. We will focus on $g$-vectors of simplicial spheres since the analogous statement for $h$-vectors of Cohen--Macaulay simplicial complexes follows from the same reasoning by Lemma \ref{gamtgrowth}. \\
		
		The statement comparing this with attainment by some simplicial sphere and involves the $\gamma_{i + 1}(\Delta)$ upper bounds 
		\begin{align*}
			g_i(\Delta)^{\langle i \rangle} - (b_{i + 1, 0} + b_{i + 1, 1} \gamma_1(\Delta) + \ldots + b_{i + 1, i} \gamma_i(\Delta)) &= (b_{i, 0} + b_{i, 1} \gamma_1(\Delta) + \ldots + b_{i, i - 1} \gamma_{i - 1}(\Delta) + \gamma_i(\Delta))^{\langle i \rangle} \\ &- (b_{i + 1, 0} + b_{i + 1, 1} \gamma_1(\Delta) + \ldots + b_{i + 1, i} \gamma_i(\Delta)) \\
			&= ( \gamma_i(\Delta) + (b_{i, 0} + b_{i, 1} \gamma_1(\Delta) + \ldots + b_{i, i - 1} \gamma_{i - 1}(\Delta)) )^{\langle i \rangle} \\
			&- (b_{i + 1, 0} + b_{i + 1, 1} \gamma_1(\Delta) + \ldots + b_{i + 1, i} \gamma_i(\Delta))
		\end{align*}
		
		from the simplicial sphere side (Part 2 of Corollary \ref{sphereposgam}) and \[ \gamma_i(\Delta)^{\langle i \rangle} \] from the $f$-vector condition (Theorem \ref{kkfvect}). To compare these bounds, we recall that \[ a^{\langle k \rangle} \sim C_k a^{ \frac{k + 1}{k} } \] with \[ C_k = \frac{(k!)^{\frac{1}{k}}}{k + 1} \] by Corollary \ref{macpowasymp} and $B^p \le (A + B)^p - A^p$ if $A, B \ge 0$ and $p \ge 1$ (see Lemma \ref{sumpowerbd}). Setting \[ A = \gamma_i(\Delta), \]  \[ B = b_{i, 0} + b_{i, 1} \gamma_1(\Delta) + \ldots + b_{i, i - 1} \gamma_{i - 1}(\Delta), \] and $p = \frac{k + 1}{k} = 1 + \frac{1}{k}$, dividing and taking limits implies that 
		
		\begin{align*}
			\gamma_i(\Delta)^{\langle i \rangle} &\le ( \gamma_i(\Delta) + (b_{i, 0} + b_{i, 1} \gamma_1(\Delta) + \ldots + b_{i, i - 1} \gamma_{i - 1}(\Delta)) )^{\langle i \rangle} \\
			&- (b_{i, 0} + b_{i, 1} \gamma_1(\Delta) + \ldots + b_{i, i - 1} \gamma_{i - 1}(\Delta))^{\langle i \rangle}
		\end{align*}
		
		when $\gamma_j \gg 0$ for $1 \le j \le i$. \\
		
		Comparing this to our simplicial sphere upper bound 
		\begin{align*}
			g_i(\Delta)^{\langle i \rangle} - (b_{i + 1, 0} + b_{i + 1, 1} \gamma_1(\Delta) + \ldots + b_{i + 1, i} \gamma_i(\Delta))
			&= ( \gamma_i(\Delta) + (b_{i, 0} + b_{i, 1} \gamma_1(\Delta) + \ldots + b_{i, i - 1} \gamma_{i - 1}(\Delta)) )^{\langle i \rangle} \\
			&- (b_{i + 1, 0} + b_{i + 1, 1} \gamma_1(\Delta) + \ldots + b_{i + 1, i} \gamma_i(\Delta)),
		\end{align*}
		
		we see that the simplicial sphere upper bound subtracts a linear combination of $1, \gamma_1(\Delta), \ldots, \gamma_i(\Delta)$ instead of something that is of order $\frac{i + 1}{i} = 1 + \frac{1}{i}$. It remains to compare the coefficients of the $\gamma_j(\Delta)$ terms. By Lemma \ref{gamtgrowth}, we have that \[ \frac{1}{3} < \frac{b_{r, s}}{b_{r - 1, s}} < d + 1. \] This indicates that most possibly ``minimal'' case for the simplicial sphere upper bound is a comparison between \[ (d + 1)(b_{i, 0} + b_{i, 1} \gamma_1(\Delta) + \ldots + b_{i, i - 1} \gamma_{i - 1}(\Delta)) \] and \[ (b_{i, 0} + b_{i, 1} \gamma_1(\Delta) + \ldots + b_{i, i - 1} \gamma_{i - 1}(\Delta))^{\langle i \rangle}. \] The asymptotic equivalence $a^{\langle i \rangle} \sim C_i a^{\frac{i + 1}{i}}$ indicates that the latter is larger and show that $\gamma_i(\Delta)^{\langle i \rangle}$ is the smaller upper bound if $\gamma_1(\Delta), \ldots, \gamma_{i - 1}(\Delta)$ are sufficiently large.

	\end{proof}

	In general, we can obtain more explicit polynomial bounds in $g_1(\Delta)$ involving the gamma vector components in terms of the (nonnegative) differences of binomial coefficients $b_{i, j}$ from Proposition \ref{gamtohgvect}. The main tool is a consequence of the $M$-vector condition giving $h$-vectors of simplicial complexes. While we will focus on realizations by simplicial spheres, there are analogues for Cohen--Macaulay/shellable simplicial complexes as usual. \\
	
	\begin{thm} (Theorem 5.11 on p. 2119 of \cite{El}, Theorem 5.1.15 on p. 219 -- 220 of \cite{BH}, Theorem 2.2 on p. 56 of \cite{St}) \label{avghbd} \\
		If $h = (h_0, \ldots, h_d)$ is the $h$-vector of a Cohen--Macaulay simplicial complex, we have that \[ q h_q \le h_1(1 + h_1 + \ldots + h_{q - 1}). \]
	\end{thm}
	
	If $\Delta$ is a simplicial sphere, Part 2c) of Theorem \ref{gconjsphere} implies that $g(\Delta) = h(T)$ for some Cohen--Macaulay simplicial complex $T$. This implies the following: \\
	
	\begin{cor} \label{spheregambds}
		Suppose that $\Delta$ is a $(d - 1)$-dimensional simplicial sphere.
		
		\begin{enumerate}
			\item For each $1 \le q \le \frac{d}{2}$, we have that \[ q g_q(\Delta) \le g_1(\Delta)(1 + g_1(\Delta) + \ldots + g_{q - 1}(\Delta)) = g_1(\Delta) h_{q - 1}(\Delta). \]
			
			\item For $\gamma(\Delta)$, Part 1 implies that \[ \gamma_q(\Delta) \le \sum_{j = 0}^{q - 1} \left[ \frac{1}{q} (\gamma_1(\Delta) + d - 1) \left( \sum_{i = j}^{q - 1} b_{i, j} \right) - b_{q, j} \right] \gamma_j(\Delta), \]  where the $b_{k, \ell}$ are the (nonnegative) differences of binomial coefficients from Proposition \ref{gamtohgvect}. \\
			
			Note that $g_1(\Delta) = h_1(\Delta) - 1$ and $h_1(\Delta) = \gamma_1(\Delta) + d$. This is consistent with $\gamma_1(\Delta) = b_{1, 0} + \gamma_1(\Delta)$ from Proposition \ref{gamtohgvect}. \\ 
			
			\item Suppose that $2 \le q \le \frac{d}{2}$. The recursion in Part 2 gives an upper bound of $\gamma_q(\Delta)$ that is a polynomial in $\gamma_1(\Delta)$ of degree $\le q$ with coefficients given by polynomials in $b_{i, j}$ (from Proposition \ref{gamtohgvect}) of degree $\le q$. If $\gamma_i(\Delta)$ has the large lower bounds from Part 2 of Remark \ref{matcoeffsize}, the monomials from these polynomials in the $b_{i, j}$ are at most linear in the $\gamma_j(\Delta)$. As a polynomial in $\gamma_1(\Delta)$, the bound we obtain is a multiple of $\gamma_1(\Delta) + d - 1$. \\
			
			\item We have ``closed'' upper bounds of $q_q(\Delta)$ and $\gamma_q(\Delta)$ given by \[ g_q(\Delta) \le \frac{g_1(\Delta) \binom{g_1(\Delta) + q - 1}{q - 1}}{q} \] and \[ \gamma_q(\Delta) \le \frac{ (\gamma_1(\Delta) + (d - 1)) \binom{\gamma_1(\Delta) + (d - 1) + (q - 1)}{q - 1}  }{q} - (b_{q, 0} + b_{q, 1} \gamma_1(\Delta) + \ldots + b_{q, q - 1} \gamma_{q - 1}(\Delta)). \]
			
			If $\gamma_q(\Delta) \ge 0$, this implies that \[ \frac{ (\gamma_1(\Delta) + (d - 1)) \binom{\gamma_1(\Delta) + (d - 1) + (q - 1)}{q - 1}  }{q} \ge b_{q, 0} + b_{q, 1} \gamma_1(\Delta) + \ldots + b_{q, q - 1} \gamma_{q - 1}(\Delta). \]
			
			For $1 \le j, k \le q - 1$, the bound above gives upper bounds on $b_{q, k}$ sizes compared to the $\gamma_j(\Delta)$ and $d$ if $\gamma_m(\Delta) \ge 0$ for all $1 \le m \le q - 1$. \\
			
		\end{enumerate}
	\end{cor}

	\begin{proof}
		\begin{enumerate}
			\item This follows directly from Part 2c) of Theorem \ref{gconjsphere} and Part 3 of Theorem \ref{hvectcmshell}. \\
			
			\item Since \[ g_\ell(\Delta) = b_{\ell, 0} + b_{\ell, 1} \gamma_1(\Delta) + \ldots + b_{\ell, \ell - 1} \gamma_{\ell - 1}(\Delta) + \gamma_\ell(\Delta) \] by Proposition \ref{gamtohgvect} and $b_{1, 0} = d - 1$, 
			
			we have 
			
			\begin{align*}
				q \left( \sum_{k = 0}^{q - 1} b_{q, k} \gamma_k(\Delta) + \gamma_q(\Delta) \right) &\le  (\gamma_1(\Delta) + d - 1) \\
				&\cdot ( 1 + (b_{1, 0} + \gamma_1(\Delta)) + \ldots + (b_{q - 1, 0} + b_{q - 1, 1} \gamma_1(\Delta) + \ldots + b_{q - 1, q - 2} \gamma_{q - 2}(\Delta) + \gamma_{q - 1}(\Delta)) ) \\
				&= (\gamma_1(\Delta) + d - 1) \left( \sum_{j = 0}^{q - 1} \left( \sum_{i = j}^{q - 1} b_{i, j} \right) \gamma_j(\Delta) \right) \\
				\Longrightarrow \gamma_q(\Delta) &\le \frac{1}{q} (\gamma_1(\Delta) + d - 1) \left( \sum_{j = 0}^{q - 1} \left( \sum_{i = j}^{q - 1} b_{i, j} \right) \gamma_j(\Delta) \right) - \sum_{k = 0}^{q - 1} b_{q, k} \gamma_k(\Delta) \\
				&= \sum_{j = 0}^{q - 1} \left[ \frac{1}{q} (\gamma_1(\Delta) + d - 1) \left( \sum_{i = j}^{q - 1} b_{i, j} \right) - b_{q, j} \right] \gamma_j(\Delta). 
			\end{align*}

			\item The $q = 2$ case follows from the $q = 2$ case of Part 1. Suppose that $\gamma_m(\Delta)$ is bounded above by a polynomial in $\gamma_1(\Delta)$ of degree $\le m$ with coefficients given by polynomials in $b_{i, j}$ (Proposition \ref{gamtohgvect}) of degree $\le m$ for each $2 \le m \le q - 1$. In the inequality \[ \gamma_q(\Delta) \le \sum_{j = 0}^{q - 1} \left[ \frac{1}{q} (\gamma_1(\Delta) + d - 1) \left( \sum_{i = j}^{q - 1} b_{i, j} \right) - b_{q, j} \right] \gamma_j(\Delta) \] from Part 2, each $\gamma_j(\Delta)$ with $0 \le j \le q - 1$ is bounded above by a polynomial in $\gamma_1(\Delta)$ of degree $\le q - 1$ with coefficients from $b_{i, j}$ of degree $\le q - 1$. Multiplying this by a linear term in $\gamma_1(\Delta)$ and $b_{i, j}$ yields a polynomial in $\gamma_1(\Delta)$ of degree $\le q$ with coefficients from $b_{i, j}$ of degree $\le q$. \\
			
			\item Combining Theorem \ref{avghbd} with the upper bound theorem yields a closed polynomial upper bound in $g_1(\Delta)$ for $g_q(\Delta)$. Since $1 + g_1(\Delta) + \ldots + g_{q - 1}(\Delta) = h_{q - 1}(\Delta)$, the bound \[ h_{q - 1}(\Delta) \le \binom{n - d + q - 1 - 1}{q - 1} \] from the upper bound theorem (p. 59 of \cite{St}, p. 228 of \cite{BH}) and $g_1(\Delta) = h_1(\Delta) - 1 = f_0(\Delta) - d - 1$ (p. 213 of \cite{BH}) implies that \[ \frac{g_q(\Delta)}{g_1(\Delta)} \le \frac{\binom{g_1(\Delta) + q - 1}{q - 1}}{q}. \]
			
			When $g_1(\Delta)$ is large, we have an (asymptotic) upper bound that is about $\frac{g_1(\Delta)^q}{q!}$. \\
			
			Since \[ g_q(\Delta) = b_{q, 0} + b_{q, 1} \gamma_1(\Delta) + \ldots + b_{q, q - 1} \gamma_{q - 1}(\Delta) + \gamma_q(\Delta) \] using the (nonnegative) differences of binomial coefficients $b_{i, j}$ (Proposition \ref{gamtohgvect}) and \[ g_1(\Delta) = b_{1, 0} + \gamma_1(\Delta) = \gamma_1(\Delta) + (d - 1), \] this implies that 
			
			\begin{equation*}
				\begin{gathered}
					\frac{b_{q, 0} + b_{q, 1} \gamma_1(\Delta) + \ldots + b_{q, q - 1} \gamma_{q - 1}(\Delta) + \gamma_q(\Delta)}{\gamma_1(\Delta) + (d - 1)} \le \frac{ \binom{\gamma_1(\Delta) + (d - 1) + (q - 1)}{q - 1} }{q} \\
					\Longrightarrow \gamma_q(\Delta) \le \frac{ (\gamma_1(\Delta) + (d - 1)) \binom{\gamma_1(\Delta) + (d - 1) + (q - 1)}{q - 1}  }{q} - (b_{q, 0} + b_{q, 1} \gamma_1(\Delta) + \ldots + b_{q, q - 1} \gamma_{q - 1}(\Delta)). 
				\end{gathered} 
			\end{equation*}

		\end{enumerate}
	\end{proof}
	
	\begin{rem} \textbf{(Analogues for Cohen--Macaulay simplicial complexes) \\}
		Inequalities analogous to those in Corollary \ref{spheregambds} hold with the binomial coefficients $a_{i, j}$ replacing (nonnegative) differences $b_{i, j}$ from Proposition \ref{gamtohgvect} when we consider gamma vectors of Cohen--Macaulay simplicial complexes. This is a direct substitution for Part 2 and Part 3. In Part 4, we would substitute in the upper bound theorem at an earlier point and would end up with a different expression.  \\
	\end{rem}
	
	\color{black}

	\subsection{Simplicial spheres satisfying the link condition} \label{linkcondimp}

	Next, we focus on realizations of nonnegative gamma vectors by simplicial spheres $\Delta$ satisfying the link condition, which is defined below. \\
	
	\begin{defn} \cite{Nev}
		A simplicial complex $\Delta$ satisfies the \textbf{link condition} if $\lk_\Delta(a) \cap \lk_\Delta(b) = \lk_\Delta(a, b)$ for every edge $\{ a, b \} \in \Delta$. \\
	\end{defn}
	
	This is a natural condition to consider from the perspective of piecewise linear (PL) homeomorphisms since they can be decomposed into edge subdivisions and contractions (Corollary [10:2d] on p. 302 of \cite{Alex}, Corollary 4.1 on p. 75 of \cite{LN}) and the link condition is equivalent to the contraction of every edge preserving the PL homeomorphism class of a compact PL manifold without boundary \cite{Nev}. In addition, the condition $\lk_\Delta(a) \cap \lk_\Delta(b) = \lk_\Delta(a, b)$ for edges $\{ a, b \} \in \Delta$ is a common application of flagness, which ties many examples from gamma positivity together (see \cite{Athgam}, \cite{Gal}). \\

	In Section \ref{mainlinkineq}, we convert local data from nonnegativity (Corollary \ref{globgintlac}) and the $M$-vector condition (Corollary \ref{gglobbounds}) on the $g$-vector $g(\widetilde{\Delta}_e)$ of the contractions $\widetilde{\Delta}_e$ of edges into global data encoding ``average'' properties of these local parts. These conditions are combined in to a single inequality in Corollary \ref{diffcush}. The main tool for this conversion was a local-global formula (Lemma \ref{glocedgesum}) that is a $g$-vector analogue of one appearing in many different contexts including the proof of the upper bound conjecture and the generalized lower bound theorem for balanced simplicial polytopes (Proposition \ref{hlocsum}). \\

	Along the way, we show that the global/averaged nonnegativity of the $g$-vectors of the edge contractions $\widetilde{\Delta}$ is usually trivially implied by the $M$-vector condition on the $g$-vectors of arbitrary simplicial spheres unless $g_1(\Delta) < d + 10$ (Proposition \ref{interlacingnontriv}). If $\Delta$ is a simplicial sphere PL homeomorphic to the boundary of a cross polytope, this special case translates as coming from a PL homeomorphism with a very small ``net number of edge subdivisions'' in edge subdivision/contraction decompositions. \\

	In particular, this comes from a shift in perspective on these global/averaged inequalities as information on \emph{ratios} of $g$-vector components $\frac{g_{\ell + 1}(\Delta)}{g_\ell(\Delta)}$. Given this situation, we mainly focus on the global/averaged version of the $M$-vector condition on the $g$-vectors of the edge contractions $\widetilde{\Delta}_e$ of edges $e \in \Delta$. We note that individual local conditions induce ``sandwiching'' of local parts (Proposition \ref{glocbounds}) with global/averaged counterparts (Corollary \ref{gglobbounds}) when combined with the PL manifold condition. We will primarily focus on the lower bounds and their comparison with nonnegativity of the $g$-vectors of the contractions $\widetilde{\Delta}_e$ of edges $e \in \Delta$ (Corollary \ref{diffcush}) since this is the bound specific to spheres satisfying the link condition rather than PL manifolds in general. \\

	The general shape of the global/averaged $M$-vector conditions on $g$-vectors of contractions $\widetilde{\Delta}_e$ of edges $e \in \Delta$ involves a lower bound on $\frac{g_{k + 2}(\Delta)}{g_{k + 1}(\Delta)}$ in terms of $\frac{g_{k + 1}(\Delta)}{g_k(\Delta)}$, the number of edges $f_1(\Delta)^{ \frac{1}{k} }$, and $\frac{g_k(\Delta)^{ \frac{k + 1}{k} }}{g_{k + 1}(\Delta)}$ (Proposition \ref{ratlowbd} and Example \ref{simpratlowbd}). The last parameter measures how far the $M$-vector condition $g_{k + 1}(\Delta) \le g_k(\Delta)^{ \langle k \rangle }$ is from equality (Part 2c) of Theorem \ref{gconjsphere} and Corollary \ref{macpowasymp}). Increasing these parameters decreases the resulting lower bound of $\frac{g_{k + 2}(\Delta)}{g_{k + 1}(\Delta)}$. \\

	We then take a closer look at the behavior of these parameters in high-dimensional spheres $\Delta$ satisfying the link condition such that $\gamma(\Delta) \ge 0$ in Section \ref{asymplink} (summarized and specialized in Section \ref{sumconst}). Both this further analysis and the earlier observations on the shape of the inequality involve a return to the perspective of the global inequalities as storing averaging data of local parts. A sort of parallel to the earlier global/averaged $g$-vector condition on contractions $\widetilde{\Delta}_e$ of edges $e \in \Delta$ is that individual global/averaged $M$-vector conditions on $g(\widetilde{\Delta}_e)$ for high-dimensional spheres are trivially implied by the $M$-vector condition on $g(\Delta)$ for arbitrary simplicial simplicial spheres when $g_1(\Delta)$ (from ``net number of edge subdivisions'') is superlinear and the number of edges $f_1(\Delta)$ is large compared to the rate of increase of the $g$-vector components (Theorem \ref{largelinktriv}). This means that the $g$-vector is not affected by the link condition in these cases. \\
	
	The same applies to possible (nonnegative) gamma vectors realized since they are obtained from $g$-vectors via invertible linear transformations (see Proposition \ref{gamtohgvect}). Since gamma vectors of barycentric subdivisions are equal to $f$-vectors of auxiliary simplicial complexes, this suggests a closer connection between gamma vectors and $f$-vectors in such settings in general. Combining this with the general triviality of the global/averaged $M$-vector condition on contractions $\widetilde{\Delta}_e$ of edges $e \in \Delta$ when $g_1(\Delta)$ is linear in $d$ (Part 2 of Corollary \ref{consttriv}), this parallel can be sort of extended to a dichotomy between the behavior of such simplicial spheres where $g_1(\Delta)$ is linear in $d$ and those where it is superlinear in $d$. Additional triviality cases come from ``generic'' cases among asymptotic equivalence relations between the number of edges $f_1(\Delta)$ and rates of increase of $g$-vector components (Part 1 of Corollary \ref{consttriv}). \\

	The special/maximal cases where among asymptotic equivalence relations between the number of edges $f_1(\Delta)$ and rates of increase of $g$-vector components (Corollary \ref{constlinknontriv}) yields cases where the global/averaged $M$-vector condition on contractions $\widetilde{\Delta}_e$ of edges $e \in \Delta$ are satisfied nontrivially by such high-dimensional simplicial spheres. The resulting limiting inequalities take the form of lower bounds between 0 and 1 for the quotient of $\frac{g_{k + 2}(\Delta)}{g_{k + 1}(\Delta)}$ by $\frac{g_{k + 1}(\Delta)}{g_k(\Delta)}$ (Corollary \ref{nontrivlinkgam} and Corollary \ref{constlinknontriv}). Since we often expect $\frac{g_{\ell + 1}(\Delta)}{g_\ell(\Delta)}$ to decrease as $\ell$ increases (Remark \ref{indexratchange}), this lower bound can be thought of as a ``control'' on this decrease. While there is initially a wide range of possible outputs, the (non)triviality depends on the comparison of ``average over edges $e \in \Delta$'' of differences $g_k(\Delta) - g_{k - 1}(\lk_\Delta(e))$ with a modified averaging over edges of ratios $\frac{g_k(\Delta)^{ \frac{k + 1}{k} }}{g_{k + 1}(\Delta)}$ measuring how far the $M$-vector inequality $g_{k + 1}(\Delta) \le g_k(\Delta)^{ \langle k \rangle}$ (Part 2c) of Theorem \ref{gconjsphere}) is from equality. This can also be phrased in terms of the size of $f_1(\Delta)^k$ compared to $g_{k + 1}(\Delta)$. Decreases in the former and increases in the latter yield larger lower bounds of this quotient of $\frac{g_{k + 2}(\Delta)}{g_{k + 1}(\Delta)}$ by $\frac{g_{k + 1}(\Delta)}{g_k(\Delta)}$. \\
	
	For gamma vectors $\gamma(\Delta)$, lower bounds on $\frac{g_{k + 2}(\Delta)}{g_{k + 1}(\Delta)}$ can be phrased as lower bounds on $\gamma_{k + 2}(\Delta)$ since this term only appears in $g_{k + 2}(\Delta)$ (see Proposition \ref{gamtohgvect}). However, the asymptotic effect on high-dimensional spheres satisfying the link condition with $\gamma(\Delta) \ge 0$ is the largest when the gamma vector components involved are asymptotically dominant in the $g$-vector components of the same index. For example, this is occurs if $\gamma_{\ell + 1} \ge (d + 1)g_\ell(\Delta)$ (e.g. when $g_{\ell + 1}(\Delta) \ge 2(d + 1) g_\ell(\Delta)$). While the results discussed here involve \emph{individual} inequalities associated to constant indices $k$, the underlying tools (Proposition \ref{linkinitasymp}, Corollary \ref{asympgratinit}, Corollary \ref{coefflinkasymp}) apply to the behavior of \emph{collections} of inequalities associated to varying indices. \\ 
	
	\begin{rem} \textbf{(Cohen--Macaulay simplicial complex generalizations) \\}
		The same reasoning as above implies analogous results for $h$-vectors of Cohen--Macaulay simplicial complexes whose edge contractions have $h$-vectors realized by Cohen--Macaulay simplicial complexes. However, we focus on simplicial spheres here since the link condition seems more natural to consider in the context of the PL homeomorphisms discussed above.
	\end{rem}

	\subsubsection{Main underlying inequalities} \label{mainlinkineq}
	
	Since the link condition is a local condition, we will start by giving a local-global relation between $g$-vectors of links over edges and the entire simplicial complex. It is a variant of a local-global result that was used in the proof of the upper bound conjecture for convex polytopes (p. 183 of \cite{McM2}). Analogous identities also appeared in the context of gamma vectors to study the top component for odd $d$ (Lemma 1.2.1 on p. 274 of \cite{Gal}) and were combined with Lagrange inversion to relate the gamma polynomial to differences in local-global behavior (Theorem 1.7 on p. 8 and p. 10 -- 11 of \cite{Pexp}). \\
	
	We will use the link condition to give a $g$-vector variant involving links over edges. \\
	
	\begin{prop} (p. 183 of \cite{McM2}, Proposition 2.3 on p. 648 -- 649 of \cite{Swa}, p. 1688 of \cite{JKM}) \label{hlocsum} \\
		Let $\Delta$ be a pure $(d - 1)$-dimensional simplicial complex. Then, we have \[ \sum_{v \in V(\Delta)} h_i(\lk_\Delta(v)) = (i + 1) h_{i + 1}(\Delta) + (d - i) h_i(\Delta). \]
	\end{prop}

	\begin{lem} \label{glocedgesum}
		Suppose that $\Delta$ is a pure $(d - 1)$-dimensional simplicial complex satisfying the link condition. Then, we have \[ 2 \sum_{ \substack{ e \in \Delta \\ |e| = 2 }} g_k(\lk_\Delta(e)) = (k + 1)(k + 2) g_{k + 2}(\Delta) + 2(k + 1)(d - k) g_{k + 1}(\Delta) + (d - k)(d - k + 1) g_k(\Delta) \] for all $0 \le k \le \frac{d}{2} - 1$. The convention we will use is that $g_m(\Delta) = 0$ for all $m \ge \frac{d}{2} + 1$. \\
	\end{lem}
	
	\begin{proof}
		
		We will first give a $g$-vector analogue of Proposition \ref{hlocsum} for sums of links over vertices.
		
		\begin{align*}
			\sum_{v \in V(\Delta)} g_k(\lk_\Delta(v)) &= \sum_{v \in V(\Delta)} h_k(\lk_\Delta(v)) - \sum_{v \in V(\Delta)} h_{k - 1}(\lk_\Delta(v)) \\
			&= (k + 1) h_{k + 1}(\Delta) + (d - k) h_k(\Delta) - k h_k(\Delta) - (d - k + 1) h_{k - 1}(\Delta) \\
			&= (k + 1) h_{k + 1}(\Delta) + (d - k) h_k(\Delta) - (k + 1) h_k(\Delta) - (d - k) h_{k - 1}(\Delta) + h_k(\Delta) - h_{k - 1}(\Delta) \\
			&= (k + 1)(h_{k + 1}(\Delta) - h_k(\Delta)) + (d - k)(h_k(\Delta) - h_{k - 1}(\Delta)) + (h_k(\Delta) - h_{k - 1}(\Delta)) \\
			&= (k + 1) g_{k + 1}(\Delta) + (d - k) g_k(\Delta) + g_k(\Delta) \\
			&= (k + 1) g_{k + 1}(\Delta) + (d - k + 1) g_k(\Delta)
		\end{align*}
		
		Given an edge $\{ v, w \} \in \Delta$, the link condition then implies that \[ \lk_{\lk_\Delta(v)}(w) = \lk_\Delta(v) \cap \lk_\Delta(w) = \lk_\Delta(v, w) \] and
		
		\begin{align*}
			2 \sum_{ \substack{ e \in \Delta \\ |e| = 2 }} g_k(\lk_\Delta(e)) &= \sum_{v \in V(\Delta)} \left( \sum_{w \in V(\lk_\Delta(v))} g_k(\lk_\Delta(v, w)) \right) \\
			&= \sum_{v \in V(\Delta)} [ (k + 1) g_{k + 1}(\lk_\Delta(v)) + (d - k) g_k(\lk_\Delta(v)) ] \\
			&= (k + 1) \sum_{v \in V(\Delta)} g_{k + 1}(\lk_\Delta(v)) + (d - k) \sum_{v \in V(\Delta)} g_k(\lk_\Delta(v)) \\
			&= (k + 1) [ (k + 2) g_{k + 2}(\Delta) + (d - k) g_{k + 1}(\Delta) ] + (d - k) [ (k + 1) g_{k + 1}(\Delta) + (d - k + 1) g_k(\Delta) ] \\
			&= (k + 1) (k + 2) g_{k + 2}(\Delta) + 2(k + 1)(d - k) g_{k + 1}(\Delta) + (d - k)(d - k + 1) g_k(\Delta).
		\end{align*}

	\end{proof}
	
	In order to relate the local-global properties back to simplicial complexes satisfying the link condition, we record changes in $h$-vectors and $g$-vectors after edge contractions (also written down on p. 5 - 6 of \cite{Prec}). \\

	\begin{lem} \label{edgeconhgvec}
		Let $\Delta$ be a $(d - 1)$-dimensional simplicial complex satisfying the link condition. Given an edge $e \in \Delta$, let $\widetilde{\Delta}_e$ be the contraction of $e$ in $\Delta$. Then, we have \[ h_{\widetilde{\Delta}_e}(x) = h_\Delta(x) - x h_{\lk_\Delta(e)}(x). \] This implies that \[ h_k(\widetilde{\Delta}_e) = h_k(\Delta) - h_{k - 1}(\lk_\Delta(e)) \] and \[ g_k(\widetilde{\Delta}_e) = g_k(\Delta) - g_{k - 1}(\lk_\Delta(e)). \]
	\end{lem}
	
	\begin{proof}
		In $\widetilde{\Delta}_e$, the $(k - 1)$-dimensional faces split into the following types:

		\begin{itemize}
			\item $(k - 1)$-dimensional faces of $\Delta$ \emph{not} containing $e$
			
			\item Taking a $(k - 1)$-dimensional face of $\Delta$ containing only one of $a$ or $b$ and replacing it by the contraction point of $e = \{ a, b \}$. The faces remain distinct unless removing $a$ or $b$ yields the same face (which must then be in $\lk_\Delta(e)$). In the latter case, we need to remove the ``second copy'' in $\Delta$ since both $a$ and $b$ are replaced by the contraction point of $e = \{ a, b \}$.  
			
			\item $(k - 1)$-dimensional faces of $\Delta$ containing $e$
		\end{itemize}

		This means that $f_{k - 1}(\widetilde{\Delta}_e) = f_{k - 1}(\Delta) - f_{k - 2}(\lk_\Delta(e)) - f_{k - 3}(\lk_\Delta(e))$ and \[ f_{\widetilde{\Delta}_e}(x) = f_\Delta(x) - (x^2 + x) f_{\lk_\Delta(e)}(x). \]
		
		\color{black} 
		Since \[ \sum_{i = 0}^D h_i(S) x^{D - i} = \sum_{i = 0}^D f_{i - 1}(S) (x - 1)^{D - i} \] for a $(D - 1)$-dimensional simplicial complex $S$ (p. 80 of \cite{St}), this implies that \[ x^d h_{\widetilde{\Delta}_e} \left( \frac{1}{x} \right) = x^d h_\Delta \left( \frac{1}{x} \right) - x \cdot x^{d - 2} h_{\lk_\Delta(e)} \left( \frac{1}{x} \right) = x^d h_\Delta \left( \frac{1}{x} \right) - x^{d - 1} h_{\lk_\Delta(e)} \left( \frac{1}{x} \right) \] and $h_{d - k}(\widetilde{\Delta}_e) = h_{d - k}(\Delta) - h_{d - k - 1}(\lk_\Delta(e))$. This implies that $h_{\widetilde{\Delta}_e}(x) = h_\Delta(x) - x h_{\lk_\Delta(e)}(x)$. \\
	\end{proof}

	Moving back towards \emph{simplicial spheres} satisfying the link condition, we first study the effect of nonnegativity of the $g$-vector on relationships between $g$-vector components. We can interpret the initial ``interlacing bounds'' as a control on the growth rate of the $g$-vector components that induces ``sandwiching'' of $g_{k + 1}(\Delta)$ between multiples of $g_k(\Delta)$. \\
	
	\begin{cor} \textbf{(Interlacing bounds from $g(\widetilde{\Delta}_e) \ge 0$) \\} \label{globgintlac}
		Suppose that $\Delta$ is a pure $(d - 1)$-dimensional simplicial complex satisfying the link condition such that the contraction $\widetilde{\Delta}_e$ of every edge $e \in \Delta$ satisfies $g(\widetilde{\Delta}_e) \ge 0$. \\
		
		\begin{enumerate}
			\item We have \[ 2 \left( \binom{d + 1}{2} - (k + 1)(d - k) + d g_1(\Delta) + g_2(\Delta) \right) g_{k + 1}(\Delta) \ge (k + 1)(k + 2) g_{k + 2}(\Delta) + (d - k)(d - k + 1) g_k(\Delta) \]

			for all $0 \le k \le \frac{d}{2}$, where we will take $g_m(\Delta) = 0$ for all $m \ge \frac{d}{2}$. Note that \[ f_1(\Delta) = \binom{d + 1}{2} + d g_1(\Delta) + g_2(\Delta). \] The coefficient of $g_{k + 1}(\Delta)$ is strictly positive for sufficiently large $d$ if $g_1(\Delta) \ge 0$ and $g_2(\Delta) \ge 0$. \\
			
			\item Suppose that $g(\Delta) \ge 0$ and $d$ is fixed. If $1 \le k \le \frac{d}{2} - 1$, then there are constants $\alpha_k, \beta_k > 0$ such that \[ \alpha_k g_k(\Delta) \le g_{k + 1}(\Delta) \le \beta_k g_k(\Delta). \] We can take \[ \alpha_k = \frac{(d - k)(d - k + 1)}{2 \left( \binom{d + 1}{2} - (k + 1)(d - k) + d g_1(\Delta) + g_2(\Delta) \right)} \] and \[ \beta_k = \frac{2 \left( \binom{d + 1}{2} - k(d - k + 1) + d g_1(\Delta) + g_2(\Delta) \right)}{k(k + 1)}. \]
			
			Note that \[ \alpha_k \ge \frac{  \left( \frac{d}{2} + 1 \right) \left( \frac{d}{2} + 2 \right) }{ 2 \left( \binom{d + 1}{2} - 2(d - 1) + d g_1(\Delta) + g_2(\Delta) \right) } \] and \[ \beta_k \le \binom{d + 1}{2} - 2(d - 1) + d g_1(\Delta) + g_2(\Delta). \] The lower bound on $\alpha_k$ is nonnegative for sufficiently large $d$ and for arbitrary $d$ if $g_1(\Delta) \ge 3$. \\
			
		\end{enumerate}
		
	\end{cor}

	\begin{proof}
		\begin{enumerate}
			\item Given an edge $e \in \Delta$, the assumption $g(\widetilde{\Delta}_e) \ge 0$ and Lemma \ref{edgeconhgvec} imply that $g_{k + 1}(\Delta) \ge g_k(\lk_\Delta(e))$ for all $0 \le k \le \frac{d}{2} - 1$. Taking the sum over all edges $e \in \Delta$, this means that \[ f_1(\Delta) g_{k + 1}(\Delta) \ge \sum_{ \substack{ e \in \Delta \\ |e| = 2 } } g_k(\lk_\Delta(e)). \] Since \[ f_1(\Delta) = \binom{d}{2} + (d - 1) h_1(\Delta) + h_2(\Delta) = \binom{d + 1}{2} + d g_1(\Delta) + g_2(\Delta) \] by Lemma 5.1.8 on p. 213 of \cite{BH}, we can rewrite this as \[ \left( \binom{d + 1}{2} + d g_1(\Delta) + g_2(\Delta) \right) g_{k + 1}(\Delta) \ge \sum_{ \substack{ e \in \Delta \\ |e| = 2 } } g_k(\lk_\Delta(e)). \]
			
			After multiplying both sides by 2, Lemma \ref{glocedgesum} implies that 
			
			\begin{align*}
				2 \left( \binom{d + 1}{2} + d g_1(\Delta) + g_2(\Delta) \right) g_{k + 1}(\Delta) &\ge (k + 1)(k + 2) g_{k + 2}(\Delta) + 2(k + 1)(d - k) g_{k + 1}(\Delta) \\ 
				&+ (d - k)(d - k + 1) g_k(\Delta) \\
				\Longrightarrow 2 \left( \binom{d + 1}{2} - (k + 1)(d - k) + d g_1(\Delta) + g_2(\Delta) \right) g_{k + 1}(\Delta) &\ge (k + 1)(k + 2) g_{k + 2}(\Delta) + (d - k)(d - k + 1) g_k(\Delta).
			\end{align*}
			
			\item Since $g(\Delta) \ge 0$, Part 1 implies that \[ (d - k)(d - k + 1) g_k(\Delta) \le 2 \left( \binom{d + 1}{2} - (k + 1)(d - k) + d g_1(\Delta) + g_2(\Delta) \right) g_{k + 1}(\Delta) \]
			
			and we can take \[ \alpha_k = \frac{(d - k)(d - k + 1)}{2 \left( \binom{d + 1}{2} - (k + 1)(d - k) + d g_1(\Delta) + g_2(\Delta) \right)}. \]
			
			The lower bound on $\alpha_k$ follows from minimizing the numerator and maximizing the denominator under the condition $1 \le k \le \frac{d}{2}$. This means the maximum of $(d - k)(d - k + 1)$ in the numerator and the maximum of $(k + 1)(d - k)$ in the denominator. \\
			
			For the upper bound and $\beta_k$, we replace $k$ by $k - 1$ in Part 1 to obtain \[ 2 \left( \binom{d + 1}{2} - k(d - k + 1) + d g_1(\Delta) + g_2(\Delta) \right) g_k(\Delta) \ge k(k + 1) g_{k + 1}(\Delta) + (d - k + 1)(d - k + 2) g_{k - 1}(\Delta). \] 
			
			Since $g(\Delta) \ge 0$, this implies that \[ k(k + 1) g_{k + 1}(\Delta) \le 2 \left( \binom{d + 1}{2} - k(d - k + 1) + d g_1(\Delta) + g_2(\Delta) \right) g_k(\Delta) \] and we can take \[ \beta_k = \frac{2 \left( \binom{d + 1}{2} - k(d - k + 1) + d g_1(\Delta) + g_2(\Delta) \right)}{k(k + 1)}. \]
			
			Similarly, the upper bound on $\beta_k$ follows from maximizing the numerator and minimizing the denominator under the condition $1 \le k \le \frac{d}{2} - 1$. This means taking the minimum of $k(d - k + 1)$ in the numerator and the minimum of $k(k + 1)$ in the denominator. \\
			
		\end{enumerate}
	\end{proof}
	
	We now make some comments related to sizes of parameters involved in cases of interest. \\
	
	\begin{exmp} \textbf{(Behavior in specific cases) \\} \label{constexmp}
		\begin{enumerate}
			\item \textbf{(Lower bounds on $g_1(\Delta)$) \\}
			
			If $\gamma_1(\Delta), \gamma_2(\Delta) \ge 0$, $\Delta$ is balanced \cite{JKM}, or $\Delta$ is a doubly Cohen--Macaulay flag simplicial complex (Observation 6.1 on p. 1377 of \cite{NPT} and Theorem 1.3 on p. 19 \cite{Athsom}), we have that $h_i(\Delta) \ge \binom{d}{i}$ for $i = 1, 2$ (and for all $0 \le i \le \frac{d}{2}$ if $\gamma(\Delta) \ge 0$ or we are in the latter case). This implies that 
			
			\begin{align*}
				f_1(\Delta) &= \binom{d}{2} + (d - 1) h_1(\Delta) + h_2(\Delta) \\
				&\ge \binom{d}{2} + (d - 1) \cdot d + \binom{d}{2} \\
				&= d(d - 1) + (d - 1) \cdot d \\
				&= 2d(d - 1).
			\end{align*}
			
			In Part 2, this means that \[ \alpha_k \le \frac{(d - k)(d - k + 1)}{ 2(2d(d - 1) - (k + 1)(d - k)) } \] and \[ \beta_k \ge \frac{ 2(2d(d - 1) - k(d - k + 1)) }{k(k + 1)}. \]
			
			Since $1 \le k \le \frac{d}{2} - 1$, maximizing the numerator and minimizing the denominator implies that \[ \alpha_k \le \frac{d(d - 1)}{2 \left( 2d(d - 1) - \frac{d}{2} \left( \frac{d}{2} + 1 \right) \right) } \approx \frac{2}{7}. \]

			Similarly, minimizing the numerator and maximizing the denominator of the lower bound of $\beta_k$ implies that \[ \beta_k \ge \frac{ 2 \left( 2d(d - 1) - \left( \frac{d}{2} - 1 \right) \left( \frac{d}{2} + 2 \right) \right) }{ \frac{d}{2} \left( \frac{d}{2} - 1 \right) } \approx 14. \]
			
			\color{black} 
			\item \textbf{(Case where $g_1(\Delta)$ is linear in $d$) \\}
			
			Consider the case where $g_1(\Delta)$ is linear in $d$ as in the setting of Proposition \ref{gamposling1h1}. While the growth rate with respect to $d$ appears to be ``larger'' on the right hand side, such a claim would depend on a uniform ratio between terms on both sides of the inequalities $g_{k + 1}(\Delta) \ge g_k(\lk_\Delta(e))$ whose upper and lower bound sums are compared. For example, suppose that $\Delta$ is the boundary of a $d$-dimensional cross polytope. Then, dividing the upper bound by the lower bound would give the ratio \[ \frac{ \binom{d}{k + 1} - \binom{d}{k} }{ \binom{d - 2}{k} - \binom{d - 2}{k - 1} }. \]
			
			If $k = 1$, this is equal to \[ \frac{ \binom{d}{2} - \binom{d}{1} }{ \binom{d - 2}{1} - \binom{d - 2}{0} } = \frac{ \binom{d}{2} - d }{(d - 2) - 1} \approx \frac{d}{2}, \]
			
			where the approximation is taken when $d \gg 0$. \\
			
			On the other hand, substituting in $k = \frac{d}{2} - 1$ gives \begin{align*}
				\frac{ \binom{d}{\frac{d}{2}} - \binom{d}{\frac{d}{2} - 1} }{ \binom{d - 2}{\frac{d}{2} - 1} - \binom{d - 2}{\frac{d}{2} - 2} } &= \frac{ \binom{d}{\frac{d}{2}} \cdot \frac{1}{ \frac{d}{2} + 1 } }{ \binom{d - 2}{\frac{d}{2} - 1} \cdot \frac{1}{ \frac{d}{2} } } \\
				&= \frac{ \binom{d}{ \frac{d}{2} } }{ \binom{d - 2}{ \frac{d}{2} - 1 } } \cdot \frac{ \frac{d}{2} }{ \frac{d}{2} + 1 } \\
				&= \frac{d(d - 1)}{ \frac{d^2}{4} } \cdot \frac{d}{d + 2} \\
				&\approx 4
			\end{align*}
			
			when $d \gg 0$. \\
			
			Thus, ratios of upper and lower bounds in $g_{k + 1}(\Delta) \ge g_k(\lk_\Delta(e))$ from $g_{k + 1}(\widetilde{\Delta}_e) \ge 0$ may vary significantly depending on $k$ and the comparison between the left and right hand sides of Corollary \ref{globgintlac} is not necessarily what we may expect from differences of powers of $d$ involved in Proposition \ref{gamposling1h1}. \\
			
		\end{enumerate}

	\end{exmp}

	Converting the bounds above into information about gamma vector components, we have ``linear sandwiching'' bounds between consecutive components that are closer to compatibility $f$-vector realization bounds from Theorem \ref{kkfvect} compared to nonnegative gamma vectors realized by arbitrary simplicial spheres (Corollary \ref{posgamfcomp}). This is especially the case when the gamma vector components are very large compared to $d$ (see Remark \ref{matcoeffsize} for further discussion). \\

	Comparing this to earlier conditions for a $g$-vector to be realized by a simplicial sphere, we can show that these ``averaged'' consequences of $g(\widetilde{\Delta}_e) \ge 0$ for contractions $\widetilde{\Delta}_e$ of edges $e \in \Delta$ only yield nontrivial conditions beyond those of general simplicial spheres when $g_1(\Delta)$ and $g_2(\Delta)$ are small. \\
	
	\begin{prop} \textbf{(Interlacing bounds vs. general simplicial spheres) \\} \label{interlacingnontriv}
		Suppose that $\Delta$ is a pure $(d - 1)$-dimensional simplicial complex satisfying the link condition such that $g(\Delta) \ge 0$ and $g(\widetilde{\Delta_e}) \ge 0$ for every edge $e \in \Delta$.
		\begin{enumerate}
			 
			\item We have $g_\ell(\Delta) \ne 0$ for all $0 \le \ell \le \frac{d}{2}$.

			\item Suppose that $\Delta$ is a simplicial sphere satisfying the link condition such that $\gamma(\Delta) \ge 0$. In order for the ``averaged'' global version of the inequalities $g(\widetilde{\Delta}_e) \ge 0$ for all the edges $e \in \Delta$ from Corollary \ref{globgintlac} to impose a nontrivial condition on the $g_{k + 2}(\Delta)$ beyond the $M$-vector condition on the $g$-vectors of arbitrary simplicial spheres (Part 2c) of Theorem \ref{gconjsphere}), we need $g_2(\Delta) < (d^2 - 2) - (d - 3) g_1(\Delta)$. By the $g(\Delta) \ge 0$ assumption, this also implies that $g_1(\Delta) < (d + 3) + \frac{7}{d - 3}$. If $d \ge 4$, the latter bound implies that $g_1(\Delta) < d + 10$. This implies that $g_2(\Delta) \le \binom{d + 11}{2}$,  which is a stronger bound for $d \gg 0$. Since there is a linear bound on $g_1(\Delta)$, each $g_\ell(\Delta)$ with $0 \le \ell \le \frac{d}{2}$ is bounded above by a degree $\ell$ polynomial in $d$ (Corollary \ref{gamposling1h1}). \\
		\end{enumerate}
		
	\end{prop}

	\begin{proof}
		\begin{enumerate}
			\item Since the coefficient of each $g$-vector component in Part 1 of Corollary \ref{globgintlac} is strictly positive and $g(\Delta) \ge 0$, having $g_{k + 1}(\Delta) = 0$ would imply that $g_{k + 2}(\Delta) = 0$ and $g_k(\Delta) = 0$. If $g_\ell(\Delta) = 0$ for \emph{some} $1 \le \ell \le \frac{d}{2}$, this would mean that $g_k(\Delta) = 0$ for all $0 \le k \le \frac{d}{2}$. However, this is impossible in our setting. \\
			
			\item Since $g_\ell(\Delta) \ne 0$ for all $0 \le \ell \le \frac{d}{2}$ by Part 1, we can rewrite \[ 2 \left( \binom{d + 1}{2} - (k + 1)(d - k) + d g_1(\Delta) + g_2(\Delta) \right) g_{k + 1}(\Delta) \ge (k + 1)(k + 2) g_{k + 2}(\Delta) + (d - k)(d - k + 1) g_k(\Delta) \] from Part 1 of Corollary \ref{globgintlac} as \[ 2 \left( \binom{d + 1}{2} - (k + 1)(d - k) + d g_1(\Delta) + g_2(\Delta) \right) \ge (k + 1)(k + 2) \frac{g_{k + 2}(\Delta)}{g_{k + 1}(\Delta)} + (d - k)(d - k + 1) \frac{g_k(\Delta)}{g_{k + 1}(\Delta)}. \]
			
			Equivalently, we have \[ \frac{g_{k + 2}(\Delta)}{g_{k + 1}(\Delta)} \le  \frac{2}{(k + 1)(k + 2)} \left( \binom{d + 1}{2} - (k + 1)(d - k) + d g_1(\Delta) + g_2(\Delta) \right) - \frac{(d - k)(d - k + 1)}{(k + 1)(k + 2)} \frac{g_k(\Delta)}{g_{k + 1}(\Delta)}.  \]
			
			We can compare these to initial bounds from the $M$-vector condition on $g$-vectors of arbitrary simplicial spheres. Since $g_{k + 2}(\Delta) \le g_{k + 1}(\Delta)^{\langle k + 1 \rangle} \le g_{k + 1}(\Delta)^{ \frac{k + 2}{k + 1} }$ and $g_{k + 1}(\Delta) \le g_1(\Delta)^{k + 1} \Longrightarrow g_{k + 1}(\Delta)^{ \frac{1}{k + 1} } \le g_1(\Delta)$ by Corollary \ref{macpowasymp} and Corollary \ref{sphereginitbd}, we have \[ \frac{g_{k + 2}(\Delta)}{g_{k + 1}(\Delta)} \le g_1(\Delta) \] for any simplicial sphere $\Delta$. \\
			
			In order for the interlacing bounds in Part 1 of Corollary \ref{globgintlac} to give nontrivial bounds beyond those for arbitrary simplicial spheres for some index $0 \le k \le \frac{d}{2}$, we need to have  \[ \frac{2}{(k + 1)(k + 2)} \left( \binom{d + 1}{2} - (k + 1)(d - k) + d g_1(\Delta) + g_2(\Delta) \right) - \frac{(d - k)(d - k + 1)}{(k + 1)(k + 2)} \frac{g_k(\Delta)}{g_{k + 1}(\Delta)} < g_1(\Delta). \]
			
			If $\gamma(\Delta) \ge 0$, then Corollary \ref{gamgratrest} implies that $\frac{g_{k + 1}(\Delta)}{g_k(\Delta)} \ge \frac{1}{3} \Longrightarrow \frac{g_k(\Delta)}{g_{k + 1}(\Delta)} \le 3$ since we assumed that $g(\Delta) \ge 0$. Substituting this into our inequality above, we have \[ \frac{2}{(k + 1)(k + 2)} \left( \binom{d + 1}{2} - (k + 1)(d - k) + d g_1(\Delta) + g_2(\Delta) \right) - 3 \frac{(d - k)(d - k + 1)}{(k + 1)(k + 2)} < g_1(\Delta). \]
			
			If $k = 1$, this implies that 
			\begin{equation*}
				\begin{gathered} 
				\frac{1}{3} \left( \binom{d + 1}{2} - 2(d - 1) + d g_1(\Delta) + g_2(\Delta) \right) - \frac{1}{2} d(d - 1) < g_1(\Delta) \\
				\Longrightarrow \binom{d + 1}{2} - 2(d - 1) + d g_1(\Delta) + g_2(\Delta) - \frac{3}{2} d(d - 1) < 3 g_1(\Delta) \\
				\Longrightarrow g_2(\Delta) < \frac{3}{2} d(d - 1) - \binom{d + 1}{2} + 2(d - 1) - (d - 3) g_1(\Delta) \\
				= (d^2 - 2) - (d - 3) g_1(\Delta). 
				\end{gathered} 
			\end{equation*}
			
			This implies that $g_2(\Delta)$ is quadratically bounded with respect to $d$. \\
			
			In order to have $g_2(\Delta) > 0$, we need to have a linear bound on $g_1(\Delta)$ when $d \gg 0$. This is because
			
			\begin{align*}
				(d^2 - 2) - (d - 3) g_1(\Delta) &> 0 \\
				\Longleftrightarrow g_1(\Delta) &< \frac{d^2 - 2}{d - 3} \\
				&= \frac{d^2 - 9}{d - 3} + \frac{7}{d - 3} \\
				&= (d + 3) + \frac{7}{d - 3}.
			\end{align*}
			
			If $d \ge 3$, this implies that $g_1(\Delta) < d + 10$. \\ 
			
			Substituting this into the bound \[ g_2(\Delta) \le g_1(\Delta)^{ \langle 1 \rangle } = \binom{g_1(\Delta) + 1}{2} \] from the $M$-vector condition on $g$-vectors of simplicial spheres and Corollary \ref{macpowasymp}, we have that \[ g_2(\Delta) \le \binom{d + 11}{2}. \] This is a stronger bound for $d \gg 0$. \\
		\end{enumerate}
	\end{proof}

	Aside from the Dehn--Sommerville relations $h_i(\widetilde{\Delta}_e) = h_{d - i}(\widetilde{\Delta}_e)$ and nonnegativity of the $g$-vector $g(\widetilde{\Delta}_e) \ge 0$, recall that the remaining condition needed for $h(\widetilde{\Delta}_e)$ to be realized by a simplicial sphere is \[ g_{i + 1}(\widetilde{\Delta}_e) \le g_i(\widetilde{\Delta}_e)^{\langle i \rangle} \] by Theorem \ref{gconjsphere}. This is a necessary condition when $\Delta$ is a simplicial sphere (Remark \ref{plassump}) that satisfies the link condition. \\
	
	In this setting, we have a ``squeezing'' of $g$-vectors of local terms.
	
	\begin{prop} \textbf{(Link condition-induced local lower bounds and sandwiching) \\}   \label{glocbounds}
		If $\Delta$ is a simplicial sphere satisfying the link condition, then \[ g_{i + 1}(\Delta) - (g_i(\Delta) - g_{i - 1}(\lk_\Delta(e)))^{\langle i \rangle} \le g_i(\lk_\Delta(e)) \le g_{i - 1}(\lk_\Delta(e))^{\langle i - 1 \rangle}. \]
	\end{prop}
	
	\begin{proof}
		We can use Lemma \ref{edgeconhgvec} to rewrite this inequality as \[ g_{i + 1}(\Delta) - g_i(\lk_\Delta(e)) \le (g_i(\Delta) - g_{i - 1}(\lk_\Delta(e)))^{\langle i \rangle}. \]
		
		Equivalently, we have \[ g_i(\lk_\Delta(e)) \ge g_{i + 1}(\Delta) - (g_i(\Delta) - g_{i - 1}(\lk_\Delta(e)))^{\langle i \rangle}. \]
		
		Since $\Delta$ is a PL simplicial manifold, we have that $\lk_\Delta(e)$ is a simplicial sphere and \[ g_i(\lk_\Delta(e)) \le g_{i - 1}(\lk_\Delta(e))^{\langle i - 1 \rangle}. \]
		
		Combining these inequalities gives the stated bounds.
	\end{proof}
	
	We can combine these local inequalities with the local-global relation for $g$-vectors in Lemma \ref{glocedgesum} obtain a global version of this ``squeezing'' property. In order to effectively apply the estimates of pseudopowers $a^{\langle k \rangle}$ from Corollary \ref{macpowasymp}, we will use the following property. \\

	\begin{lem} \label{sumpowerbd}
		If $a_k \ge 0$ for all $1 \le k \le N$ and $p \ge 1$, then \[ \sum_{k = 1}^N a_k^p \le \left( \sum_{k = 1}^N a_k \right)^p. \]
	\end{lem}
	
	\begin{rem} \textbf{(Application to differences) \\} \label{diffpowbd} 
		The $N = 2$ case is equivalent to the statement that $(P - Q)^p \le P^p - Q^p$ if $P \ge Q \ge 0$ and $p \ge 1$. \\
	\end{rem}
	
	\begin{proof}
		We focus on the $N = 2$ case and the general case follows from induction on $N$. \\
		
		If $N = 2$, the statement is $a_1^p + a_2^p \le (a_1 + a_2)^p$. Since we always have equality if $a_1 = 0$ or $a_2 = 0$, we will assume that $a_1, a_2 > 0$. Then, this is equivalent to the statement \[ \left( \frac{a_1}{a_2} \right)^p + 1 \le \left( \frac{a_1}{a_2} + 1 \right)^p. \] In other words, we have reduced the problem to showing that \[ x^p + 1 \le (x + 1)^p \] if $x > 0$ and $p \ge 1$. Equivalently, we want to show that \[ (x + 1)^p - x^p \ge 1 \] if $x > 0$ and $p \ge 1$. \\
		
		Note that we have equality when $p = 1$. The idea is that this difference will increase as $p$ increases since higher powers of positive numbers increase more quickly. Fixing $x$ and taking the derivative with respect to $p$, we have
		
		\begin{align*}
			\frac{d}{dp} ((x + 1)^p - x^p) &= (x + 1)^p \log(x + 1) - x^p \log x \\
			&> 0
		\end{align*}
		
		since $(x + 1)^p > x^p$ and $\log(x + 1) > \log x$ if $x > 0$ and $p > 0$. This implies that $(x + 1)^p - x^p \ge 1$ for all $p \ge 1$ if $x > 0$ and we are done with the $N = 2$ case. \\
		
		For general $N$, we use induction on $N$. Suppose that \[ \sum_{k = 1}^{N - 1} a_k^p \le \left( \sum_{k = 1}^{N - 1} a_k \right)^p. \]
		
		This implies that
		
		\begin{align*}
			\sum_{k = 1}^N a_k^p &=  \sum_{k = 1}^{N - 1} a_k^p + a_N^p \\
			&\le \left( \sum_{k = 1}^{N - 1} a_k \right)^p + a_N^p \\
			&\le \left( \sum_{k = 1}^{N - 1} a_k + a_N \right)^p \\
			&= \left( \sum_{k = 1}^N a_k  \right)^p,
		\end{align*}
		
		where the first inequality follows from the induction assumption and the second one follows from the $N = 2$ case. 
		
	\end{proof}

	This leads us to the global version of Proposition \ref{glocbounds}.
	
	\begin{thm} \textbf{(Link condition-induced global lower bounds and sandwiching) \\} \label{gglobbounds}
		Recall from Corollary \ref{globgintlac} that \[ f_1(\Delta) = \binom{d + 1}{2} + d g_1(\Delta) + g_2(\Delta). \]
		
		if $\Delta$ is a $(d - 1)$-dimensional simplicial complex. \\
		
		Suppose that $\Delta$ is a simplicial sphere satisfying the link property. Then, we have

		\begin{equation*}
		\begin{gathered}
			2 \left( \binom{d + 1}{2} + d g_1(\Delta) + g_2(\Delta) \right) g_{i + 1}(\Delta) \\ 
			- 2 \left[ \left( \binom{d + 1}{2} + d g_1(\Delta) + g_2(\Delta) \right) g_i(\Delta) - \frac{1}{2} [i(i + 1) g_{i + 1}(\Delta) + 2i(d - i + 1) g_i(\Delta) + (d - i + 1)(d - i + 2) g_{i - 1}(\Delta)] \right]^{\frac{i + 1}{i}} \\
			\le (i + 1)(i + 2) g_{i + 2}(\Delta) + 2(i + 1)(d - i) g_{i + 1}(\Delta) + (d - i)(d - i + 1) g_i(\Delta)  \\
			\le 2 \left[ \frac{i(i + 1) g_{i + 1}(\Delta) + 2i(d - i + 1) g_i(\Delta) + (d - i + 1)(d - i + 2) g_{i - 1}(\Delta) }{2} \right]^{\frac{i}{i - 1}}.
		\end{gathered}
		\end{equation*}

		Equivalently, we have 
		
		\begin{equation*}
			\begin{gathered}
				2 \left( \binom{d + 1}{2} + d g_1(\Delta) + g_2(\Delta) \right) g_{i + 1}(\Delta) \\
				- 2 \left[ \left( \binom{d + 1}{2} - i(d - i + 1) + d g_1(\Delta) + g_2(\Delta) \right) g_i(\Delta) - \frac{i(i + 1)}{2} g_{i + 1}(\Delta) - \frac{(d - i + 1)(d - i + 2)}{2} g_{i - 1}(\Delta) \right]^{\frac{i + 1}{i}} \\
				\le (i + 1)(i + 2) g_{i + 2}(\Delta) + 2(i + 1)(d - i) g_{i + 1}(\Delta) + (d - i)(d - i + 1) g_i(\Delta)  \\
				\le \frac{ [i(i + 1) g_{i + 1}(\Delta) + 2i(d - i + 1) g_i(\Delta) + (d - i + 1)(d - i + 2) g_{i - 1}(\Delta)]^{\frac{i}{i - 1}} }{2^{\frac{1}{i - 1}}}.
			\end{gathered}
		\end{equation*}

		Note that the upper bounds above hold for any piecewise linear (PL) manifold (p. 3 of \cite{KN}, p. 26 of \cite{Hud}) satisfying the link condition. \\
		
	\end{thm}

	\begin{rem} \textbf{(Asymptotic properties) \\} \label{asympcom}
		For large $g$-vector entries, indices $i$, and dimensions $d - 1$, Corollary \ref{macpowasymp} and upper and lower bounds for $N!$ by constant multiples of $N^{N + \frac{1}{2}} e^{-N}$ for arbitrary $N$ from p. 287 of \cite{ElM1} indicate that we would expect the (equivalent) inequalities 
		
		\begin{equation*}
			\begin{gathered}
				2 \left( \binom{d + 1}{2} + d g_1(\Delta) + g_2(\Delta) \right) g_{i + 1}(\Delta) \\ 
				- \frac{2}{e} \left[ \left( \binom{d + 1}{2} + d g_1(\Delta) + g_2(\Delta) \right) g_i(\Delta) - \frac{1}{2} [i(i + 1) g_{i + 1}(\Delta) + 2i(d - i + 1) g_i(\Delta) + (d - i + 1)(d - i + 2) g_{i - 1}(\Delta)] \right]^{\frac{i + 1}{i}} \\
				\le (i + 1)(i + 2) g_{i + 2}(\Delta) + 2(i + 1)(d - i) g_{i + 1}(\Delta) + (d - i)(d - i + 1) g_i(\Delta)  \\
				\le \frac{2}{e} \left[ \frac{i(i + 1) g_{i + 1}(\Delta) + 2i(d - i + 1) g_i(\Delta) + (d - i + 1)(d - i + 2) g_{i - 1}(\Delta) }{2} \right]^{\frac{i}{i - 1}}
			\end{gathered}
		\end{equation*} 
		
		and 
		
		\begin{equation*}
			\begin{gathered}
				2 \left( \binom{d + 1}{2} + d g_1(\Delta) + g_2(\Delta) \right) g_{i + 1}(\Delta) \\
				- \frac{2}{e} \left[ \left( \binom{d + 1}{2} - i(d - i + 1) + d g_1(\Delta) + g_2(\Delta) \right) g_i(\Delta) - \frac{i(i + 1)}{2} g_{i + 1}(\Delta) - \frac{(d - i + 1)(d - i + 2)}{2} g_{i - 1}(\Delta) \right]^{\frac{i + 1}{i}} \\
				\le (i + 1)(i + 2) g_{i + 2}(\Delta) + 2(i + 1)(d - i) g_{i + 1}(\Delta) + (d - i)(d - i + 1) g_i(\Delta)  \\
				\le \frac{1}{e} \frac{ [i(i + 1) g_{i + 1}(\Delta) + 2i(d - i + 1) g_i(\Delta) + (d - i + 1)(d - i + 2) g_{i - 1}(\Delta)]^{\frac{i}{i - 1}} }{2^{\frac{1}{i - 1}}}.
			\end{gathered}
		\end{equation*}

	\end{rem}
	
	\begin{proof}
		We will add up local bounds \[ g_{i + 1}(\Delta) - (g_i(\Delta) - g_{i - 1}(\lk_\Delta(e)))^{\langle i \rangle} \le g_i(\lk_\Delta(e)) \le g_{i - 1}(\lk_\Delta(e))^{\langle i - 1 \rangle} \] from Proposition \ref{glocbounds} and apply estimates on $a^{\langle k \rangle}$ from Corollary \ref{macpowasymp}. \\
		
		After adding terms from each edge $e \in \Delta$ and multiplying by 2, Lemma \ref{glocedgesum} implies that 
		
		\begin{align*}
			2 \sum_{ \substack{e \in \Delta \\ |e| = 2} } [  g_{i + 1}(\Delta) - (g_i(\Delta) - g_{i - 1}(\lk_\Delta(e)))^{\langle i \rangle} ] &\le (i + 1)(i + 2) g_{i + 2}(\Delta) + 2(i + 1)(d - i) g_{i + 1}(\Delta) + (d - i)(d - i + 1) g_i(\Delta) \\ 
			&\le 2 \sum_{ \substack{e \in \Delta \\ |e| = 2} } g_{i - 1}(\lk_\Delta(e))^{\langle i - 1 \rangle}.
		\end{align*}
		
		The stated inequalities will be obtained by giving a lower bound for the lower bound and an upper bound for the upper bound. \\
		
		We will start by giving an upper bound for the upper bound. By Corollary \ref{macpowasymp}, we have that \[ 2 \sum_{ \substack{e \in \Delta \\ |e| = 2} } g_{i - 1}(\lk_\Delta(e))^{\langle i - 1 \rangle} \le 2 \sum_{ \substack{e \in \Delta \\ |e| = 2} } g_{i - 1}(\lk_\Delta(e))^{\frac{i}{i - 1}}. \]
		
		Since $\frac{i}{i - 1} = 1 + \frac{1}{i - 1} > 1$, Lemma \ref{sumpowerbd} implies that \[ 2 \sum_{ \substack{e \in \Delta \\ |e| = 2} } g_{i - 1}(\lk_\Delta(e))^{\frac{i}{i - 1}} \le 2 \left( \sum_{ \substack{e \in \Delta \\ |e| = 2} } g_{i - 1}(\lk_\Delta(e)) \right)^{\frac{i}{i - 1}}. \]
		
		Then, we can substitute the local-global formula from Lemma \ref{glocedgesum} to find that 
		
		\begin{align*}
			2 \sum_{ \substack{e \in \Delta \\ |e| = 2} } g_{i - 1}(\lk_\Delta(e))^{\langle i - 1 \rangle} &\le 2 \left[ \frac{i(i + 1) g_{i + 1}(\Delta) + 2i(d - i + 1) g_i(\Delta) + (d - i + 1)(d - i + 2) g_{i - 1}(\Delta)}{2} \right]^{\frac{i}{i - 1}} \\
			&= \frac{ [ i(i + 1) g_{i + 1}(\Delta) + 2i(d - i + 1) g_i(\Delta) + (d - i + 1)(d - i + 2) g_{i - 1}(\Delta) ]^{\frac{i}{i - 1}} }{2^{\frac{1}{i - 1}}}.
		\end{align*}

		Next, we give a lower bound for the lower bound \[ 2 \sum_{ \substack{e \in \Delta \\ |e| = 2} } [  g_{i + 1}(\Delta) - (g_i(\Delta) - g_{i - 1}(\lk_\Delta(e)))^{\langle i \rangle} ]. \]
		
		Since $a^{\langle k \rangle} \le a^{\frac{k + 1}{k}}$ by Corollary \ref{macpowasymp}, we have \[ 2 \sum_{ \substack{e \in \Delta \\ |e| = 2} } [  g_{i + 1}(\Delta) - (g_i(\Delta) - g_{i - 1}(\lk_\Delta(e)))^{\langle i \rangle} ] \ge 2 \sum_{ \substack{e \in \Delta \\ |e| = 2} } [ g_{i + 1}(\Delta) - (g_i(\Delta) - g_{i - 1}(\lk_\Delta(e)))^{\frac{i + 1}{i}} ].  \]
		
		After rearranging terms and substituting \[ f_1(\Delta) = \binom{d + 1}{2} + d g_1(\Delta) + g_2(\Delta), \] we have
		
		\begin{align*}
			2 \sum_{ \substack{e \in \Delta \\ |e| = 2} } [  g_{i + 1}(\Delta) - (g_i(\Delta) - g_{i - 1}(\lk_\Delta(e)))^{\langle i \rangle} ] &\ge 2 f_1(\Delta) g_{i + 1}(\Delta) - 2 \sum_{ \substack{e \in \Delta \\ |e| = 2} } (g_i(\Delta) - g_{i - 1}(\lk_\Delta(e)))^{\frac{i + 1}{i}} \\
			&= 2 \left( \binom{d + 1}{2} + d g_1(\Delta) + g_2(\Delta) \right) g_{i + 1}(\Delta) \\
			&- 2 \sum_{ \substack{e \in \Delta \\ |e| = 2} } (g_i(\Delta) - g_{i - 1}(\lk_\Delta(e)))^{\frac{i + 1}{i}}. 
		\end{align*}
		 
		Since $\frac{i + 1}{i} = 1 + \frac{1}{i} > 1$, Lemma \ref{sumpowerbd} implies that 
		
		\begin{equation*}
			\begin{gathered}
			2 \left( \binom{d + 1}{2} + d g_1(\Delta) + g_2(\Delta) \right) g_{i + 1}(\Delta) - 2 \sum_{ \substack{e \in \Delta \\ |e| = 2} } (g_i(\Delta) - g_{i - 1}(\lk_\Delta(e)))^{\frac{i + 1}{i}} \\ 
			\ge 2 \left( \binom{d + 1}{2} + d g_1(\Delta) + g_2(\Delta) \right) g_{i + 1}(\Delta) - 2 \left( \sum_{ \substack{e \in \Delta \\ |e| = 2} } (g_i(\Delta) - g_{i - 1}(\lk_\Delta(e))) \right)^{\frac{i + 1}{i}} \\
			= 2 \left( \binom{d + 1}{2} + d g_1(\Delta) + g_2(\Delta) \right) g_{i + 1}(\Delta) - 2 \left[ \left( \binom{d + 1}{2} + d g_1(\Delta) + g_2(\Delta) \right) g_i(\Delta) - \sum_{ \substack{e \in \Delta \\ |e| = 2} } g_{i - 1}(\lk_\Delta(e)) \right]^{\frac{i + 1}{i}}.
			\end{gathered} 
		\end{equation*}
		
		Substituting in the local-global formula from Lemma \ref{glocedgesum}, we have 
		
		\begin{equation*}
			\begin{gathered}
			2 \sum_{ \substack{e \in \Delta \\ |e| = 2} } [  g_{i + 1}(\Delta) - (g_i(\Delta) - g_{i - 1}(\lk_\Delta(e)))^{\langle i \rangle} ] \\ 
			\ge 2 \left( \binom{d + 1}{2} + d g_1(\Delta) + g_2(\Delta) \right) g_{i + 1}(\Delta) \\
			-  2 \left[ \left( \binom{d + 1}{2} + d g_1(\Delta) + g_2(\Delta) \right) g_i(\Delta) - \frac{1}{2} [ i(i + 1) g_{i + 1}(\Delta) + 2i(d - i + 1) g_i(\Delta) + (d - i + 1)(d - i + 2) g_{i - 1}(\Delta) ] \right]^{\frac{i + 1}{i}} \\
			= 2 \left( \binom{d + 1}{2} + d g_1(\Delta) + g_2(\Delta) \right) g_{i + 1}(\Delta) \\
			- 2 \left[ \left( \binom{d + 1}{2} - i(d - i + 1) + d g_1(\Delta) + g_2(\Delta) \right) g_i(\Delta) - \frac{i(i + 1)}{2} g_{i + 1}(\Delta) - \frac{(d - i + 1)(d - i + 2)}{2} g_{i - 1}(\Delta) ] \right]^{\frac{i + 1}{i}} .
			\end{gathered} 
		\end{equation*}
		 
	\end{proof}

	In some sense, the main role of the $M$-vector condition on the $g$-vector of the contractions $\widetilde{\Delta}_e$ of the edges $e \in \Delta$ is in the \emph{lower} bound  
	
	\begin{equation*}
		\begin{gathered}
			2 \left( \binom{d + 1}{2} + d g_1(\Delta) + g_2(\Delta) \right) g_{i + 1}(\Delta) \\ 
			- 2 \left[ \left( \binom{d + 1}{2} + d g_1(\Delta) + g_2(\Delta) \right) g_i(\Delta) - \frac{1}{2} [i(i + 1) g_{i + 1}(\Delta) + 2i(d - i + 1) g_i(\Delta) + (d - i + 1)(d - i + 2) g_{i - 1}(\Delta)] \right]^{\frac{i + 1}{i}} \\
			\le (i + 1)(i + 2) g_{i + 2}(\Delta) + 2(i + 1)(d - i) g_{i + 1}(\Delta) + (d - i)(d - i + 1) g_i(\Delta)
		\end{gathered}
	\end{equation*}
	
	of $(i + 1)(i + 2) g_{i + 2}(\Delta) + 2(i + 1)(d - i) g_{i + 1}(\Delta) + (d - i)(d - i + 1) g_i(\Delta)$ from Theorem \ref{gglobbounds}. This is because the upper bound in Theorem \ref{gglobbounds} follows from the condition of $\lk_\Delta(e)$ being a simplicial sphere (p. 26 of \cite{Hud}) and the local-global $g$-vector formula (Lemma \ref{glocedgesum}), resulting in an upper bound that holds for any PL manifold satisfying the link condition. \\

	Rewriting this inequality as 
	
	\begin{equation*}
		\begin{gathered}
			2 \left( \binom{d + 1}{2} - (i + 1)(d - i) + d g_1(\Delta) + g_2(\Delta) \right) g_{i + 1}(\Delta) \\ 
			- 2 \left[ \left( \binom{d + 1}{2} + d g_1(\Delta) + g_2(\Delta) \right) g_i(\Delta) - \frac{1}{2} [i(i + 1) g_{i + 1}(\Delta) + 2i(d - i + 1) g_i(\Delta) + (d - i + 1)(d - i + 2) g_{i - 1}(\Delta)] \right]^{\frac{i + 1}{i}} \\
			\le (i + 1)(i + 2) g_{i + 2}(\Delta) + (d - i)(d - i + 1) g_i(\Delta),
		\end{gathered}
	\end{equation*}
	
	we note that the terms look similar to the inequality
	
	\[ 2 \left( \binom{d + 1}{2} - (i + 1)(d - i) + d g_1(\Delta) + g_2(\Delta) \right) g_{i + 1}(\Delta) \ge (i + 1)(i + 2) g_{i + 2}(\Delta) + (d - k)(d - i + 1) g_i(\Delta) \]
	
	from Corollary \ref{globgintlac}, which is a consequence of the assumption $g(\widetilde{\Delta}_e) \ge 0$. \\
	
	This indicates that the term \[ 2 \left[ \left( \binom{d + 1}{2} + d g_1(\Delta) + g_2(\Delta) \right) g_i(\Delta) - \frac{1}{2} [i(i + 1) g_{i + 1}(\Delta) + 2i(d - i + 1) g_i(\Delta) + (d - i + 1)(d - i + 2) g_{i - 1}(\Delta)] \right]^{\frac{i + 1}{i}} \] is a sort of ``cushion'' giving an upper bound on the difference that is induced by an index change (replacing $i$ by $i - 1$). In other words, we have the following: \\
	
	\begin{cor} \textbf{(Global comparison with $g(\widetilde{\Delta}_e) \ge 0$-induced bounds) \\} \label{diffcush}
		Let $\Delta$ be a $(d - 1)$-dimensional simplicial complex. If the contraction $\widetilde{\Delta}_e$ of every edge $e \in \Delta$ is a $(d - 1)$-dimensional simplicial sphere (e.g. $\Delta$ a simplicial sphere satisfying the link condition), then

		\begin{equation*} 
			\begin{gathered}
				0 \le 2 \left( \binom{d + 1}{2} - (i + 1)(d - i) + d g_1(\Delta) + g_2(\Delta) \right) g_{i + 1}(\Delta) - (i + 1)(i + 2) g_{i + 2}(\Delta) - (d - i)(d - i + 1) g_i(\Delta) \\
				\le 2 \left[ \left( \binom{d + 1}{2} + d g_1(\Delta) + g_2(\Delta) \right) g_i(\Delta) - \frac{1}{2} [i(i + 1) g_{i + 1}(\Delta) + 2i(d - i + 1) g_i(\Delta) + (d - i + 1)(d - i + 2) g_{i - 1}(\Delta)] \right]^{\frac{i + 1}{i}}.
			\end{gathered}
		\end{equation*}

		Asymptotically, Remark \ref{asympcom} suggests that we can multiply the bracketed term by $\frac{2}{e}$ instead of $2$ in the upper bound.  \\
	\end{cor}

	\subsubsection{Behavior in large dimensions} \label{asymplink}
	
	Since $g_\ell(\Delta) \ne 0$ in our setting (Proposition \ref{interlacingnontriv}), we will rewrite the inequality above in terms of ratios of consecutive terms $\frac{g_{\ell + 1}(\Delta)}{g_\ell(\Delta)}$. We will also ``normalize'' terms by dividing by \[ f_1(\Delta) = \binom{d + 1}{2} + d g_1(\Delta) + g_2(\Delta). \] In order to make some simplifications, we will first record some notation. \\
	
	\begin{defn} \label{shortglob}
		Given a pure $(d - 1)$-dimensional simplicial complex $\Delta$ satisfying the link condition, we will define
		\[ \mathcal{C}_k \coloneq \sum_{ \substack{ e \in \Delta \\ |e| = 2 } } g_k(\lk_\Delta(e)) = \frac{1}{2} [(k + 1)(k + 2) g_{k + 2}(\Delta) + 2(k + 1)(d - k) g_{k + 1}(\Delta) + (d - k)(d - k + 1) g_k(\Delta) ], \]
		
		which is the global/averaged version of the local $g$-vector terms as from Lemma \ref{glocedgesum}. \\
	\end{defn}
	
	After dividing by 2 on both sides, the second inequality in Corollary \ref{diffcush} from the link condition can be rewritten as
	
	\begin{align*}
		f_1(\Delta) g_{k + 1}(\Delta) - \mathcal{C}_k &\le (f_1(\Delta) g_k(\Delta) - \mathcal{C}_{k - 1})^{ \frac{k + 1}{k} } \\
		\Longleftrightarrow g_{k + 1}(\Delta) - \frac{\mathcal{C}_k}{f_1(\Delta)} &\le f_1(\Delta)^{ \frac{1}{k} } \left( g_k(\Delta) - \frac{\mathcal{C}_{k - 1}}{f_1(\Delta)} \right)^{\frac{k + 1}{k}} \\
		\Longleftrightarrow g_{k + 1}(\Delta) \left( 1 - \frac{1}{f_1(\Delta)} \frac{\mathcal{C}_k}{g_{k + 1}(\Delta)} \right) &\le f_1(\Delta)^{ \frac{1}{k} } g_k(\Delta)^{ \frac{k + 1}{k} } \left( 1 - \frac{1}{f_1(\Delta)} \frac{\mathcal{C}_{k - 1}}{g_k(\Delta)} \right)^{\frac{k + 1}{k}} \\
		\Longleftrightarrow \frac{g_{k + 1}(\Delta)}{g_k(\Delta)^{ \frac{k + 1}{k} }} &\le f_1(\Delta)^{ \frac{1}{k} } \frac{ \left( 1 - \frac{1}{f_1(\Delta)} \frac{\mathcal{C}_{k - 1}}{g_k(\Delta)} \right)^{\frac{k + 1}{k}} }{ \left( 1 - \frac{1}{f_1(\Delta)} \frac{\mathcal{C}_k}{g_{k + 1}(\Delta)} \right) }. 
	\end{align*}
	
	We will take a closer look at the terms \[ \frac{\mathcal{C}_k}{g_{k + 1}(\Delta)} = \frac{1}{2} \left( (k + 1)(k + 2) \frac{g_{k + 2}(\Delta)}{g_{k + 1}(\Delta)} + 2(k + 1)(d - k) + (d - k)(d - k + 1) \frac{g_k(\Delta)}{g_{k + 1}(\Delta)}  \right) \]
	
	for $0 \le k \le \frac{d}{2} - 1$. \\
	
	Since $\Delta$ is a simplicial sphere in our setting, the $M$-vector condition on $g(\Delta)$ (Part 2c) of Theorem \ref{gconjsphere}) and Corollary \ref{macpowasymp} imply that $g_{k + 2}(\Delta) \le g_{k + 1}(\Delta)^{\langle k + 1 \rangle} \le g_{k + 1}(\Delta)^{ \frac{k + 2}{k + 1} }$. This means that \[ \frac{g_{k + 2}(\Delta)}{g_{k + 1}(\Delta)} \le g_{k + 1}(\Delta)^{ \frac{1}{k + 1} } \le g_1(\Delta), \] where the second inequality follows from Corollary \ref{sphereginitbd}. If $\gamma(\Delta) \ge 0$, then Corollary \ref{gamgratrest} implies that \[ \frac{g_{k + 2}(\Delta)}{g_{k + 1}(\Delta)} \ge \frac{1}{3} \] and \[ \frac{g_k(\Delta)}{g_{k + 1}(\Delta)} \le 3. \]
	
	As for the ``coefficient'' quadratic terms from $0 \le k \le \frac{d}{2} - 1$, we have the following bounds:
	
	\begin{itemize}
		\item $(k + 1)(k + 2) \le \left( \frac{d}{2} + 1 \right) \left( \frac{d}{2} + 2 \right)$
		
		\item $2(k + 1)(d - k) \le d \left( \frac{d}{2} + 1 \right)$
		
		\item $(d - k)(d - k + 1) \le d(d - 1)$
	\end{itemize}
	
	Similar properties hold for $\frac{\mathcal{C}_{k - 1}}{g_k(\Delta)}$. \\
	
	We will now consider the ``normalizations'' of $\frac{\mathcal{C}_k}{g_{k + 1}(\Delta)}$ and $\frac{\mathcal{C}_{k - 1}}{g_k(\Delta)}$ from division by \[ f_1(\Delta) = \binom{d + 1}{2} + d g_1(\Delta) + g_2(\Delta). \]

	In order to study asymptotic properties, we will review some notation. \\

	\begin{defn} (based on p. 54, p. 56, and second alternative from p. 60 of \cite{CLRS}) \label{asympdef} \\
		Given two functions $F : \mathbb{R}_{\ge 0} \longrightarrow \mathbb{R}$ and $G : \mathbb{R} \longrightarrow \mathbb{R}$ such that $g(x) > 0$ for $x \gg 0$, we define the following notation: 
		
		\begin{enumerate}
			\item $F = o(G(x))$ if $\lim_{x \to +\infty} \frac{F(x)}{G(x)} = 0$. \\
			
			\item $F = \Theta(G(x))$ if $F(x)$ is bounded above and below by strictly positive multiples of $G(x)$ if $x \gg 0$. \\
			
			\item $F = O(G(x))$ if $F(x)$ is bounded above by a strictly positive multiple of $G(x)$ if $x \gg 0$. \\
			
			\item $F(x)$ is \textbf{sublinear} with respect to $x$ if $F(x) = o(x)$. Similarly, $F(x)$ is \textbf{superlinear} with respect to $x$ if $x = o(F(x))$. \\
		\end{enumerate}
	\end{defn}

	Suppose that $d = o(g_1(\Delta))$ or $d^2 = o(g_2(\Delta))$. Note that the latter is a stronger condition than the former since $g_2(\Delta) \le g_1(\Delta)^{\langle 1 \rangle} = \binom{g_1(\Delta) + 1}{2}$. \\
	
	If \[ \frac{ g_{k + 2}(\Delta) }{ g_{k + 1}(\Delta) } = o \left( \frac{f_1(\Delta)}{d^2} \right) \] and \[ \frac{ g_{k + 1}(\Delta) }{ g_k(\Delta) } = o \left( \frac{f_1(\Delta)}{d^2} \right) \] with respect to $d$, then \[ \lim_{d \to +\infty} \frac{1}{f_1(\Delta)} \frac{\mathcal{C}_k}{g_{k + 1}(\Delta)} = 0 \] and \[ \lim_{d \to +\infty} \frac{1}{f_1(\Delta)} \frac{\mathcal{C}_{k - 1}}{g_k(\Delta)} = 0. \]
	
	This can be summarized as follows: \\
	
	\begin{thm} \textbf{(Large parameters and lack of nontriviality) \\} \label{largelinktriv}
		Let $\Delta$ be a $(d - 1)$-dimensional simplicial sphere satisfying the link condition such that $\gamma(\Delta) \ge 0$. Suppose that $d = o(g_1(\Delta))$ with respect to $d$ (e.g. when $d^2 = o(g_2(\Delta))$). Fix $1 \le k \le \frac{d}{2} - 2$. \\
		
		\begin{enumerate}
			\item We have \[ \lim_{d \to +\infty} \frac{1}{f_1(\Delta)} \frac{\mathcal{C}_k}{g_{k + 1}(\Delta)} = \lim_{d \to +\infty} \frac{(k + 1)(k + 2) \frac{g_{k + 2}(\Delta)}{g_{k + 1}(\Delta)} }{f_1(\Delta)}. \]
			
			\item If \[ \frac{ g_{k + 1}(\Delta) }{ g_k(\Delta) } = o \left( \frac{f_1(\Delta)}{d^2} \right) \] with respect to $d$, then global/averaged version of the $M$-vector inequalities $g_{k + 1}(\widetilde{\Delta}_e) \le g_k(\widetilde{\Delta}_e)^{\langle k \rangle}$ (Corollary \ref{diffcush}) is already implied by the $M$-vector inequalities for $\Delta$ for $d \gg 0$. \\
		\end{enumerate}
		
	\end{thm}

	\begin{proof}
		\begin{enumerate}
			\item This is implied by our discussion above since the remaining terms of $\frac{\mathcal{C}_k}{g_{k + 1}(\Delta)}$ are either quadratic polynomials in $d$ or bounded above and below by constant multiples of them. Note that $2k(d - k + 1)$ with $0 \le k \le \frac{d}{2}$ is also bounded above by a constant multiple of a quadratic polynomial in $d$ when $k$ is a sublinear function in $d$ (where we can also phrase a more precise version). The comparison between the examples follows from $g_2(\Delta) \le g_1(\Delta)^{\langle 1 \rangle} = \binom{g_1(\Delta) + 1}{2}$. \\
			
			\item Recall that the global/averaged inequality under consideration is \[ \frac{g_{k + 1}(\Delta)}{g_k(\Delta)^{ \frac{k + 1}{k} }} \le f_1(\Delta)^{ \frac{1}{k} } \frac{ \left( 1 - \frac{1}{f_1(\Delta)} \frac{\mathcal{C}_{k - 1}}{g_k(\Delta)} \right)^{\frac{k + 1}{k}} }{ \left( 1 - \frac{1}{f_1(\Delta)} \frac{\mathcal{C}_k}{g_{k + 1}(\Delta)} \right) }. \]
			
			If \[ \frac{ g_{k + 2}(\Delta) }{ g_{k + 1}(\Delta) } = o \left( \frac{f_1(\Delta)}{d^2} \right) \] and \[ \frac{ g_{k + 1}(\Delta) }{ g_k(\Delta) } = o \left( \frac{f_1(\Delta)}{d^2} \right) \] with respect to $d$, then \[ \lim_{d \to +\infty} \frac{ \left( 1 - \frac{1}{f_1(\Delta)} \frac{\mathcal{C}_{k - 1}}{g_k(\Delta)} \right)^{\frac{k + 1}{k}} }{ \left( 1 - \frac{1}{f_1(\Delta)} \frac{\mathcal{C}_k}{g_{k + 1}(\Delta)} \right) } = 1 \] and the inequality gives \[ \frac{g_{k + 1}(\Delta)}{g_k(\Delta)^{ \frac{k + 1}{k} }} \le f_1(\Delta)^{\frac{1}{k}} \] for $d \gg 0$. Note that this holds both when $0 \le k \le \frac{d}{2}$ is a linear or sublinear function in $d$. However, the $M$-vector inequalities for $\Delta$ (Part 2c) of Theorem \ref{gconjsphere}) already imply that \[ g_{k + 1}(\Delta) \le g_k(\Delta)^{\langle k \rangle} \le g_k(\Delta)^{\frac{k + 1}{k}} \Longrightarrow \frac{g_{k + 1}(\Delta)}{g_k(\Delta)^{\frac{k + 1}{k}}} \le 1 \le f_1(\Delta)^{ \frac{1}{k} }. \]
			
			If we only assume that \[  \frac{ g_{k + 1}(\Delta) }{ g_k(\Delta) } = o \left( \frac{f_1(\Delta)}{d^2} \right) \] with respect to $d$, then the denominator may \emph{have a value in} $(0, 1)$ as $d \to +\infty$ while the numerator of the fraction in the upper bound is 1. However, decreasing the denominator while increasing the numerator would give rise to an even larger upper bound which is ``more trivial'' than the one above. Thus, it suffices to only assume that the asymptotic upper bound for $\frac{ g_{k + 1}(\Delta) }{ g_k(\Delta) }$.
			
		\end{enumerate}
	\end{proof}

	\begin{rem} \textbf{(Interpretation of effect of initial conditions) \\} \label{largeparamint}
		\begin{enumerate}
			\item \textbf{((Non)triviality of link condition vs. $\gamma(\Delta) \ge 0$) \\} 
			Heuristically, we can take Part 2 of Theorem \ref{largelinktriv} to mean that nonnegativity of the gamma vector $\gamma(\Delta) \ge 0$ on high-dimensional simplicial spheres $\Delta$ is not affected by the link condition if $g_1(\Delta)$ is very large compared to the rate of growth of the $g$-vector components. In the setting of simplicial spheres PL homeomorphic to the boundary of a cross polytope, this can be taken to mean having a large ``net number of edge subdivisions'' compared to $\dim \Delta = d - 1$. If $g_1(\Delta)$ is asymptotically significantly larger than $d$ (e.g. exponential in $d$ or even larger as in Remark \ref{matcoeffsize} in the general realizability context), the asymptotic upper bound is not very far from $f_1(\Delta)$ or $g_1(\Delta)$ itself. Note that the condition is determined by the size of $\gamma_{k + 1}(\Delta)$ if $\frac{g_{k + 1}(\Delta)}{g_k(\Delta)} \ge d + 1$. \\

			\item \textbf{(Comparison with $f$-vectors) \\} 
			Given that gamma vectors of barycentric subdivisions of homology spheres are equal to $f$-vectors of auxiliary simplicial complexes (Theorem 1.1 on p. 1365 and Corollary 5.6 on p. 1376 of \cite{NPT}) and barycentric subdivisions can be obtained via repeated edge subdivisions (Proposition 3.1 on p. 73 of \cite{LN}), we may compare the relationship between $\gamma(\Delta) \ge 0$ and realizability by $f$-vectors in the ``large parameter'' setting of Theorem \ref{largelinktriv}. More specifically, we can consider whether there is a closer relationship between the two conditions than indicated in Corollary \ref{posgamfcomp}. \\
		\end{enumerate}
	\end{rem}

	Recall that the inequality \[ \frac{g_{k + 1}(\Delta)}{g_k(\Delta)^{ \frac{k + 1}{k} }} \le f_1(\Delta)^{ \frac{1}{k} } \frac{ \left( 1 - \frac{1}{f_1(\Delta)} \frac{\mathcal{C}_{k - 1}}{g_k(\Delta)} \right)^{\frac{k + 1}{k}} }{ \left( 1 - \frac{1}{f_1(\Delta)} \frac{\mathcal{C}_k}{g_{k + 1}(\Delta)} \right) } \] is equivalent to \[ 1 - \frac{1}{f_1(\Delta)} \frac{\mathcal{C}_k}{g_{k + 1}(\Delta)} \le f_1(\Delta)^{ \frac{1}{k} } \frac{ \left( 1 - \frac{1}{f_1(\Delta)} \frac{\mathcal{C}_{k - 1}}{g_k(\Delta)} \right)^{\frac{k + 1}{k}} }{ \frac{g_{k + 1}(\Delta)}{g_k(\Delta)^{ \frac{k + 1}{k} }} } = f_1(\Delta)^{ \frac{1}{k} } \frac{ \left( 1 - \frac{1}{f_1(\Delta)} \frac{\mathcal{C}_{k - 1}}{g_k(\Delta)} \right)^{\frac{k + 1}{k}} }{ \frac{g_{k + 1}(\Delta)}{g_k(\Delta) }} \cdot g_k(\Delta)^{ \frac{1}{k} }. \]
	
	Unwinding the definitions of $\mathcal{C}_k$ and $\mathcal{C}_{k - 1}$, the inequality above can be interpreted as a lower bound of $\frac{g_{k + 2}(\Delta)}{g_{k + 1}(\Delta)}$ in terms of $\frac{g_{k + 1}(\Delta)}{g_k(\Delta)}$, $g_k(\Delta)$, and $f_1(\Delta)$. Note that $f_1(\Delta)^{ \frac{1}{k} } \ge 1$ and $g_k(\Delta)^{ \frac{1}{k} } \ge 1$ for each $k$ since $f_1(\Delta) > \binom{d + 1}{2} > 1$ and $g_k(\Delta) \ge \binom{d}{k} - \binom{d}{k - 1} \ge 1$ for $0 \le k \le \frac{d}{2}$ if $\gamma(\Delta) \ge 0$. Given a fixed $d$, we have $\lim_{k \to +\infty} f_1(\Delta)^{ \frac{1}{k} } = 1$. Note that $f_1(\Delta)$ does \emph{not} depend on $k$. We can describe how the induced lower bound of $\frac{g_{k + 2}(\Delta)}{g_{k + 1}(\Delta)}$ increases/decreases with respect to these parameters. 
	
	\begin{itemize}
		\item Keeping $g_k(\Delta)$ and $f_1(\Delta)$ fixed, we see that increasing $\frac{g_{k + 1}(\Delta)}{g_k(\Delta)}$ decreases the numerator and increases the denominator of the upper bound of $1 - \frac{1}{f_1(\Delta)} \frac{\mathcal{C}_k}{g_{k + 1}(\Delta)}$. Since this decreases the upper bound of $1 - \frac{1}{f_1(\Delta)} \frac{\mathcal{C}_k}{g_{k + 1}(\Delta)}$, we have a larger lower bound of $\frac{g_{k + 2}(\Delta)}{g_{k + 1}(\Delta)}$.
		
		\item On the other hand, the upper bound of $1 - \frac{1}{f_1(\Delta)} \frac{\mathcal{C}_k}{g_{k + 1}(\Delta)}$ increases as we increase $f_1(\Delta)$ or $g_k(\Delta)$ while keeping $\frac{g_{k + 1}(\Delta)}{g_k(\Delta)}$ fixed. This would mean a smaller lower bound for $\frac{g_{k + 2}(\Delta)}{g_{k + 1}(\Delta)}$. \\ 
	\end{itemize}

	In general, we have \[ 0 \le \left( 1 - \frac{1}{f_1(\Delta)} \frac{\mathcal{C}_{k - 1}}{g_k(\Delta)} \right)^{\frac{k + 1}{k}} \le 1 \]
	
	since $0 \le x \le 1$ implies that $0 \le x^{ \frac{1}{k} } \le 1$ if $k > 0$ and $x > 0$. Note that $x^{ \frac{1}{k} } \ge x$ in this case. \\ 
	
	Fix $1 \le k \le \frac{d}{2} - 1$. Although \[ \frac{g_k(\Delta)^{ \frac{k + 1}{k} }}{g_{k + 1}(\Delta)} \ge 1 \] and \[ \lim_{d \to +\infty} f_1(\Delta)^{ \frac{1}{k} } = +\infty \] for each fixed $k$, it may still be possible that \[ \left( 1 - \frac{1}{f_1(\Delta)} \frac{\mathcal{C}_{k - 1}}{g_k(\Delta)} \right)^{\frac{k + 1}{k}} \] can be arbitrarily small even when $d \gg 0$ (e.g. if it approaches 0 as $d \to +\infty$). \\

	In order to have a more precise idea of the overall behavior, we will start by taking a closer look at the shape of the inequality. \\

	 \begin{prop} \label{ratlowbd}
	 	Given a $(d - 1)$-dimensional simplicial sphere $\Delta$ satisfying the link condition and $1 \le k \le \frac{d}{2} - 1$, we have 
	 	
	 	\begin{align*}
	 		\frac{1}{f_1(\Delta)} \frac{\mathcal{C}_k}{g_{k + 1}(\Delta)} &\ge \left( 1 - \omega_k^{ \frac{k}{k + 1} } \left( 1 - \frac{1}{f_1(\Delta)} \frac{\mathcal{C}_{k - 1}}{g_k(\Delta)} \right) \right)^{ \frac{k + 1}{k} } \\
	 		&= \left( \omega_k^{ \frac{k}{k + 1} } \left( \frac{1}{f_1(\Delta)} \frac{\mathcal{C}_{k - 1}}{g_k(\Delta)} \right) + \left( 1 - \omega_k^{ \frac{k}{k + 1} } \right) \right)^{ \frac{k + 1}{k} },
	 	\end{align*}
	 	
	 	where \[ \omega_k = f_1(\Delta)^{ \frac{1}{k} } \frac{g_k(\Delta)^{ \frac{k + 1}{k} }}{g_{k + 1}(\Delta)}. \]
	 	
	 \end{prop}
	 
	 \begin{proof}
	 	This follows from applying the inequality $(P - Q)^p \le P^p - Q^p$ for $P \ge Q \ge 0$ and $p \ge 1$ (Remark \ref{diffpowbd}) to the inequality \[ 1 - \frac{1}{f_1(\Delta)} \frac{\mathcal{C}_k}{g_{k + 1}(\Delta)} \le f_1(\Delta)^{ \frac{1}{k} } \frac{ \left( 1 - \frac{1}{f_1(\Delta)} \frac{\mathcal{C}_{k - 1}}{g_k(\Delta)} \right)^{\frac{k + 1}{k}} }{ \frac{g_{k + 1}(\Delta)}{g_k(\Delta)^{ \frac{k + 1}{k} }} } = f_1(\Delta)^{ \frac{1}{k} } \frac{g_k(\Delta)^{ \frac{k + 1}{k} }}{g_{k + 1}(\Delta)} \left( 1 - \frac{1}{f_1(\Delta)} \frac{\mathcal{C}_{k - 1}}{g_k(\Delta)} \right)^{\frac{k + 1}{k}}, \]
	 	
	 	which is equivalent to \[ \frac{1}{f_1(\Delta)} \frac{\mathcal{C}_k}{g_{k + 1}(\Delta)} \ge 1 - f_1(\Delta)^{ \frac{1}{k} } \frac{g_k(\Delta)^{ \frac{k + 1}{k} }}{g_{k + 1}(\Delta)} \left( 1 - \frac{1}{f_1(\Delta)} \frac{\mathcal{C}_{k - 1}}{g_k(\Delta)} \right)^{\frac{k + 1}{k}}. \]
	 \end{proof}

	 \begin{exmp} \textbf{(Simplification from a stronger bound) \\} \label{simpratlowbd}
	 	If $g(\Delta)$ satisfies the stronger bound \[ 1 - \frac{1}{f_1(\Delta)} \frac{\mathcal{C}_k}{g_{k + 1}(\Delta)} \le  \left( 1 - \frac{1}{f_1(\Delta)} \frac{\mathcal{C}_{k - 1}}{g_k(\Delta)} \right)^{\frac{k + 1}{k}} \] with the upper bound terms $f_1(\Delta)^{ \frac{1}{k} } \ge 1$ and $\frac{g_k(\Delta)^{ \frac{k + 1}{k} }}{g_{k + 1}(\Delta)} \ge 1$ replaced by 1, then the method above yields the simplified bound \[ \frac{1}{f_1(\Delta)} \frac{\mathcal{C}_k}{g_{k + 1}(\Delta)} \ge \left( \frac{1}{f_1(\Delta)} \frac{\mathcal{C}_{k - 1}}{g_k(\Delta)} \right)^{ \frac{k + 1}{k} }. \]

	 \end{exmp}

	If $\gamma(\Delta) \ge 0$, we can view the inequalities in Proposition \ref{ratlowbd} and Example \ref{simpratlowbd} asymptotically with respect to $d$ (i.e. for $d \gg 0$) as lower bounds for $\frac{g_{k + 2}(\Delta)}{g_{k + 1}(\Delta)}$ where the lower bounds increase as $\frac{g_{k + 1}(\Delta)}{g_k(\Delta)}$ increases. For example, only the terms involving $\frac{g_{k + 2}(\Delta)}{g_{k + 1}(\Delta)}$ and $\frac{g_{k + 1}(\Delta)}{g_k(\Delta)}$ remain in the upper and lower bounds respectively when $d = o(g_1(\Delta))$ (e.g. case where $d^2 = o(g_2(\Delta))$). \\

	We will now study the behavior of the terms of the inequality as $d \to +\infty$ and split into cases according to asymptotic properties of the numerator of the quotient in the upper bound. While $\lim_{d \to +\infty} \frac{1}{f_1(\Delta)} \frac{\mathcal{C}_{k - 1}}{g_k(\Delta)}$ does not necessarily exist if the upper bound term $\frac{\mathcal{C}_{k - 1}}{g_k(\Delta)}$ is an arbitrary function of $d$, we will look at instances where the limit exists in order to get some insight into what functions yield the possible output behaviors. For example, this may include functions added to form the given one or those bounding oscillating functions. The same applies for the lower bound term $\frac{1}{f_1(\Delta)} \frac{\mathcal{C}_k}{g_{k + 1}(\Delta)}$. Note that \[ 0 \le \lim_{d \to +\infty} \frac{1}{f_1(\Delta)} \frac{\mathcal{C}_{k - 1}}{g_k(\Delta)} \le 1 \] if it exists. In the remainder of this section, we will both take $1 \le k \le \frac{d}{2} - 2$ to be a constant or linear/sublinear functions in $d$ such that $\lim_{d \to +\infty} k = +\infty$. The former involves limiting behavior of \emph{individual} link condition inequalities while the latter considers those of \emph{collections} of such inequalities. \\
	
	\begin{prop} \textbf{(Initial asymptotic conditions and the link condition) \\} \label{linkinitasymp}
		Let $\Delta$ be a $(d - 1)$-dimensional simplicial sphere satisfying the link condition such that $\gamma(\Delta) \ge 0$. Consider the inequality \[ 1 - \frac{1}{f_1(\Delta)} \frac{\mathcal{C}_k}{g_{k + 1}(\Delta)} \le f_1(\Delta)^{ \frac{1}{k} } \frac{ \left( 1 - \frac{1}{f_1(\Delta)} \frac{\mathcal{C}_{k - 1}}{g_k(\Delta)} \right)^{\frac{k + 1}{k}} }{ \frac{g_{k + 1}(\Delta)}{g_k(\Delta)^{ \frac{k + 1}{k} }} } = f_1(\Delta)^{ \frac{1}{k} } \frac{g_k(\Delta)^{ \frac{k + 1}{k} }}{g_{k + 1}(\Delta)} \left( 1 - \frac{1}{f_1(\Delta)} \frac{\mathcal{C}_{k - 1}}{g_k(\Delta)} \right)^{\frac{k + 1}{k}}. \]
		
		Since $1 \le k \le \frac{d}{2} - 2$, we can take $k$ to be a linear or sublinear (e.g. constant) function in $d$. \\
		
		\begin{enumerate}
			\item If \[ \lim_{d \to +\infty} \frac{1}{f_1(\Delta)} \frac{\mathcal{C}_{k - 1}}{g_k(\Delta)} = 0, \] the induced limit inequality is trivial. This is equivalent to the last term of the product being equal to 1 if $k$ is constant or a sublinear function of $d$ such that $\lim_{d \to +\infty} k = +\infty$. \\
			
			\item Suppose that $1 \le k \le \frac{d}{2} - 1$ is a positive constant or $\lim_{d \to +\infty} k = +\infty$. If $\lim_{d \to +\infty} f_1(\Delta)^{ \frac{1}{k} } = +\infty$ (e.g. $k$ a constant or the small functions in $d$ from Proposition \ref{zerolinkcoeff}), the limit of the upper bound of the inequality as $d \to +\infty$ is equal to 0 only if \[ \lim_{d \to +\infty}  \left( 1 - \frac{1}{f_1(\Delta)} \frac{\mathcal{C}_{k - 1}}{g_k(\Delta)} \right)^{ \frac{k + 1}{k} } = 0. \]
			
			This condition is equivalent to \[ \lim_{d \to +\infty} \frac{1}{f_1(\Delta)} \frac{\mathcal{C}_{k - 1}}{g_k(\Delta)} = 1. \]
			
			The upper bound of the inequality is then equal to 0 if $\left( 1 - \frac{1}{f_1(\Delta)} \frac{\mathcal{C}_{k - 1}}{g_k(\Delta)} \right)^{ \frac{k + 1}{k} }$ approaches 0 more quickly than \[ \frac{1}{f_1(\Delta)^{ \frac{1}{k} }} \frac{g_{k + 1}(\Delta)}{g_k(\Delta)^{ \frac{k + 1}{k} }} \] as $d \to +\infty$. \\

			\item Suppose that \[ \lim_{d \to +\infty} \frac{1}{f_1(\Delta)} \frac{\mathcal{C}_{k - 1}}{g_k(\Delta)} = \mu \] for some $0 < \alpha < 1$. \\
			
			 If \[ \lim_{d \to +\infty} f_1(\Delta)^{ \frac{1}{k} } \frac{g_k(\Delta)^{ \frac{k + 1}{k} }}{g_{k + 1}(\Delta)} = +\infty \] or \[ \lim_{d \to +\infty} f_1(\Delta)^{ \frac{1}{k} } \frac{g_k(\Delta)^{ \frac{k + 1}{k} }}{g_{k + 1}(\Delta)} = \beta > \frac{1}{(1 - \mu)^{ \frac{k + 1}{k} }}, \] the induced limit inequality is trivial. \\
		\end{enumerate}
		
	\end{prop}
	
	\begin{rem} \textbf{(Trivial limit inequalities) \\} \label{trivineqasymp}
		Variants of Part 1 and Part 3 on triviality of the limit inequality hold whenever there is a strictly positive constant $C > 0$ such that \[ \left( 1 - \frac{1}{f_1(\Delta)} \frac{\mathcal{C}_{k - 1}}{g_k(\Delta)} \right)^{\frac{k + 1}{k}} > C \] for all $d \gg 0$. \\
	\end{rem}
	
	\begin{proof}
		\begin{enumerate}
			\item The assumption implies that \[ \lim_{d \to +\infty} \left( 1 - \frac{1}{f_1(\Delta)} \frac{\mathcal{C}_{k - 1}}{g_k(\Delta)} \right)^{ \frac{k + 1}{k} } = 1. \] Since $f_1(\Delta)^{\frac{1}{k}} \ge 1$ and \[ \frac{g_k(\Delta)^{ \frac{k + 1}{k} }}{g_{k + 1}(\Delta)} \ge 1 \] by Part 2c) of Theorem \ref{gconjsphere} and Corollary \ref{macpowasymp}, the upper bound of the inequality is $\ge 1$. However, the lower bound is already between 0 and 1 by Corollary \ref{globgintlac} (used in equivalent inequality Corollary \ref{diffcush}). Thus, the inequality must hold for trivial reasons. \\ 
			
			\item Since $\lim_{d \to +\infty} k = +\infty$, $\lim_{d \to +\infty} f_1(\Delta)^{ \frac{1}{k} } = +\infty$, and \[ \frac{g_k(\Delta)^{ \frac{k + 1}{k} }}{g_{k + 1}(\Delta)} \ge 1 \] as mentioned in the proof of Part 1, the limit of the upper bound is equal to 0 only if \[ \lim_{d \to +\infty} \left( 1 - \frac{1}{f_1(\Delta)} \frac{\mathcal{C}_{k - 1}}{g_k(\Delta)} \right)^{ \frac{k + 1}{k} } = 0. \]
			
			This is equivalent to \[ \lim_{d \to +\infty} \left( 1 + \frac{1}{k} \right) \log \left( 1 - \frac{1}{f_1(\Delta)} \frac{\mathcal{C}_{k - 1}}{g_k(\Delta)} \right) = -\infty. \]
			
			If $k$ is a constant, the statement follows after factoring out $1 + \frac{1}{k}$. Suppose that $\lim_{d \to +\infty} k = +\infty$. Fix $0 < \varepsilon \ll 1$. For any $N > 0$, there is a $d_N$ such that \[ \left( 1 + \frac{1}{k} \right) \log \left( 1 - \frac{1}{f_1(\Delta)} \frac{\mathcal{C}_{k - 1}}{g_k(\Delta)} \right) < -N \] and \[ - \left( 1 + \frac{1}{k} \right) \log \left( 1 - \frac{1}{f_1(\Delta)} \frac{\mathcal{C}_{k - 1}}{g_k(\Delta)} \right) > N \] for all $d \ge d_N$. Since $\lim_{d \to +\infty} k = +\infty$, we can assume without loss of generality that $1 + \frac{1}{k} < 1 + \varepsilon$. This implies that \[ -\log \left( 1 - \frac{1}{f_1(\Delta)} \frac{\mathcal{C}_{k - 1}}{g_k(\Delta)} \right) > \frac{N}{1 + \varepsilon}. \]  Given $M > 0$, using the $d_N$ for $N = (1 + \varepsilon) M$ would imply that \[ -\log \left( 1 - \frac{1}{f_1(\Delta)} \frac{\mathcal{C}_{k - 1}}{g_k(\Delta)} \right) > M \] for all $d \ge d_N$. This implies that \[ \lim_{d \to +\infty} \log \left( 1 - \frac{1}{f_1(\Delta)} \frac{\mathcal{C}_{k - 1}}{g_k(\Delta)} \right) = -\infty \] and \[ \lim_{d \to +\infty} 1 - \frac{1}{f_1(\Delta)} \frac{\mathcal{C}_{k - 1}}{g_k(\Delta)} = 0. \]
			
			\item In these cases, the limit of the upper bound as $d \to +\infty$ is either $+\infty$ or a finite number $> 1$. \\

		\end{enumerate}
	\end{proof}
	
	Given these initial asymptotic conditions, we will take a closer look at when they apply. \\
	
	\begin{cor} \textbf{(Ratios of $g$-vector components and initial conditions) \\} \label{asympgratinit}
		Suppose that $\Delta$ is a $(d - 1)$-dimensional simplicial sphere satisfying the link condition such that $\gamma(\Delta) \ge 0$. We will take $1 \le k \le \frac{d}{2} - 1$ to be asymptotically linear or sublinear (e.g. constant) in $d$. \\
		
		\begin{enumerate}
			\item Suppose that \[ \lim_{d \to +\infty} \frac{1}{f_1(\Delta)} \frac{\mathcal{C}_{k - 1}}{g_k(\Delta)} = 0 \] as in Part 1 of Proposition \ref{linkinitasymp}. \\  
				\begin{enumerate}
					\item Assume that $d = o(g_1(\Delta))$ (e.g. when $d^2 = o(g_2(\Delta))$). If $k$ is linear in $d$, then \[ \frac{g_{k + 1}(\Delta)}{g_k(\Delta)} = o \left( \frac{g_1(\Delta)}{d} \right) \] or \[ \frac{g_{k + 1}(\Delta)}{g_k(\Delta)} = o \left( \frac{g_2(\Delta)}{d^2} \right) \] with respect to $d$. Similarly, we have \[ \frac{g_{k + 1}(\Delta)}{g_k(\Delta)} = o \left( \frac{d g_1(\Delta)}{k^2} \right)  \] or \[ \frac{g_{k + 1}(\Delta)}{g_k(\Delta)} = o \left( \frac{g_2(\Delta)}{k^2} \right) \] with respect to $d$ if $k$ is a sublinear function in $d$ (e.g. a constant). \\
					
					\item Assume that $g_1(\Delta)$ is linear in $d$. Then, we have that \[ \frac{g_{k + 1}(\Delta)}{g_k(\Delta)} = o \left( \frac{d^2}{k^2} \right), \] and \[ \lim_{d \to +\infty} \frac{g_{k - 1}(\Delta)}{g_k(\Delta)} = 0. \]  For example, this holds for any constant $k$. \\ 
				\end{enumerate}
			
			\item Suppose that \[ \lim_{d \to +\infty} \frac{1}{f_1(\Delta)} \frac{\mathcal{C}_{k - 1}}{g_k(\Delta)} = 1 \] as in Part 2 of Proposition \ref{linkinitasymp}. \\
				\begin{enumerate}
					\item Assume that $d = o(g_1(\Delta))$ (e.g. when $d^2 = o(g_2(\Delta))$). Then, we have \[ \lim_{d \to +\infty} \frac{ \frac{g_{k + 1}(\Delta)}{g_k(\Delta)} }{ \frac{f_1(\Delta)}{k^2} } = 2 \] 
					
					if $\lim_{d \to +\infty} k = +\infty$ and \[ \lim_{d \to +\infty}  \frac{ \frac{g_{k + 1}(\Delta)}{g_k(\Delta)} }{f_1(\Delta)} = \frac{2}{k(k + 1)}  \] if $k$ is a constant. \\

					If $\lim_{d \to +\infty} \frac{k}{d} = \alpha$ for some $0 < \alpha \le \frac{1}{2}$, then \[ \alpha^2 \lim_{d \to +\infty} \frac{  \frac{g_{k + 1}(\Delta)}{g_k(\Delta)} }{ \frac{1}{2} + \frac{g_1(\Delta)}{d} + \frac{g_2(\Delta)}{d^2} } = 2.  \]
					
					\item Assume that $g_1(\Delta)$ is linear in $d$. If $g_1(\Delta)$ is linear in $d$ and $\lim_{d \to +\infty} \frac{k}{d} = \alpha$ for some $0 < \alpha \le \frac{1}{2}$, then \[ \alpha^2 \lim_{d \to +\infty} \frac{g_{k + 1}(\Delta)}{g_k(\Delta)} + 2\alpha(1 - \alpha) + (1 - \alpha)^2 \lim_{d \to +\infty} \frac{g_{k - 1}(\Delta)}{g_k(\Delta)} = 2 \max \left( \lim_{d \to +\infty} \frac{g_1(\Delta)}{d}, \lim_{d \to +\infty} \frac{g_2(\Delta)}{d^2} \right). \]
					
					If $k$ is a sublinear function of $d$ and $\lim_{d \to +\infty} k = +\infty$, then \[ \lim_{d \to +\infty} \frac{ \frac{g_{k + 1}(\Delta)}{g_k(\Delta)} + \frac{d^2}{k^2} \frac{g_{k - 1}(\Delta)}{g_k(\Delta)} }{ \frac{f_1(\Delta)}{k^2} } = 2  \] and \[  \lim_{d \to +\infty} \left(  \frac{ \frac{g_{k + 1}(\Delta)}{g_k(\Delta)} }{ \frac{d^2}{k^2} } +  \frac{g_{k - 1}(\Delta)}{g_k(\Delta)} \right) = 1 + 2 \lim_{d \to +\infty} \frac{g_1(\Delta)}{d} +  2 \lim_{d \to +\infty}  \frac{g_2(\Delta)}{d^2}.  \]

					If $k$ is a constant, this is impossible. \\
					
					\color{black} 
				\end{enumerate}

			\item Suppose that \[ \lim_{d \to +\infty} \frac{1}{f_1(\Delta)} \frac{\mathcal{C}_{k - 1}}{g_k(\Delta)} = \mu \] with $0 < \mu < 1$ as in Part 3 of Proposition \ref{linkinitasymp}. \\
				\begin{enumerate}
					\item Assume that $d = o(g_1(\Delta))$ (e.g. when $d^2 = o(g_2(\Delta))$). Then, we have \[ \lim_{d \to +\infty} \frac{ \frac{g_{k + 1}(\Delta)}{g_k(\Delta)} }{ \frac{f_1(\Delta)}{k^2} } = 2 \mu \] 
					
					if $\lim_{d \to +\infty} k = +\infty$ and \[ \lim_{d \to +\infty}  \frac{ \frac{g_{k + 1}(\Delta)}{g_k(\Delta)} }{f_1(\Delta)} = \frac{2 \mu}{k(k + 1)}  \] if $k$ is a constant. \\
					
					If $\lim_{d \to +\infty} \frac{k}{d} = \alpha$ for some $0 < \alpha \le \frac{1}{2}$, then \[ \alpha^2 \lim_{d \to +\infty} \frac{  \frac{g_{k + 1}(\Delta)}{g_k(\Delta)} }{ \frac{1}{2} + \frac{g_1(\Delta)}{d^2} + \frac{g_2(\Delta)}{d^2} } = 2 \mu.  \]
					
					\item Assume that $g_1(\Delta)$ is linear in $d$. If $g_1(\Delta)$ is linear in $d$ and $\lim_{d \to +\infty} \frac{k}{d} = \alpha$ for some $0 < \alpha \le \frac{1}{2}$, then \[ \alpha^2 \lim_{d \to +\infty} \frac{g_{k + 1}(\Delta)}{g_k(\Delta)} + 2\alpha(1 - \alpha) + (1 - \alpha)^2 \lim_{d \to +\infty} \frac{g_{k - 1}(\Delta)}{g_k(\Delta)} = 2 \mu \max \left( \lim_{d \to +\infty} \frac{g_1(\Delta)}{d}, \lim_{d \to +\infty} \frac{g_2(\Delta)}{d^2} \right). \]
					
					If $k$ is a sublinear function of $d$ and $\lim_{d \to +\infty} k = +\infty$, then \[ \lim_{d \to +\infty} \frac{ \frac{g_{k + 1}(\Delta)}{g_k(\Delta)} + \frac{d^2}{k^2} \frac{g_{k - 1}(\Delta)}{g_k(\Delta)} }{ \frac{f_1(\Delta)}{k^2} } = 2 \mu \] and \[  \lim_{d \to +\infty} \left(  \frac{ \frac{g_{k + 1}(\Delta)}{g_k(\Delta)} }{ \frac{d^2}{k^2} } +  \frac{g_{k - 1}(\Delta)}{g_k(\Delta)} \right) = \mu \left( 1 + 2 \lim_{d \to +\infty} \frac{g_1(\Delta)}{d} +  2 \lim_{d \to +\infty}  \frac{g_2(\Delta)}{d^2} \right).  \]

					If $k$ is a constant, this is impossible. \\ 
					
					\color{black} 
					
				\end{enumerate}
		\end{enumerate}
	\end{cor}

	\begin{rem} \textbf{(Role of the index $k$) \\} \label{indexratchange}
		If $k$ is \emph{not} a constant with respect to $d$, ratios of the form $\frac{g_{k + 1}(\Delta)}{g_k(\Delta)}$ may not increase with $d$ as quickly as expected from an increase in the ``order with respect to $d$'' (e.g. $g_1(\Delta)$ linear in $d$ case from Proposition \ref{gamposling1h1}). For example, this includes the case of boundaries of cross polytopes discussed in Part 2 of Example \ref{constexmp}. While increasing the index by 1 increases the this ``order'', the index $k$ affects the ratios between consecutive terms enough to offset this increase. This contributes to keeping $\frac{1}{f_1(\Delta)} \frac{\mathcal{C}_{k - 1}}{g_k(\Delta)}$ bounded between 0 and 1 for all $0 \le k \le \frac{d}{2} - 1$ even as $d \to +\infty$. In addition, $\gamma(\Delta) \ge 0$ implies the upper bound $\frac{g_{k - 1}(\Delta)}{g_k(\Delta)} \le 3$ which is \emph{independent} of the index $k$ (Corollary \ref{gamgratrest}). Note that $g_k(\Delta) \ne 0$ in our setting by Proposition \ref{interlacingnontriv}. \\
	\end{rem}

	\begin{proof}
		\begin{enumerate}
			\item 
				\begin{enumerate}
					\item Since $\gamma(\Delta) \ge 0$, we have that \[ 0 < \frac{g_{k - 1}(\Delta)}{g_k(\Delta)} \le 3 \] by Corollary \ref{gamgratrest} and Proposition \ref{interlacingnontriv}. This means that the only possible term that is asymptotically strictly larger than a quadratic term in $d$ is \[ k(k + 1) \frac{g_{k + 1}(\Delta)}{g_k(\Delta)}. \] If $k$ is a constant, we factor out $k(k + 1)$ and focus on the $\frac{g_{k + 1}(\Delta)}{g_k(\Delta)}$ term to obtain the stated limit. Suppose that $\lim_{d \to +\infty} k = +\infty$. Then, we have \[ \frac{g_{k + 1}(\Delta)}{g_k(\Delta)} = o \left( \frac{f_1(\Delta)}{k^2} \right). \] 
					
					Since $d = o(g_1(\Delta))$, the terms $d g_1(\Delta)$ is asymptotically strictly larger than $\binom{d + 1}{2}$, which is a quadratic polynomial in $d$. The term $g_2(\Delta)$ is the only other one that could have this property. This implies that \[ \frac{g_{k + 1}(\Delta)}{g_k(\Delta)} = o \left( \frac{d g_1(\Delta)}{k^2} \right) \] or \[ \frac{g_{k + 1}(\Delta)}{g_k(\Delta)} = o \left( \frac{g_2(\Delta)}{k^2} \right). \] 
					
					Specializing this to the case where $k$ is linear in $d$, we would have \[ \frac{g_{k + 1}(\Delta)}{g_k(\Delta)} = o \left( \frac{f_1(\Delta)}{d^2} \right). \] Then, we have \[  \frac{g_{k + 1}(\Delta)}{g_k(\Delta)} = o \left( \frac{g_1(\Delta)}{d} \right) \] or \[ \frac{g_{k + 1}(\Delta)}{g_k(\Delta)} = o \left( \frac{g_2(\Delta)}{d^2} \right). \]

					\item Since $g_1(\Delta)$ is linear in $d$, Proposition \ref{interlacingnontriv} implies that \[ f_1(\Delta) = \binom{d + 1}{2} + d g_1(\Delta) + g_2(\Delta) \] is quadratic in $d$. Note that $g_1(\Delta) \ge d - 1$ and $g_2(\Delta) \ge \frac{1}{3} g_1(\Delta)$ since $\gamma(\Delta) \ge 0$ (Corollary \ref{gamgratrest}). This implies that \[ \lim_{d \to +\infty} \frac{1}{d^2} \frac{\mathcal{C}_{k - 1}}{g_k(\Delta)} = 0. \]
					
					If $k$ is asymptotically linear in $d$ with $\lim_{\alpha \to +\infty} \frac{k}{d} = \alpha$ for some $0 < \alpha \le \frac{1}{2}$, then we have \[ \frac{\mathcal{C}_{k - 1}}{g_k(\Delta)} \ge k(d - k + 1) \Longrightarrow \lim_{d \to +\infty} \frac{1}{d^2} \frac{\mathcal{C}_{k - 1}}{g_k(\Delta)} \ge \lim_{d \to +\infty} \frac{1}{d^2} k(d - k + 1) = \alpha(1 - \alpha) > 0. \] Thus, it is impossible to have \[  \lim_{d \to +\infty} \frac{1}{d^2} \frac{\mathcal{C}_{k - 1}}{g_k(\Delta)} = 0 \] in this case. \\
					
					Suppose that $k$ is sublinear in $d$. Then, nonnegativity of the terms added to form $\frac{\mathcal{C}_{k - 1}}{g_k(\Delta)}$ implies that \[ \frac{g_{k + 1}(\Delta)}{g_k(\Delta)} = o \left( \frac{d^2}{k^2} \right) \] and \[ \lim_{d \to +\infty} \frac{g_{k - 1}(\Delta)}{g_k(\Delta)} = 0. \]

				\end{enumerate}

			\item 
				\begin{enumerate}
					\item If $d = o(g_1(\Delta))$, we have that \[ f_1(\Delta) = \binom{d + 1}{2} + d g_1(\Delta) + g_2(\Delta) \] is asymptotically strictly larger than any quadratic polynomial in $d$. Since $\gamma(\Delta) \ge 0$, this implies that the quotients of the last two terms of $\frac{ \mathcal{C}_{k - 1} }{g_k(\Delta)}$ by $f_1(\Delta)$ approach 0 as $d \to +\infty$ since $\frac{g_{k - 1}(\Delta)}{g_k(\Delta)} \le 3$ (Corollary \ref{gamgratrest}). This means that the only possible term contributing a nonzero limit is \[ k(k + 1) \frac{g_{k + 1}(\Delta)}{g_k(\Delta)}. \]
					
					Dividing the numerator and denominator by $k^2$ when $\lim_{d \to +\infty} k = +\infty$ and factoring out $k(k + 1)$ when $k$ is a constant gives the stated limits. \\
					
					\item If $g_1(\Delta)$ is linear in $d$, then $d g_1(\Delta)$ and $g_2(\Delta)$ are both quadratic in $d$ (Proposition \ref{gamposling1h1}). Suppose that $k$ is linear in $d$ with $\lim_{d \to +\infty} \frac{k}{d} = \alpha$ for some $0 < \alpha \le \frac{1}{2}$. Then, dividing the numerator and denominator of the limit by $d^2$ and taking limits gives the claimed limit. \\
					
					Now consider the case where $k$ is sublinear in $d$ (e.g. a constant). This means that the term $k(d - k + 1)$ is asymptotically strictly smaller than a quadratic polynomial in $d$. Then, we have \[ \lim_{d \to +\infty} \frac{1}{f_1(\Delta)} \frac{\mathcal{C}_{k - 1}}{g_k(\Delta)} = \frac{1}{2} \lim_{d \to +\infty} \frac{ k(k + 1) \frac{g_{k + 1}(\Delta)}{g_k(\Delta)} + (d - k + 1)(d - k + 2) \frac{g_{k - 1}(\Delta)}{g_k(\Delta)} }{f_1(\Delta)}. \] Note that $\frac{g_{k - 1}(\Delta)}{g_k(\Delta)} \le 3$ since $\gamma(\Delta) \ge 0$ (Corollary \ref{gamgratrest}). \\

					If $\lim_{d \to +\infty} k = +\infty$, then we divide the numerator and denominator by $k^2$. The stated limit is then obtained by taking the asymptotically largest terms. \\
					
					The case where $k$ is a constant follows from dividing the numerator and denominator by $d^2$ since each term added to form $\frac{\mathcal{C}_{k - 1}}{g_k(\Delta)}$ is at most linear in $d$ by Proposition \ref{gamposling1h1} (also see Section \ref{sumconst}). \\

				\end{enumerate}

			\item Since $\mu \ne 0$, this follows from multiplying the expressions in Part 2 by $\mu$.

		\end{enumerate}
	\end{proof}

	In the inequality \[ 1 - \frac{1}{f_1(\Delta)} \frac{\mathcal{C}_k}{g_{k + 1}(\Delta)} \le f_1(\Delta)^{ \frac{1}{k} } \frac{ \left( 1 - \frac{1}{f_1(\Delta)} \frac{\mathcal{C}_{k - 1}}{g_k(\Delta)} \right)^{\frac{k + 1}{k}} }{ \frac{g_{k + 1}(\Delta)}{g_k(\Delta)^{ \frac{k + 1}{k} }} } = f_1(\Delta)^{ \frac{1}{k} } \frac{g_k(\Delta)^{ \frac{k + 1}{k} }}{g_{k + 1}(\Delta)} \left( 1 - \frac{1}{f_1(\Delta)} \frac{\mathcal{C}_{k - 1}}{g_k(\Delta)} \right)^{\frac{k + 1}{k}}, \] it remains to consider interactions with the terms \[ \frac{1}{f_1(\Delta)^{ \frac{1}{k} }} \frac{g_{k + 1}(\Delta)}{g_k(\Delta)^{ \frac{k + 1}{k} }}. \] Note that this is equal to $\frac{1}{\omega_k}$, where $\omega_k$ is involved in the lower bound  
	\begin{align*}
		\frac{1}{f_1(\Delta)} \frac{\mathcal{C}_k}{g_{k + 1}(\Delta)} &\ge \left( 1 - \omega_k^{ \frac{k}{k + 1} } \left( 1 - \frac{1}{f_1(\Delta)} \frac{\mathcal{C}_{k - 1}}{g_k(\Delta)} \right) \right)^{ \frac{k + 1}{k} } \\
		&= \left( \omega_k^{ \frac{k}{k + 1} } \left( \frac{1}{f_1(\Delta)} \frac{\mathcal{C}_{k - 1}}{g_k(\Delta)} \right) + \left( 1 - \omega_k^{ \frac{k}{k + 1} } \right) \right)^{ \frac{k + 1}{k} },
	\end{align*}

	in Proposition \ref{ratlowbd}. \\

	Before relating this back to the earlier limits, we will make some comments on the behavior of the term \[ \frac{1}{f_1(\Delta)^{ \frac{1}{k} }} \frac{g_{k + 1}(\Delta)}{g_k(\Delta)^{ \frac{k + 1}{k} }}  \] as $d \to +\infty$. \\
	
	\begin{prop} \label{zerolinkcoeff}
		Let $\Delta$ be a $(d - 1)$-dimensional simplicial sphere. Consider $1 \le k \le \frac{d}{2} - 1$. \\
		
		\begin{enumerate}
			\item If $k$ is a constant with respect to $d$, then \[ \lim_{d \to +\infty} \frac{1}{f_1(\Delta)^{ \frac{1}{k} }} \frac{g_{k + 1}(\Delta)}{g_k(\Delta)^{ \frac{k + 1}{k} }} = 0. \] 
			
			\item Suppose that $k$ is a linear or sublinear function in $d$. Then, we have that \[ \lim_{d \to +\infty} \frac{1}{f_1(\Delta)^{ \frac{1}{k} }} = 0 \] if and only if \[ k = o(\log f_1(\Delta)). \]
			
			\begin{itemize}
				\item If $g_1(\Delta)$ is linear in $d$, this is equivalent to $k = o(\log d)$. \\
				
				\item  If $d = o(g_1(\Delta))$ (e.g. $d^2 = o(g_2(\Delta))$), we would either have \[ k = o(\log (d g_1(\Delta))) = o(\log d + \log g_1(\Delta)) \] or $k = o(\log g_2(\Delta))$. \\
			\end{itemize}

			Note that \[ \lim_{d \to +\infty} \frac{1}{f_1(\Delta)^{ \frac{1}{k} }} \frac{g_{k + 1}(\Delta)}{g_k(\Delta)^{ \frac{k + 1}{k} }} = 0 \] under such conditions. \\
			
		\end{enumerate}
		
	\end{prop}
	
	\begin{proof}
		\begin{enumerate}
			\item Recall that $g_{k + 1}(\Delta) \le g_k(\Delta)^{ \frac{k + 1}{k} }$ by Corollary \ref{macpowasymp} and the $M$-vector condition on $g(\Delta)$ (Part 2c) of Theorem \ref{macpowasymp}). This means that \[ \frac{1}{f_1(\Delta)^{ \frac{1}{k} }} \frac{g_{k + 1}(\Delta)}{g_k(\Delta)^{ \frac{k + 1}{k} }} \le \frac{1}{f_1(\Delta)^{ \frac{1}{k} }}. \] 
			
			Since $\lim_{d \to +\infty} f_1(\Delta) = +\infty$ and $k$ is a constant, we have that \[ \lim_{d \to +\infty} \frac{1}{f_1(\Delta)^{ \frac{1}{k} }} = 0. \] \\
			
			This implies that \[ \lim_{d \to +\infty}  \frac{1}{f_1(\Delta)^{ \frac{1}{k} }} \frac{g_{k + 1}(\Delta)}{g_k(\Delta)^{ \frac{k + 1}{k} }} = 0 \] since the terms we take limits of are nonnegative. \\
			
			\item The condition is equivalent to \[ \lim_{d \to +\infty} - \frac{1}{k} \log f_1(\Delta) = -\infty. \] This would mean that \[ \lim_{d \to +\infty} \frac{k}{\log f_1(\Delta)} = 0 \] and $k = o(\log f_1(\Delta))$. \\
			
			If $g_1(\Delta)$ is linear in $d$, then \[ f_1(\Delta) = \binom{d + 1}{2} + d g_1(\Delta) + g_2(\Delta) \] is quadratic in $d$ by Proposition \ref{gamposling1h1}. This implies that $k = o(\log d^2) = o(\log d)$. In general, having $k = o(\log M)$ for some function $M = M(d)$ with $\lim_{d \to +\infty} M = +\infty$ would imply that $k = o(\log (c M))$ for any positive constant $c > 0$ with respect to $d$. \\
			
			Suppose that $d = o(g_1(\Delta))$ (e.g. when $d^2 = o(g_2(\Delta))$). Then, we have $k = o(\log(d g_1(\Delta)))$ or $k = o(\log g_2(\Delta))$ by the same reasoning as above since $d g_1(\Delta)$ or $g_2(\Delta)$ are asymptotically strictly larger than the quadratic term $\binom{d + 1}{2}$. \\
		\end{enumerate}
	\end{proof}

	Since the induced limit inequality holds trivially when \[ \lim_{d \to +\infty} \frac{1}{f_1(\Delta)} \frac{\mathcal{C}_{k - 1}}{g_k(\Delta)} = 0 \] by Part 1 of Proposition \ref{linkinitasymp}, we will focus on the case where \[ 0 < \lim_{d \to +\infty} \frac{1}{f_1(\Delta)} \frac{\mathcal{C}_{k - 1}}{g_k(\Delta)} \le 1. \]
	
	Substituting in the cases above, we have the following: \\
	
	\begin{cor} \textbf{(Coefficients of main terms and initial link condition asymptotics) \\} \label{coefflinkasymp}
		Suppose that $\Delta$ is a $(d - 1)$-dimensional simplicial sphere satisfying the link condition such that $\gamma(\Delta) \ge 0$. 
		
		\begin{enumerate}
			\item Suppose that \[ \lim_{d \to +\infty} \frac{1}{f_1(\Delta)} \frac{\mathcal{C}_{k - 1}}{g_k(\Delta)} = 1 \] as in Part 2 of Corollary \ref{asympgratinit}. \\
				\begin{enumerate}
					\item If \[ \lim_{d \to +\infty} f_1(\Delta)^{ \frac{1}{k} } \frac{ \left( 1 - \frac{1}{f_1(\Delta)} \frac{\mathcal{C}_{k - 1}}{g_k(\Delta)} \right)^{\frac{k + 1}{k}} }{ \frac{g_{k + 1}(\Delta)}{g_k(\Delta)^{ \frac{k + 1}{k} }} } = +\infty, \] the induced limit inequality holds trivially. \\
					
					Note that this would require \[ \lim_{d \to +\infty} \frac{1}{f_1(\Delta)^{ \frac{1}{k} }} \frac{g_{k + 1}(\Delta)}{g_k(\Delta)^{ \frac{k + 1}{k} }} = 0. \]
					
					Examples of sufficiently small $k$ where the limit above is equal to 0 (e.g. $k$ a constant) are described in Proposition \ref{zerolinkcoeff}. In order to for the upper bound of the inequality to approach $+\infty$ as $d \to +\infty$, this term must approach 0 more quickly than the numerator. \\
					
					\item Assume that \[ \lim_{d \to +\infty} f_1(\Delta)^{ \frac{1}{k} } \frac{ \left( 1 - \frac{1}{f_1(\Delta)} \frac{\mathcal{C}_{k - 1}}{g_k(\Delta)} \right)^{\frac{k + 1}{k}} }{ \frac{g_{k + 1}(\Delta)}{g_k(\Delta)^{ \frac{k + 1}{k} }} } = \beta \] for some finite constant $\beta$. \\
						\begin{enumerate}
							\item If $\beta > 0$, then we have \[ \lim_{d \to +\infty} \frac{1}{f_1(\Delta)^{ \frac{1}{k} }} \frac{g_{k + 1}(\Delta)}{g_k(\Delta)^{ \frac{k + 1}{k} }} = 0. \] 
							
							Examples of sufficiently small $k$ where this holds (e.g. $k$ a constant) are described in Proposition \ref{zerolinkcoeff}. \\

							Since $\beta > 0$ is finite, this limit implies that \[ \left( 1 - \frac{1}{f_1(\Delta)} \frac{\mathcal{C}_{k - 1}}{g_k(\Delta)} \right)^{\frac{k + 1}{k}} = \Theta \left( \frac{1}{f_1(\Delta)^{ \frac{1}{k} }} \frac{g_{k + 1}(\Delta)}{g_k(\Delta)^{ \frac{k + 1}{k} }} \right) = \Theta \left( \left( \frac{ f_1(\Delta)^{ \frac{k}{k + 1} } }{g_{k + 1}(\Delta)^{ \frac{1}{k + 1} }} \frac{1}{f_1(\Delta)} \frac{g_{k + 1}(\Delta)}{g_k(\Delta)} \right)^{ \frac{k + 1}{k} } \right) \] with respect to $d$. Equivalently, we have \[ \frac{1}{f_1(\Delta)^{ \frac{1}{k} }} \frac{g_{k + 1}(\Delta)}{g_k(\Delta)^{ \frac{k + 1}{k} }} = \left( \frac{ f_1(\Delta)^{ \frac{k}{k + 1} } }{g_{k + 1}(\Delta)^{ \frac{1}{k + 1} }} \frac{1}{f_1(\Delta)} \frac{g_{k + 1}(\Delta)}{g_k(\Delta)} \right)^{ \frac{k + 1}{k} } = \Theta \left( \left( 1 - \frac{1}{f_1(\Delta)} \frac{\mathcal{C}_{k - 1}}{g_k(\Delta)} \right)^{\frac{k + 1}{k}} \right). \]
							
							Note that \[ \frac{1}{f_1(\Delta)^{ \frac{1}{k} }} \frac{g_{k + 1}(\Delta)}{g_k(\Delta)^{ \frac{k + 1}{k} }} \le \frac{1}{f_1(\Delta)^{ \frac{1}{k} }}. \]
							
							\item If $\beta = 0$, then we have \[ \left( 1 - \frac{1}{f_1(\Delta)} \frac{\mathcal{C}_{k - 1}}{g_k(\Delta)} \right)^{\frac{k + 1}{k}} = o \left( \frac{1}{f_1(\Delta)^{\frac{1}{k}} } \frac{g_{k + 1}(\Delta)}{g_k(\Delta)^{ \frac{k + 1}{k} }} \right) = o \left( \left( \frac{ f_1(\Delta)^{ \frac{k}{k + 1} } }{g_{k + 1}(\Delta)^{ \frac{1}{k + 1} }} \frac{1}{f_1(\Delta)} \frac{g_{k + 1}(\Delta)}{g_k(\Delta)} \right)^{ \frac{k + 1}{k} } \right) \] by definition. Equivalently, we have \[ \lim_{d \to +\infty} \frac{ \frac{1}{f_1(\Delta)^{\frac{1}{k}}  } \frac{g_{k + 1}(\Delta)}{g_k(\Delta)^{ \frac{k + 1}{k} }} }{ \left( 1 - \frac{1}{f_1(\Delta)} \frac{\mathcal{C}_{k - 1}}{g_k(\Delta)} \right)^{ \frac{k + 1}{k} } } = \lim_{d \to +\infty} \frac{ \left( \frac{ f_1(\Delta)^{ \frac{k}{k + 1} } }{g_{k + 1}(\Delta)^{ \frac{1}{k + 1} }} \frac{1}{f_1(\Delta)} \frac{g_{k + 1}(\Delta)}{g_k(\Delta)} \right)^{ \frac{k + 1}{k} } }{ \left( 1 - \frac{1}{f_1(\Delta)} \frac{\mathcal{C}_{k - 1}}{g_k(\Delta)} \right)^{ \frac{k + 1}{k} } } =  +\infty. \]
							
							In particular, this would imply that \[ \left( 1 - \frac{1}{f_1(\Delta)} \frac{\mathcal{C}_{k - 1}}{g_k(\Delta)} \right)^{\frac{k + 1}{k}} = o \left( \frac{1}{f_1(\Delta)^{\frac{1}{k}} } \right)  \] and \[ \lim_{d \to +\infty} \frac{ \frac{1}{f_1(\Delta)^{\frac{1}{k}}  } }{ \left( 1 - \frac{1}{f_1(\Delta)} \frac{\mathcal{C}_{k - 1}}{g_k(\Delta)} \right)^{ \frac{k + 1}{k} } } = +\infty. \]
							
							There is a wider range of possible behavior of \[ \frac{1}{f_1(\Delta)^{ \frac{1}{k} }} \frac{g_{k + 1}(\Delta)}{g_k(\Delta)^{\frac{k + 1}{k}}} \] since we can get the same limit whether it approaches a nonzero constant as $d \to +\infty$ or approaches 0 as $d \to +\infty$ but does so more slowly than $\left( 1 - \frac{1}{f_1(\Delta)} \frac{\mathcal{C}_{k - 1}}{g_k(\Delta)} \right)^{\frac{k + 1}{k}}$. \\
							
						\end{enumerate}

				\end{enumerate}

			\item Suppose that \[ \lim_{d \to +\infty} \frac{1}{f_1(\Delta)} \frac{\mathcal{C}_{k - 1}}{g_k(\Delta)} = \mu \] with $0 < \mu < 1$ as in Part 3 of Corollary \ref{asympgratinit}. \\
				\begin{enumerate}
					\item If \[ \lim_{d \to +\infty} f_1(\Delta)^{ \frac{1}{k} } \frac{ g_k(\Delta)^{ \frac{k + 1}{k} } }{g_{k + 1}(\Delta)} = +\infty \] or \[ \lim_{d \to +\infty} f_1(\Delta)^{ \frac{1}{k} } \frac{ g_k(\Delta)^{ \frac{k + 1}{k} } }{g_{k + 1}(\Delta)}  = \eta \] with \[ \eta \ge \frac{1}{(1 - \mu)^{ \frac{k + 1}{k} }}, \] the induced limit inequality holds trivially. For example, the first case includes the small $k$ (e.g. $k$ a constant) from Proposition \ref{zerolinkcoeff}. \\
					
					\item Assume that \[ \lim_{d \to +\infty} f_1(\Delta)^{ \frac{1}{k} } \frac{ g_k(\Delta)^{ \frac{k + 1}{k} } }{g_{k + 1}(\Delta)}  = \eta \] with \[ \eta < \frac{1}{(1 - \mu)^{ \frac{k + 1}{k} }}. \]
					
					In order for this to occur, the parameter $k$ must increase sufficiently quickly as $d$ increases (Proposition \ref{zerolinkcoeff}). \\
					
					Then, we have \[ \lim_{d \to +\infty} \frac{1}{f_1(\Delta)} \frac{\mathcal{C}_k}{g_{k + 1}(\Delta)} \ge 1 - \eta(1 - \mu)^{ \frac{k + 1}{k} }. \]
					
					Since there is a strictly positive lower bound, this means that we are in the setting of Part 3 of Corollary \ref{asympgratinit} with $k$ replaced by $k + 1$ and \[ \sigma \coloneq \lim_{d \to +\infty} \frac{1}{f_1(\Delta)} \frac{\mathcal{C}_k}{g_{k + 1}(\Delta)} \] replacing $\mu$. \\

				\end{enumerate}

		\end{enumerate}
		
	\end{cor}

	\subsubsection{Summary and asymptotic behavior of individual inequalities} \label{sumconst}

	Let $\Delta$ be a $(d - 1)$-dimensional simplicial sphere satisfying the link condition. Recall that the inequality \[ g_{k + 1}(\Delta) - \frac{\mathcal{C}_k}{f_1(\Delta)} \le f_1(\Delta)^{ \frac{1}{k} } \left( g_k(\Delta) - \frac{\mathcal{C}_{k - 1}}{f_1(\Delta)} \right)^{\frac{k + 1}{k}} \] is a comparison of consecutive components of ``average $g$-vectors'' on local parts. We mainly work with its (equivalent) counterpart  \[ 1 - \frac{1}{f_1(\Delta)} \frac{\mathcal{C}_k}{g_{k + 1}(\Delta)} \le f_1(\Delta)^{ \frac{1}{k} } \frac{ \left( 1 - \frac{1}{f_1(\Delta)} \frac{\mathcal{C}_{k - 1}}{g_k(\Delta)} \right)^{\frac{k + 1}{k}} }{ \frac{g_{k + 1}(\Delta)}{g_k(\Delta)^{ \frac{k + 1}{k} }} } = f_1(\Delta)^{ \frac{1}{k} } \frac{ \left( 1 - \frac{1}{f_1(\Delta)} \frac{\mathcal{C}_{k - 1}}{g_k(\Delta)} \right)^{\frac{k + 1}{k}} }{ \frac{g_{k + 1}(\Delta)}{g_k(\Delta) }} \cdot g_k(\Delta)^{ \frac{1}{k} } \] comparing ratios between average $g$-vector components of local parts (Definition \ref{shortglob}). \\

	Expanding the averaged terms $\mathcal{C}_k$ and $\mathcal{C}_{k - 1}$ in terms of global parts, this gives rise to a lower bound of $\frac{g_{k + 2}(\Delta)}{g_{k + 1}(\Delta)}$ in terms of $\frac{g_{k + 1}(\Delta)}{g_k(\Delta)}$, $f_1(\Delta)$, and $\frac{g_{k + 1}(\Delta)}{g_k(\Delta)^{ \frac{k + 1}{k} }}$. The simplicial sphere $\Delta$ having a smaller number of edges or the $M$-vector inequality on $g(\Delta)$ (Part 2c) of Theorem \ref{gconjsphere}) would yield a tighter bound (i.e. larger lower bound). When $d \gg 0$, the lower bound increases as $\frac{g_{k + 1}(\Delta)}{g_k(\Delta)}$ decreases. \\

	The general shape of the resulting inequality is described in Proposition \ref{ratlowbd} as 
	\begin{align*}
		\frac{1}{f_1(\Delta)} \frac{\mathcal{C}_k}{g_{k + 1}(\Delta)} &\ge \left( 1 - \omega_k^{ \frac{k}{k + 1} } \left( 1 - \frac{1}{f_1(\Delta)} \frac{\mathcal{C}_{k - 1}}{g_k(\Delta)} \right) \right)^{ \frac{k + 1}{k} } \\
		&= \left( \omega_k^{ \frac{k}{k + 1} } \left( \frac{1}{f_1(\Delta)} \frac{\mathcal{C}_{k - 1}}{g_k(\Delta)} \right) + \left( 1 - \omega_k^{ \frac{k}{k + 1} } \right) \right)^{ \frac{k + 1}{k} },
	\end{align*}
	
	where \[ \omega_k = f_1(\Delta)^{ \frac{1}{k} } \frac{g_k(\Delta)^{ \frac{k + 1}{k} }}{g_{k + 1}(\Delta)}. \]
	
	In a special case (Example \ref{simpratlowbd}), this simplifies to  \[ \frac{1}{f_1(\Delta)} \frac{\mathcal{C}_k}{g_{k + 1}(\Delta)} \ge \left( \frac{1}{f_1(\Delta)} \frac{\mathcal{C}_{k - 1}}{g_k(\Delta)} \right)^{ \frac{k + 1}{k} }. \]
	
	Comparing this to general properties of simplicial spheres, we note that this link condition inequality does \emph{not} seem to impose nontrivial conditions beyond the $M$-vector inequality for the $g$-vectors of arbitrary simplicial spheres if $g_1(\Delta)$ is very large compared to the rate of growth of the $g$-vector components (Theorem \ref{largelinktriv}) and suggest a closer connection between $\gamma(\Delta) \ge 0$ and $f$-vector realizations in this setting (Remark \ref{largeparamint}). \\

	In general, we take a closer look at the specific types of bounds obtained when $\dim \Delta$ is large. This involves behavior of ``limiting inequalities'' as $d \to +\infty$. We have considered index parameters $1 \le k \le \frac{d}{2}$ that are be linear or sublinear (e.g. constant) in $d$ (e.g. see Proposition \ref{linkinitasymp} and Corollary \ref{asympgratinit}). In this section, we focus on the case where $k$ is constant. It describes limiting behavior of \emph{individual} link condition inequalities, while the other types of functions describe those of \emph{collections} of link condition inequalities. This is where we will mainly start using the assumption $\gamma(\Delta) \ge 0$. The induced limit inequality (i.e. for $d \gg 0$) holds trivially if there is a constant $C > 0$ such that \[ \left( 1 - \frac{1}{f_1(\Delta)} \frac{\mathcal{C}_{k - 1}}{g_k(\Delta)} \right)^{\frac{k + 1}{k}} > C \] for all $d \gg 0$ (Remark \ref{trivineqasymp}). Note that \[ 0 \le \frac{1}{f_1(\Delta)} \frac{\mathcal{C}_{k - 1}}{g_k(\Delta)} \le 1 \] and \[ 0 \le \frac{1}{f_1(\Delta)} \frac{\mathcal{C}_k}{g_{k + 1}(\Delta)}  \le 1. \] 
	
	While the above triviality condition does not require a limit to exist as $d \to +\infty$, we will assume that \[ \lim_{d \to +\infty} \frac{1}{f_1(\Delta)} \frac{\mathcal{C}_{k - 1}}{g_k(\Delta)} \] and \[ \lim_{d \to +\infty} \frac{1}{f_1(\Delta)} \frac{\mathcal{C}_k}{g_{k + 1}(\Delta)} \] exist in order study what functions yield the possible output behaviors. A similar analysis can be used to control behavior in cases where a limit does not exist (e.g. oscillating functions) bounded by such functions. \\

	For example, such a lower bound $C > 0$ exists if \[ \lim_{d \to +\infty} \frac{1}{f_1(\Delta)} \frac{\mathcal{C}_{k - 1}}{g_k(\Delta)} = \mu \] for some $0 < \mu < 1$ and $k$ is constant. The $g$-vector ratios inducing this behavior described in Part 3 of Corollary \ref{asympgratinit}. If $d = o(g_1(\Delta))$ (e.g. when $d^2 = o(g_2(\Delta))$), we have \[ \lim_{d \to +\infty} \frac{ \frac{g_{k + 1}(\Delta)}{g_k(\Delta)} }{f_1(\Delta)} = \frac{2 \mu}{k(k + 1)}.  \] 
	
	Note that asymptotic equivalence of $\frac{g_{k + 1}(\Delta)}{g_k(\Delta)}$ with $f_1(\Delta)$ can be replaced by $d g_1(\Delta)$ or $g_2(\Delta)$ in this setting. The same reasoning translates to comparisons with $\frac{g_1(\Delta)}{d}$ and $\frac{g_2(\Delta)}{d^2}$ when $g_1(\Delta)$ is (asymptotically) linear in $d$. Since $g_\ell(\Delta)$ is of degree $\ell$ in $d$ (Proposition \ref{gamposling1h1}) in this setting, we would have \[ \lim_{d \to +\infty}  \frac{g_{k - 1}(\Delta)}{g_k(\Delta)} = \mu \left( 1 + 2 \lim_{d \to +\infty} \frac{g_1(\Delta)}{d} + 2 \lim_{d \to +\infty} \frac{g_2(\Delta)}{d^2} \right). \] However, this is impossible since the left hand side is equal to 0 and the right hand side is strictly positive. Other cases (e.g. involving $k$ as a linear or sublinear function of $d$) inducing trivial induced limit inequalities are Part 1 and Part 3 of Proposition \ref{linkinitasymp} with $g$-vector ratios inducing them described in other parts of Corollary \ref{asympgratinit}.  \\
	
	Another instance where the induced limit inequality holds trivially is when \[ \lim_{d \to +\infty} \frac{1}{f_1(\Delta)} \frac{\mathcal{C}_{k - 1}}{g_k(\Delta)} = 0. \] By Part 1 of Corollary \ref{asympgratinit}, this splits into the following cases:
	
	\begin{itemize}
		\item If $d = o(g_1(\Delta))$ (e.g. when $d^2 = o(g_2(\Delta))$), then \[ \frac{g_{k + 1}(\Delta)}{g_k(\Delta)} = o(dg_1(\Delta)) \] or \[ \frac{g_{k + 1}(\Delta)}{g_k(\Delta)} = o(g_2(\Delta)). \]
		
		This case was implicitly discussed in Theorem \ref{largelinktriv} and Remark \ref{largeparamint}. \\
		
		\item If $g_1(\Delta)$ is (asymptotically) linear in $d$, this is always true since \[ \frac{g_{k + 1}(\Delta)}{g_k(\Delta)} = \Theta(d) \] and \[ \frac{g_{k - 1}(\Delta)}{g_k(\Delta)} = \Theta \left( \frac{1}{d} \right) \] by Proposition \ref{gamposling1h1}. \\
		
	\end{itemize}
	
	The trivial cases discussed above are summarized below. \\
	
	\begin{cor} \textbf{(Triviality criteria for individual link condition limiting inequalities) \\} \label{consttriv} 
		Let $\Delta$ be a $(d - 1)$-dimensional simplicial sphere satisfying the link condition such that $\gamma(\Delta) \ge 0$. Fix $1 \le k \le \frac{d}{2} - 2$. The limiting inequality of the global/averaged lower bound induced by the link condition (Corollary \ref{diffcush}) holds trivially for $d \gg 0$ under the following conditions: \\
		
		\begin{enumerate}
			\item Suppose that $d = o(g_1(\Delta))$ (e.g. when $d^2 = o(g_2(\Delta))$). \\
				\begin{enumerate}
					\item There is a constant $0 < \mu < 1$ such that \[ \lim_{d \to +\infty} \frac{ \frac{g_{k + 1}(\Delta)}{g_k(\Delta)} }{f_1(\Delta)} = \frac{2 \mu}{k(k + 1)}.  \] Equivalently, we have  \[ \lim_{d \to +\infty} \frac{ \frac{g_{k + 1}(\Delta)}{g_k(\Delta)} }{d g_1(\Delta)} = \frac{2 \mu}{k(k + 1)}  \] or \[ \lim_{d \to +\infty} \frac{ \frac{g_{k + 1}(\Delta)}{g_k(\Delta)} }{g_2(\Delta)} = \frac{2 \mu}{k(k + 1)}.  \]
					
					\item The ratio $\frac{g_{k + 1}(\Delta)}{g_k(\Delta)}$ is constant or increases slowly compared to $f_1(\Delta)$ as $d \to +\infty$. \\ 
					
					More specifically, we have \[ \frac{g_{k + 1}(\Delta)}{g_k(\Delta)} = o(dg_1(\Delta)) \] or \[ \frac{g_{k + 1}(\Delta)}{g_k(\Delta)} = o(g_2(\Delta)). \]
					
					This case was implicitly discussed in Theorem \ref{largelinktriv} and Remark \ref{largeparamint}. \\
				\end{enumerate} 
			
			\item The term $g_1(\Delta)$ is (asymptotically) linear in $d$. \\
		\end{enumerate}
	\end{cor}
	
	\begin{rem} \textbf{(Gamma vector implications and further generalizations) \\} 
	\vspace{-3mm} 
	\begin{enumerate}
		\item \textbf{(Nonnegative gamma vector realizations) \\} 
		
		Since the gamma vector is obtained from the $g$-vector via an invertible linear transformation (coefficients in Proposition \ref{gamtohgvect}), this is an indication that the link condition only has a small effect on nonnegative gamma vectors realized by high-dimensional simplicial spheres in the setting of Corollary \ref{consttriv}. \\
		
		\item \textbf{(Generalizations to other linear/sublinear functions $k$ in $d$) \\}
		
		Taking a constant $k$ means that we consider asymptotic behavior of \emph{individual} limiting inequalities as $d \to +\infty$. We can study variations of this behavior from modification of the indices by applying generalizations to other linear/sublinear functions $k$ in $d$ from Proposition \ref{linkinitasymp} and Corollary \ref{asympgratinit}. \\

	\end{enumerate}

	\end{rem}

	We will now consider cases where the limiting inequality induced by the link condition lower bound may hold nontrivially. As mentioned above, we will assume that $k$ is constant for simplicity and refer to earlier results where we other sublinear functions or linear functions in $d$ are covered (e.g. Proposition \ref{linkinitasymp} and Corollary \ref{asympgratinit}). Given the context above, we will focus on the case where  \[ \lim_{d \to +\infty} \frac{1}{f_1(\Delta)} \frac{\mathcal{C}_{k - 1}}{g_k(\Delta)} = 1 \] discussed in Part 2 of Proposition \ref{linkinitasymp} and Part 2 of Corollary \ref{asympgratinit}, which means that \[ \lim_{d \to +\infty} \left( 1 - \frac{1}{f_1(\Delta)} \frac{\mathcal{C}_{k - 1}}{g_k(\Delta)} \right)^{ \frac{k + 1}{k} } = 0. \]
	
	The possible $g$-vector ratios inducing such behavior is discussed in Part 2 of Corollary \ref{asympgratinit}. By Corollary \ref{consttriv}, we can assume that $d = (o(g_1(\Delta)))$ (e.g. when $d^2 = o(g_2(\Delta))$). Then, we have \[ \lim_{d \to +\infty} \frac{ \frac{g_{k + 1}(\Delta)}{g_k(\Delta)} }{f_1(\Delta)} = \frac{2}{k(k + 1)}.  \] In this setting, we can replace $f_1(\Delta)$ by $d g_1(\Delta)$ or $g_2(\Delta)$. \\

	We will now compare these to the ``coefficient terms'' related to the number of edges $f_1(\Delta)$ and the ratio $\frac{g_k(\Delta)^{ \frac{k + 1}{k} }}{g_{k + 1}(\Delta)}$ from the $M$-vector condition on $g(\Delta)$ for general simplicial spheres (Part 2c) of Theorem \ref{gconjsphere}). Recall that \[ \lim_{d \to +\infty} f_1(\Delta)^{ \frac{1}{k} } \frac{g_k(\Delta)^{ \frac{k + 1}{k} }}{g_{k + 1}(\Delta)} = +\infty \] since $k$ is a constant. Equivalently, we have \[ \lim_{d \to +\infty} \frac{1}{f_1(\Delta)^{ \frac{1}{k} }} \frac{g_{k + 1}(\Delta)}{g_k(\Delta)^{ \frac{k + 1}{k} }} = 0. \] As mentioned in Corollary \ref{coefflinkasymp}, the limit \[ \lim_{d \to +\infty} f_1(\Delta)^{ \frac{1}{k} } \frac{ \left( 1 - \frac{1}{f_1(\Delta)} \frac{\mathcal{C}_{k - 1}}{g_k(\Delta)} \right)^{\frac{k + 1}{k}} }{ \frac{g_{k + 1}(\Delta)}{g_k(\Delta)^{ \frac{k + 1}{k} }} } = \lim_{d \to +\infty} \frac{ \left( 1 - \frac{1}{f_1(\Delta)} \frac{\mathcal{C}_{k - 1}}{g_k(\Delta)} \right)^{\frac{k + 1}{k}} }{ \frac{1}{f_1(\Delta)^{ \frac{1}{k} }} \frac{g_{k + 1}(\Delta)}{g_k(\Delta)^{ \frac{k + 1}{k} }} } \] has a range of possible behaviors since the numerator and denominator both approach 0 as $d \to +\infty$. If \[ \lim_{d \to +\infty} \frac{ \left( 1 - \frac{1}{f_1(\Delta)} \frac{\mathcal{C}_{k - 1}}{g_k(\Delta)} \right)^{\frac{k + 1}{k}} }{ \frac{1}{f_1(\Delta)^{ \frac{1}{k} }} \frac{g_{k + 1}(\Delta)}{g_k(\Delta)^{ \frac{k + 1}{k} }} } = +\infty, \] then the limiting inequality induced by the lower bound inequality holds trivially (i.e. for $d \gg 0$) since the upper bound can get arbitrarily large while the lower bound stays between 0 and 1. This is the case where the denominator approaches 0 more quickly. \\
	
	Otherwise, we have \[ \lim_{d \to +\infty} \frac{ \left( 1 - \frac{1}{f_1(\Delta)} \frac{\mathcal{C}_{k - 1}}{g_k(\Delta)} \right)^{\frac{k + 1}{k}} }{ \frac{1}{f_1(\Delta)^{ \frac{1}{k} }} \frac{g_{k + 1}(\Delta)}{g_k(\Delta)^{ \frac{k + 1}{k} }} } = \beta, \] for (finite) constant $\beta \ge 0$. \\
	
	The induced limiting inequality is then 
	\begin{align*}
		1 - \lim_{d \to +\infty} \frac{1}{f_1(\Delta)} \frac{\mathcal{C}_k}{g_{k + 1}(\Delta)} &\le \beta \\
		\Longleftrightarrow \lim_{d \to +\infty} \frac{1}{f_1(\Delta)} \frac{\mathcal{C}_k}{g_{k + 1}(\Delta)} &\ge 1 - \beta.
	\end{align*}

	If $\beta > 0$, this implies that \[ \left( 1 - \frac{1}{f_1(\Delta)} \frac{\mathcal{C}_{k - 1}}{g_k(\Delta)} \right)^{\frac{k + 1}{k}} = \Theta \left( \frac{1}{f_1(\Delta)^{ \frac{1}{k} }} \frac{g_{k + 1}(\Delta)}{g_k(\Delta)^{ \frac{k + 1}{k} }} \right). \] 
	
	We can take a closer look at this asymptotic relation. Since $d = o(g_1(\Delta))$ (e.g. when $d^2 = o(g_2(\Delta))$), the dominant terms translate the condition above to \[ \left( 1 - \frac{1}{2} \frac{ k(k + 1) \frac{g_{k + 1}(\Delta)}{g_k(\Delta)} }{f_1(\Delta)}  \right)^{\frac{k + 1}{k}} = \Theta \left( \frac{1}{f_1(\Delta)^{ \frac{1}{k} }} \frac{g_{k + 1}(\Delta)}{g_k(\Delta)^{ \frac{k + 1}{k} }} \right). \]
	
	Since $k$ is a constant and the terms involved are nonnegative, the identity 
	\begin{align*}
		\frac{1}{f_1(\Delta)^{ \frac{1}{k} }} \frac{g_{k + 1}(\Delta)}{g_k(\Delta)^{ \frac{k + 1}{k} }} &= \left( \frac{1}{f_1(\Delta)^{ \frac{1}{k + 1} }} \frac{g_{k + 1}(\Delta)^{ \frac{k}{k + 1} }}{g_k(\Delta)} \right)^{ \frac{k + 1}{k} } \\
		&= \left( \frac{1}{f_1(\Delta)^{ \frac{1}{k + 1} } g_{k + 1}(\Delta)^{ \frac{1}{k + 1} } } \frac{g_{k + 1}(\Delta)}{g_k(\Delta)} \right)^{ \frac{k + 1}{k} } \\
		&= \left( \frac{ f_1(\Delta)^{ \frac{k}{k + 1} } }{g_{k + 1}(\Delta)^{ \frac{1}{k + 1} }} \frac{1}{f_1(\Delta)} \frac{g_{k + 1}(\Delta)}{g_k(\Delta)} \right)^{ \frac{k + 1}{k} }
	\end{align*}

	would imply that \[  \left( 1 - \frac{1}{2} \frac{ k(k + 1) \frac{g_{k + 1}(\Delta)}{g_k(\Delta)} }{f_1(\Delta)} \right)^{ \frac{k + 1}{k} } = \Theta \left( \left( \frac{ f_1(\Delta)^{ \frac{k}{k + 1} } }{g_{k + 1}(\Delta)^{ \frac{1}{k + 1} }} \frac{1}{f_1(\Delta)} \frac{g_{k + 1}(\Delta)}{g_k(\Delta)} \right)^{ \frac{k + 1}{k} } \right). \] 
	
	This depends on a comparison between $f_1(\Delta)^k$ and $g_{k + 1}(\Delta)$. Note that \[ f_1(\Delta) = \binom{d + 1}{2} + d g_1(\Delta) + g_2(\Delta) \] and $g_2(\Delta) \ge \frac{1}{3} g_1(\Delta)$ since $\gamma(\Delta) \ge 0$ (Corollary \ref{gamgratrest}). \\
	
	Since \[ 0 < \frac{1}{f_1(\Delta)} \frac{\mathcal{C}_k}{g_{k + 1}(\Delta)} \le 1, \] the limiting inequality \[ \lim_{d \to +\infty} \frac{1}{f_1(\Delta)} \frac{\mathcal{C}_k}{g_{k + 1}(\Delta)} \ge 1 - \beta \] is only nontrivial if $\beta \le 1$. When $\beta > 0$, the lower bound $1 - \beta$ is strictly positive and we are in the setting of Part 3 of Corollary \ref{asympgratinit} with $k$ replaced by $k + 1$. Setting \[ \sigma \coloneq \lim_{d \to +\infty} \frac{1}{f_1(\Delta)} \frac{\mathcal{C}_k}{g_{k + 1}(\Delta)}, \] the inequality is \[ \sigma > 1 - \beta. \] Replacing $\mu$ by $\sigma$ in Part 3 of Corollary \ref{asympgratinit}, we have \[ \lim_{d \to +\infty} \frac{ \frac{g_{k + 2}(\Delta)}{g_{k + 1}(\Delta)} }{f_1(\Delta)} = \frac{2 \sigma}{(k + 1)(k + 2)} \] since $d = o(g_1(\Delta))$ (e.g. when $d^2 = o(g_2(\Delta))$). \\

	Suppose that $\beta = 0$. Then, we have \[ \left( 1 - \frac{1}{f_1(\Delta)} \frac{\mathcal{C}_{k - 1}}{g_k(\Delta)} \right)^{\frac{k + 1}{k}} = o \left( \frac{1}{f_1(\Delta)^{ \frac{1}{k} }} \frac{g_{k + 1}(\Delta)}{g_k(\Delta)^{ \frac{k + 1}{k} }} \right). \]
	
	Earlier identities yield the following simplification \[ \left( 1 - \frac{1}{2} \frac{ k(k + 1) \frac{g_{k + 1}(\Delta)}{g_k(\Delta)} }{f_1(\Delta)} \right)^{ \frac{k + 1}{k} } = o \left( \left( \frac{ f_1(\Delta)^{ \frac{k}{k + 1} } }{g_{k + 1}(\Delta)^{ \frac{1}{k + 1} }} \frac{1}{f_1(\Delta)} \frac{g_{k + 1}(\Delta)}{g_k(\Delta)} \right)^{ \frac{k + 1}{k} } \right) \] since $d = o(g_1(\Delta))$ (e.g. when $d^2 = o(g_2(\Delta))$). \\

	The induced lower bound would then be $\sigma \ge 1$. Since $0 \le \sigma \le 1$, we have \[ \sigma \coloneq \lim_{d \to +\infty} \frac{1}{f_1(\Delta)} \frac{\mathcal{C}_k}{g_{k + 1}(\Delta)} = 1. \]
	
	Then, the same reasoning as above using Part 2 of Corollary \ref{asympgratinit} implies that \[ \lim_{d \to +\infty} \frac{ \frac{g_{k + 2}(\Delta)}{g_{k + 1}(\Delta)} }{f_1(\Delta)} = \frac{2}{(k + 1)(k + 2)} \] since $d = o(g_1(\Delta))$ (e.g. when $d^2 = o(g_2(\Delta))$). \\
 	
	The possible nontrivial limiting behaviors discussed above are listed below. \\
	
	\begin{cor} \textbf{(Nontriviality criteria and an additional trivial case for individual link condition limiting inequalities) \\} \label{constlinknontriv}
		Let $\Delta$ be a $(d - 1)$-dimensional simplicial sphere such that $\gamma(\Delta) \ge 0$. Fix $1 \le k \le \frac{d}{2} - 2$. Suppose $d = (o(g_1(\Delta)))$ (e.g. when $d^2 = o(g_2(\Delta))$). 
		
		\begin{enumerate}
			\item The limiting inequality of the global/averaged lower bound induced by the link condition (Corollary \ref{diffcush}) behaves nontrivially if the following conditions are satisfied: 
				\begin{enumerate}
					\item We have \[ \lim_{d \to +\infty} \frac{ \frac{g_{k + 1}(\Delta)}{g_k(\Delta)} }{f_1(\Delta)} = \frac{2}{k(k + 1)}.  \] Equivalently, we have \[ \lim_{d \to +\infty} \frac{ \frac{g_{k + 1}(\Delta)}{g_k(\Delta)} }{d g_1(\Delta)} = \frac{2}{k(k + 1)}  \] or \[ \lim_{d \to +\infty} \frac{ \frac{g_{k + 1}(\Delta)}{g_k(\Delta)} }{g_2(\Delta)} = \frac{2}{k(k + 1)}.  \]

					\item We have \[ \left( 1 - \frac{1}{2} \frac{ k(k + 1) \frac{g_{k + 1}(\Delta)}{g_k(\Delta)} }{f_1(\Delta)} \right)^{ \frac{k + 1}{k} } = \Theta \left( \frac{1}{f_1(\Delta)^{ \frac{1}{k} }} \frac{g_{k + 1}(\Delta)}{g_k(\Delta)^{ \frac{k + 1}{k} }} \right) = \Theta \left( \left( \frac{ f_1(\Delta)^{ \frac{k}{k + 1} } }{g_{k + 1}(\Delta)^{ \frac{1}{k + 1} }} \frac{1}{f_1(\Delta)} \frac{g_{k + 1}(\Delta)}{g_k(\Delta)} \right)^{ \frac{k + 1}{k} } \right) \] or \[ \left( 1 - \frac{1}{2} \frac{ k(k + 1) \frac{g_{k + 1}(\Delta)}{g_k(\Delta)} }{f_1(\Delta)} \right)^{ \frac{k + 1}{k} } = o \left( \frac{1}{f_1(\Delta)^{ \frac{1}{k} }} \frac{g_{k + 1}(\Delta)}{g_k(\Delta)^{ \frac{k + 1}{k} }} \right) = o \left( \left( \frac{ f_1(\Delta)^{ \frac{k}{k + 1} } }{g_{k + 1}(\Delta)^{ \frac{1}{k + 1} }} \frac{1}{f_1(\Delta)} \frac{g_{k + 1}(\Delta)}{g_k(\Delta)} \right)^{ \frac{k + 1}{k} } \right) \]
					
					with the quotient of the left hand side by the right hand side approaching a constant $0 \le \beta \le 1$ as $d \to +\infty$. \\
					
					This depends on the size of $f_1(\Delta)^k$ relative to $g_{k + 1}(\Delta)$. Recall that \[ f_1(\Delta) = \binom{d + 1}{2} + d g_1(\Delta) + g_2(\Delta) \] and $g_2(\Delta) \ge \frac{1}{3} g_1(\Delta)$ since $\gamma(\Delta) \ge 0$ (Corollary \ref{gamgratrest}). Comparing this to the upper bound $g_2(\Delta) \le g_1(\Delta)^{\langle 1 \rangle} = \binom{g_1(\Delta) + 1}{2}$, the asymptotic properties above may be an indication that $g_2(\Delta)$ is not very large compared to $g_1(\Delta)$. \\
					
				\end{enumerate}
			
			\item If we are in the setting of Part 1 with condition 1b) replaced by \[ \frac{1}{f_1(\Delta)^{ \frac{1}{k} }} \frac{g_{k + 1}(\Delta)}{g_k(\Delta)^{ \frac{k + 1}{k} }} = \left( \frac{ f_1(\Delta)^{ \frac{k}{k + 1} } }{g_{k + 1}(\Delta)^{ \frac{1}{k + 1} }} \frac{1}{f_1(\Delta)} \frac{g_{k + 1}(\Delta)}{g_k(\Delta)} \right)^{ \frac{k + 1}{k} } = o \left( \left( 1 - \frac{1}{2} \frac{ k(k + 1) \frac{g_{k + 1}(\Delta)}{g_k(\Delta)} }{f_1(\Delta)} \right)^{ \frac{k + 1}{k} } \right), \] the limiting inequality holds trivially. \\ 
			
			\item The resulting limiting inequalities in Part 1 are of the form $\sigma \ge 1 - \beta$, where $\sigma$ and $\beta$ satisfy the following properties: \\
				\begin{enumerate}
					\item We have \[ \lim_{d \to +\infty} \frac{ \frac{g_{k + 2}(\Delta)}{g_{k + 1}(\Delta)} }{f_1(\Delta)} = \frac{2 \sigma}{(k + 1)(k + 2)}. \]
					
					Equivalently, we have \[ \lim_{d \to +\infty} \frac{ \frac{g_{k + 2}(\Delta)}{g_{k + 1}(\Delta)} }{d g_1(\Delta)} = \frac{2 \sigma}{(k + 1)(k + 2)} \] or \[ \lim_{d \to +\infty} \frac{ \frac{g_{k + 2}(\Delta)}{g_{k + 1}(\Delta)} }{g_2(\Delta)} = \frac{2 \sigma}{(k + 1)(k + 2)}. \]
					
					\item We have \[ \lim_{d \to +\infty} \frac{ \left( 1 - \frac{1}{2} \frac{ k(k + 1) \frac{g_{k + 1}(\Delta)}{g_k(\Delta)} }{f_1(\Delta)} \right)^{ \frac{k + 1}{k} } }{ \frac{1}{f_1(\Delta)^{ \frac{1}{k} }} \frac{g_{k + 1}(\Delta)}{g_k(\Delta)^{ \frac{k + 1}{k} }} } = \lim_{d \to +\infty} \frac{ \left( 1 - \frac{1}{2} \frac{ k(k + 1) \frac{g_{k + 1}(\Delta)}{g_k(\Delta)} }{f_1(\Delta)} \right)^{ \frac{k + 1}{k} }   }{ \left( \frac{ f_1(\Delta)^{ \frac{k}{k + 1} } }{g_{k + 1}(\Delta)^{ \frac{1}{k + 1} }} \frac{1}{f_1(\Delta)} \frac{g_{k + 1}(\Delta)}{g_k(\Delta)} \right)^{ \frac{k + 1}{k} }  } = \beta. \]
					
					As mentioned in Part 1b), this seems to depend on the relationship between the sizes of $f_1(\Delta)^k$ and $g_{k + 1}(\Delta)$. \\
					
				\end{enumerate}

			Since $0 \le \sigma \le 1$, this bound is only nontrivial if $0 \le \beta \le 1$. If $\beta = 0$, we have $\sigma = 1$. \\
			
		\end{enumerate}
		
	\end{cor}
	
	This gives rise to a general interpretation of situations where the limiting link condition inequality for $d \gg 0$ from Corollary \ref{diffcush} behaves nontrivially. \\
	
	\begin{cor} \textbf{(Nontrivial link condition inequality asymptotic interpretation and gamma vector implications for individual inequalities) \\} \label{nontrivlinkgam}
		Let $\Delta$ be a $(d - 1)$-dimensional simplicial sphere satisfying the link condition such that $\gamma(\Delta) \ge 0$. Fix $1 \le k \le \frac{d}{2} - 2$.
		
		\begin{enumerate}
			\item The upper bound $\sigma$ in the induced limiting link condition inequality $\sigma \ge 1 - \beta$ as $d \to +\infty$ can be interpreted asymptotically (with respect to $d$) as the quotient of $\frac{g_{k + 2}(\Delta)}{g_{k + 1}(\Delta)}$ by $\frac{g_{k + 1}(\Delta)}{g_k(\Delta)}$ since we are dividing by the constant $k + 2$ (Part 1 and Part 3 of Corollary \ref{constlinknontriv}). In the context of Remark \ref{indexratchange}, this would be a constraint on the rate of decrease of the terms $\frac{g_{m + 1}(\Delta)}{g_m(\Delta)}$. \\
			
			\item The lower bound $1 - \beta$ in Part 1 is a number between 0 and 1. It increases when the inequality $g_{k + 1}(\Delta) \le g_k(\Delta)^{ \frac{k + 1}{k} }$ ($M$-vector condition on $g(\Delta)$ from Part 2c) of Theorem \ref{gconjsphere} and Corollary \ref{macpowasymp}) is closer to equality or the number of edges $f_1(\Delta)$ decreases. The lower bound also increases if the size of $f_1(\Delta)^k$ relative to $g_{k + 1}(\Delta)$ increases. \\
			
			\item While the lower bound for the quotient of $\frac{g_{k + 2}(\Delta)}{g_{k + 1}(\Delta)}$ by $\frac{g_{k + 1}(\Delta)}{g_k(\Delta)}$ from Corollary \ref{diffcush} can be translated to one for $\gamma_{k + 2}(\Delta)$, the gamma vector components $\gamma_i(\Delta)$ affected are mainly those that are asymptotically dominant in $g_i(\Delta)$. For example, having $\gamma_{\ell + 1}(\Delta) \ge (d + 1) g_\ell(\Delta)$ (e.g. when $g_{\ell + 1}(\Delta) \ge 2(d + 1) g_\ell(\Delta)$) would imply that $\gamma_{\ell + 1}(\Delta)$ must be the asymptotically dominant term in $g_{\ell + 1}(\Delta)$ (Corollary \ref{gamgratrest}). If this holds for all $\ell \in \{ k, k + 1, k + 2 \}$, we can take the inequality to be an asymptotic lower bound on the quotient of $\frac{\gamma_{k + 2}(\Delta)}{\gamma_{k + 1}(\Delta)}$ by $\frac{\gamma_{k + 1}(\Delta)}{\gamma_k(\Delta)}$.  \\

		\end{enumerate}
	\end{cor}

	\color{black}

	\section{Generalized positivity properties from orthogonal polynomials and lattice paths} \label{genorthlat}

	Recall from Proposition \ref{gamtohgvect} that matrices transforming the gamma vectors to the first half of the $h$-vector or the $g$-vector are totally nonnegative via a Lindstr\"om--Gessel--Viennot argument. Such arguments and connections to lattice paths in general have also been involved in recent gamma positivity-related work such as \cite{BV} and \cite{EHR}. In fact, there is a common structure underlying this argument and the inverted Chebyshev expansion of the gamma polynomial from Theorem \ref{gamchebinv} which we have used in earlier positivity and structural results involving the gamma vector (e.g. Corollary 1.35 on p. 25 of \cite{Pflag} and Corollary 1.10 on p. 9 of \cite{Pcheb}). \\ 
	
	More specifically, this is the general connection between orthogonal polynomials and the combinatorics of lattice paths (e.g. see \cite{V1} and \cite{V2}). Viewing the inverted Chebyshev expansion from this perspective, we have an interpretation of the gamma vector components in terms of monomer/dimer covers with colored monomers (Proposition \ref{gamorthpov}). In Remark \ref{dimtchebsubdiv}, we put this in the context of connections between dimer covers and repeated (stellar) edge subdivisions. This follows from compatibility of Tchebyshev subdivisions from \cite{Het} and \cite{HetB} with modified $f$-polynomials, which were the main connection between the inverted Chebyshev expansion and positivity/geometric properties in the results mentioned above. \\
	
	Finally, we show that combinatorial properties of lattice paths give orthogonal polynomial generalizations of the implication of unimodality by gamma positivity (Proposition \ref{formalgpos}) and realizability properties of nonnegative gamma vectors analogous to those in Section \ref{gamposout} for simplicial spheres (Proposition \ref{geninvgamsphere}) and Cohen--Macaulay simplicial complexes (Remark \ref{cmorthgen}). While we will go through realizations by the entire space of such objects as in Section \ref{gennonnegrealiz} explicitly, we will sketch how one would obtain analogues of the results of Section \ref{linkcondimp} (Remark \ref{linkapproxorth}) since it will also use largely the same reasoning and specific growth rates will depend on the initial parameters involved. \\

	By the inverted Chebyshev expansion from Theorem \ref{gamchebinv}, we have that \[ (2w)^{ \frac{d}{2} } \gamma \left( \frac{1}{2w} \right) =  g(u) \coloneq h_{\frac{d}{2}} + 2 \sum_{j = 1}^{ \frac{d}{2} } h_{ \frac{d}{2} - j } T_j(w - 1). \]

	In order to relate this expression to lattice paths, we will normalize it as a linear combination of unitary orthogonal polynomials (Th\'eor\`eme 9 on I-4 and I-5 of \cite{V2} and p. 142 of \cite{V1}). Since $\widehat{T}_0(x) = T_0(x) = 1$ and $\widehat{T}_m(x) = \frac{1}{2^{m - 1}} T_m(x)$ if $m \ge 1$ for the unitary orthogonal polynomials $\widehat{T}_m(x)$ corresponding to Chebyshev polynomials of the first kind $T_m(x)$ (II-4 from \cite{V2}), we have \[ (2w)^{ \frac{d}{2} } \gamma \left( \frac{1}{2w} \right) = \sum_{j = 0}^{\frac{d}{2}} 2^j h_{\frac{d}{2} - j} \widehat{T}_j(w - 1). \]
	
	Note that the difference in the factor of 2 in the original inverted Chebyshev expansion corresponds to the difference between the index 0 and positive index relations between the Chebyshev polynomials of the first kind and the unitary orthogonal polynomials corresponding to them. \\
	
	From the perspective of unitary orthogonal polynomials $P_r(x)$ satisfying the recursion \[ P_{m + 1}(x) = (x - b_m) P_m(x) - \lambda_m P_{m - 1}(x) \] for all $m \ge 1$ (I-5 from \cite{V2} and p. 142 of \cite{V1}), taking $P_r(x) = \widehat{T}_r(x - 1)$ means setting \[ b_m = 1 \] for all $m \ge 0$ and 
	\[ \lambda_m = \begin{cases}
		\frac{1}{4} & \text{ if $m \ge 2$} \\
		\frac{1}{2} & \text{ if $m = 1$}
	\end{cases} \]
	
	by II-4 from \cite{V2}. \\
	
	The matrix corresponding to the linear transformation $F : k[x] \longrightarrow k[x]$ defined by $x^m \mapsto P_m(x)$ acts on row vectors as \[ z \mapsto z \cdot P, \] where $z$ is a row vector listing coefficients in order of increasing degrees and $P$ is the matrix whose $(m, i)$ entry is the coefficient of $x^i$ in $P_m(x)$ (see III-1 from \cite{V2}). \\
	
	In our setting, we have \[ P  : ( h_{\frac{d}{2}}, 2 h_{\frac{d}{2} - 1}, \ldots, 2^{\frac{d}{2} - 1} h_1, 2^{\frac{d}{2}} h_0 ) \mapsto ( \gamma_{\frac{d}{2}}, 2 \gamma_{\frac{d}{2} - 1}, \ldots, 2^{\frac{d}{2} - 1} \gamma_1, 2^{ \frac{d}{2} } \gamma_0 ) \]
	
	and \[ P^{-1} : ( \gamma_{\frac{d}{2}}, 2 \gamma_{\frac{d}{2} - 1}, \ldots, 2^{\frac{d}{2} - 1} \gamma_1, 2^{ \frac{d}{2} } \gamma_0 ) \mapsto ( h_{\frac{d}{2}}, 2 h_{\frac{d}{2} - 1}, \ldots, 2^{\frac{d}{2} - 1} h_1, 2^{\frac{d}{2}} h_0 ). \]
	
	The latter is analogous to the matrix \[ A = \left[ \binom{d - 2j}{i - j} \right]_{0 \le i, j \le \frac{d}{2}} \] from Proposition \ref{gamtohgvect}
	such $A \gamma = (h_0, \ldots, h_{\lfloor \frac{d}{2} \rfloor})$. \\
	
	The connection to Lindstr\"om--Gessel--Viennot methods involving lattice paths mentioned in Proposition \ref{gamtohgvect} fits into more general combinatorial properties of inverses $P^{-1}$ of such matrices $P$ (see Th\'eor\`eme 1 on III-2 from \cite{V2} and p. 143 of \cite{V1}). It is known that $P^{-1} = (\mu_{N, k})$, where \[ \mu_{N, k} = \sum_{\omega \in M_{N, k}} v(\omega), \] where $M_{N, k}$ consists of the Motzkin paths of length $N$ starting from level 0 and ending at level $k$ and $v(\omega)$ is the sum of the valuations of the paths. They are given by the products of the valuations of the edges. The valuation of a step starting at level $m$ is determined by the $b_m$ and $\lambda_m$ as follows:
	
		\begin{itemize}
			\item A northeast (NE) step has valuation 1. 
			
			\item An east (E) step has valuation $b_m$.
			
			\item A southeast (SE) step has valuation $\lambda_m$.
		\end{itemize}   
	
	We note that $P^{-1} = (\mu_{N, k})$ has positive entries in our setting since $b_m > 0$ and $\lambda_m > 0$ for all $m \ge 1$. In the material below, we will use the term ``weight'' in place of ``valuation'' for convenience since we will discuss weighted sums. \\

	Also, there is an interpretation of degree $r$ coefficients of unitary orthogonal polynomials $P_m(x)$ as sums of weighted disjoint monomer/dimer covers of $[m] = \{ 1, \ldots, m \}$ with $r$ missing elements, where monomers $\{ k \}$ have weight $-b_{k - 1}$ and dimers $\{ k, k + 1 \}$ have weight $-\lambda_k$ (Lemme 13 on I-9 and I-10 of \cite{V2} and p. 150 -- 151 of \cite{V1}). \\
	
	In our context, these properties can be translated as follows:

	\begin{prop} \label{gamorthpov}
		Suppose that $h = h(\Delta)$ is the $h$-vector of a $(d - 1)$-dimensional shellable simplicial complex $\Delta$ (equivalently that of a Cohen--Macaulay simplicial complex). Assume that $d$ is even. \\
		\begin{enumerate}
			\item The truncated $h$-vector $(h_0, \ldots, h_{\frac{d}{2}})$ is the $h$-vector of a $\left( \frac{d}{2} - 1 \right)$-dimensional partitionable simplicial complex $\widetilde{\Delta}$. \\
			
			\item For each induced partionining $[r(F_i), \widetilde{F}_i]$ of $\widetilde{\Delta}$ from Part 1, let $d_i$ be the ``minimal dimer'' corresponding to the pair of faces $r(F_i)$ and its union with the minimal vertex of $\widetilde{F}_i \setminus r(F_i)$ with respect to the vertex ordering from the proof of Part 1. We take the $r(F_i)$ from $\Delta$ with $|r(F_i)| \le \frac{d}{2}$ and $|\widetilde{F}_i| = \frac{d}{2}$ for all $i$. \\
			
			Given $0 \le \ell \le \frac{d}{2}$, we have \[ \gamma_\ell = A_\ell - 2 B_\ell, \] where $A_\ell$ and $B_\ell$ are the differences between the number of disjoint monomer/dimer covers using an even number of pieces and those using an odd number of pieces among the following covers: \\
			
				\begin{itemize}
					\item In $A_\ell$, we consider disjoint monomer/dimer covers of intervals $[r(F_i), \widetilde{F}_i]$ with $|r(F_i)| \le \ell + 1$ that do \emph{not} use the minimal dimer $d_i$ and have $\frac{d}{2} - \ell$ missing elements. Each monomer is assigned to one of 2 colors and the dimers are left uncolored. \\
					
					\item In $B_\ell$, we consider disjoint monomer/dimer covers of intervals $[r(F_i), \widetilde{F}_i] \setminus d_i$ with $|r(F_i)| \le \ell - 1$ that have $\frac{d}{2} - \ell$ missing elements. Again, each monomer is assigned to one of 2 colors and the dimers are left uncolored. \\
				\end{itemize}  
			
			If $\ell = \frac{d}{2}$, we consider the constant term of the inverted Chebyshev expansion. This means multiplying $2^j$ by the constants $\widehat{T}_j(-1) = (-1)^j \frac{1}{2^{j - 1}}$ for $j \ge 1$ and $\widehat{T}_0(-1) = 1$. The identities follow from using induction and the defining recursion of this unitary orthogonal polynomial. \\

			\item Consider the unitary orthogonal polynomials $P_m(x)$ with $b_i > 0$ and $\lambda_j > 0$ for all $0 \le i \le N$ and $1 \le j \le N$. The image of nonnegative row vector where the $N^{\text{th}}$ component is equal to $2^N$ under the inverse map $P^{-1}$ discussed above gives an analogue of the implication $\gamma(\Delta) \ge 0 \Longrightarrow h_i(\Delta) \ge \binom{d}{i}$ from Proposition \ref{gamtohgvect}. The corresponding lower bound of the index $\ell$ output/$\ell^{\text{th}}$ coordinate would be $q_\ell \ge 2^N \mu_{N, \ell}$ and the analogous lower bound of the associated ``formal $h$-vector'' would be $h_k \ge 2^k \mu_{N, N - k}$. \\
		\end{enumerate}
	\end{prop}
	
	\begin{proof}
		\begin{enumerate}
			\item Start with a fixed ordering of the vertices of $\Delta$ and an ordering $F_1 < \cdots < F_s$ of the facets $F_i \in \Delta$ yielding a shelling of $\Delta$. Let $\Delta = [r(F_1), F_1] \cup \cdots \cup [r(F_s), F_s]$ be the induced partitioning of $\Delta$ (p. 79 -- 80 of \cite{St}). In each interval $[r(F_i), F_i]$, take the interval $[r(F_i), \widetilde{F}_i]$, where $\widetilde{F}_i$ is formed by filling vertices of this interval in increasing order until we form a face with $\frac{d}{2}$ vertices. Since we already started with a disjoint collection of vertices, the intervals $[r(F_i), \widetilde{F}_i]$ form a partitioning of the simplicial complex $\widetilde{\Delta}$ whose facets are these $\left( \frac{d}{2} - 1 \right)$-dimensional faces $\widetilde{F}_i$. \\
			
			\item Recall that \[ P  : ( h_{\frac{d}{2}}, 2 h_{\frac{d}{2} - 1}, \ldots, 2^{\frac{d}{2} - 1} h_1, 2^{\frac{d}{2}} h_0 ) \mapsto ( \gamma_{\frac{d}{2}}, 2 \gamma_{\frac{d}{2} - 1}, \ldots, 2^{\frac{d}{2} - 1} \gamma_1, 2^{ \frac{d}{2} } \gamma_0 ) \] 
			
			via the $P_j(w) = \widehat{T}_j(w - 1)$ in the inverted Chebyshev expansion \[ (2w)^{ \frac{d}{2} } \gamma \left( \frac{1}{2w} \right) = \sum_{j = 0}^{\frac{d}{2}} 2^j h_{\frac{d}{2} - j} \widehat{T}_j(w - 1). \] In each $P_j(w)$, we are considering a disjoint monomer/dimer cover of the interval $[0, j - 1]$. Since each weight $h_{\frac{d}{2} - j}$ is the number of partitioning intervals $[r(F_i), \widetilde{F}_i]$ with $|r(F_i)| = \frac{d}{2} - j$, they amount to disjoint monomer/dimer covers of such intervals with elements of $[0, j - 1]$ corresponding to faces of $\widetilde{\Delta}$ lying in this interval. We will take the elements of $[0, j - 1]$ listed in increasing order to correspond to the faces of $\widetilde{\Delta}$ in the interval $[r(F_i), \widetilde{F}_i]$ with vertices added in increasing order with respect to the vertex ordering in the proof of Part 1. For example, the element $0$ corresponds to the face $r(F_i) \in \widetilde{\Delta}$. The statement about the sign follows from the definition of the signed weights from earlier. \\

			The degree $\frac{d}{2} - \ell$ term $2^{\frac{d}{2} - \ell} \gamma_\ell$ comes from disjoint monomer/dimer covers $\beta$ of such size $j$ intervals (called pavings on I-9 and I-10 of \cite{V2}) with $\frac{d}{2} - \ell$ missing elements. In the sum, this means that the relevant indices are those where $\frac{d}{2} - \ell \le \frac{d}{2} - \left( \frac{d}{2} - j \right) + 1$. This is equivalent to $|r(F_i)| = \frac{d}{2} - j \le \ell + 1$. In the context of these covers, the only cases where $\lambda_1 = \frac{1}{2}$ is used comes from covers that use the dimer $\{ 0, 1 \}$. For each interval $[r(F_i), \widetilde{F}_i]$, this dimer corresponds to the pair of faces $r(F_i)$ and the face formed by adding the smallest vertex of $\widetilde{F}_i \setminus r(F_i)$ to $r(F_i)$ with respect to the ordering of the vertices of $\Delta$ from the proof of Part 1. We define this pair of faces be the ``minimal dimer'' $d_i$ associated to the interval $[r(F_i), \widetilde{F}_i]$. \\
			
			We will split each disjoint monomer/dimer cover of the intervals $[r(F_i), \widetilde{F}_i]$ into those that use the minimal dimer $d_i$ and those that do not. Consider the monomer/dimer covers that do \emph{not} use the minimal dimer $d_i$. We will add terms of degree $\frac{d}{2} - \ell$ from each index $j$ of the sum \[ (2w)^{ \frac{d}{2} } \gamma \left( \frac{1}{2w} \right) = \sum_{j = 0}^{\frac{d}{2}} 2^j h_{\frac{d}{2} - j} \widehat{T}_j(w - 1). \]

			In the notation of the sum above, we have $|r(F_i)| = \frac{d}{2} - j$. The condition $\frac{d}{2} - \ell \le \frac{d}{2} - \left( \frac{d}{2} - j \right) + 1$ is equivalent to $\frac{d}{2} - j \le \ell + 1$ and we consider the intervals $[r(F_i), \widetilde{F}_i]$ such that $|r(F_i)| \le \ell + 1$. Since every dimer $\{ k, k + 1 \}$ of $[0, j - 1]$ corresponding to a dimer of such a monomer/dimer cover has $k \ge 1$, we have $\lambda_{k + 1} = \frac{1}{4}$. Let $m(\beta)$ be the number of monomers and $d(\beta)$ be the number of dimers. Since $\lambda_{k + 1} = \frac{1}{4}$ and $b_k = 1$ for all $k \ge 1$, we will add terms of the form $(-1)^{m(\beta) + d(\beta)} 2^{j - 2d(\beta)}$. By definition, we have that $m(\beta) + 2 d(\beta) = j - \left( \frac{d}{2} - \ell \right)$. This implies that \[ j - 2 d(\beta) = m(\beta) + \left( \frac{d}{2} - \ell \right). \]

			After factoring out $2^{\frac{d}{2} - \ell}$, the terms from disjoint monomer/dimer covers of $[r(F_i), \widetilde{F}_i]$ that do \emph{not} use the minimal dimer $d_i$ give a weighted alternating sum where the weights of the monomers are equal to 2 and the weights of the dimers are equal to 1. We can also take this to mean that the monomers are each assigned one of 2 colors (e.g. red or blue) while the dimers are left uncolored. Recall that the sign comes from the number of pieces used. Thus, the sum of such terms is the difference between the number of such monomer/dimer covers with colored monomers using an even number of pieces and those using an odd number of pieces. \\

			Next, we consider the disjoint monomer/dimer covers of $[r(F_i), \widetilde{F}_i]$ that \emph{do} use the minimal dimer $d_i$. Recall that the degree comes from the number of \emph{missing} elements. In the induced disjoint monomer/dimer cover of $[r(F_i), \widetilde{F}_i] \setminus d_i$ formed by removing $d_i$ from the cover, the number of missing elements still stays the same. On the other hand, we would use $j - 2$ instead of $j$ in the steps above since the total size of the interval to cover is reduced by 2. This means that $\frac{d}{2} - 2 - |r(F_i)| + 1 \ge \frac{d}{2} - \ell$, which is equivalent to $|r(F_i)| \le \ell - 1$. Since the weight of a particular disjoint monomer/dimer cover is the product of the weights of the pieces used and $\lambda_{k + 1} = \frac{1}{2}$ only if $k = 0$, the use of $j - 2$ in place of $j$ for these covers also implies that the sum of the terms involved is $-\frac{1}{2} \cdot 2^2 \cdot 2^{ \frac{d}{2} - \ell } = -2 \cdot 2^{ \frac{d}{2} - \ell }$ multiplied by an alternating sum analogous to the previous case for these induced disjoint monomer/dimer covers of $[r(F_i), \widetilde{F}_i] \setminus d_i$. More specifically, this alternating sum is the difference between the number of induced monomer/dimer covers with red/blue colored monomers that use an even number of pieces and those using an odd number of pieces.   \\

			\item In this part, we consider the inverse map \[ P^{-1} : ( \gamma_{\frac{d}{2}}, 2 \gamma_{\frac{d}{2} - 1}, \ldots, 2^{\frac{d}{2} - 1} \gamma_1, 2^{ \frac{d}{2} } \gamma_0 ) \mapsto ( h_{\frac{d}{2}}, 2 h_{\frac{d}{2} - 1}, \ldots, 2^{\frac{d}{2} - 1} h_1, 2^{\frac{d}{2}} h_0 ) \] where the $P_j(w)$ from the inverted Chebyshev expansion are replaced by arbitrary unitary orthogonal polynomials with $b_k > 0$ and $\lambda_k > 0$ for all $k \ge 1$. The assumption $z_N = 2^N$ is analogous to $\gamma_0 = 1$ and we are replacing $\frac{d}{2}$ with $N$. \\
			
			The implication $\gamma(\Delta) \ge 0 \Longrightarrow h_i(\Delta) \ge \binom{d}{i}$ from Proposition \ref{gamtohgvect} comes from the expansion \[ h_i = a_{i, 0} \gamma_0 + a_{i, 1} \gamma_1 + \ldots + a_{i, i - 1} \gamma_{i - 1} + \gamma_i  \] since $a_{i, j} = \binom{d - 2j}{i - j}$, $\gamma_0 = 1$, and $\gamma_j \ge 0$ for all $1 \le j \le \frac{d}{2}$. \\
			
			The earlier discussion implies that the map $z \mapsto z \cdot P^{-1}$ sends a row vector $z = (z_0, \ldots, z_N)$ to the vector $q = (q_0, \ldots, q_N)$ with 
			\begin{align*}
				q_\ell &= z_\ell \mu_{\ell, \ell} + z_{\ell + 1} \mu_{\ell + 1, \ell} + \ldots + z_N \mu_{N, \ell} \\
				&= z_\ell + z_{\ell + 1} \mu_{\ell + 1, \ell} + \ldots + z_N \mu_{N, \ell}
			\end{align*}
			
			since $\mu_{\ell, \ell} = 1$. Since we assumed that $b_i > 0$ and $\lambda_j > 0$ for all $0 \le i \le N$ and $1 \le j \le N$, we have that $\mu_{r, s} > 0$ for all $r \ge s$. Since $z_N$ corresponds to $2^N \gamma_0$ here, the assumption that $z_j \ge 0$ for all $0 \le j \le N$ implies that $q_\ell \ge z_N \mu_{N, \ell} = 2^N \mu_{N, \ell}$. \\
			
			Comparing this to the $h$-vector above, we have that $q_\ell = 2^\ell h_{ N - \ell }$ and $h_k = 2^{ - \left( N - k \right) } q_{ N - k }$. The induced lower bound would then be 
			
			\begin{align*}
				h_k &= 2^{ - \left( N - k \right) } q_{ N - k } \\
				&\ge 2^{ - \left( N - k \right) } \cdot 2^N \mu_{N, N - k} \\
				&= 2^k \mu_{N, N - k}. 
			\end{align*}
			
		\end{enumerate}

	\end{proof}

	\begin{rem} \textbf{(Dimer covers and repeated edge subdivisions) \\} 
		\label{dimtchebsubdiv}
		
		Combining the monomer/dimer cover interpretation of orthogonal polynomials with effects of repeated edge subdivisions on modified $f$-polynomials (Proposition 3.3 on p. 579 of \cite{Het}) suggests connections between dimer covers and such simplicial complex operations. Note that faces of subdivisions can be tracked using chains of faces of the original simplicial complex (e.g. barycentric subdivisions). \\
		
		More specifically, the edge subdivisions are Tchebyshev subdivisions $T(\Delta)$ (Definition 2 on p. 921 of \cite{HetB}, Definition 2.1 on p. 574 of \cite{Het}), which are successive (stellar) subdivisions of all the edges of the starting simplicial complex $\Delta$ in some order. While the underlying simplicial complex may depend on the order of the edges, the resulting change in $f$-vectors does \emph{not} depend on the choice of order on the edges (Theorem 3.1 on p. 577 of \cite{Het}). The formula involved implies that $T(F_\Delta(x)) = F_{T(\Delta)}(x)$ (Proposition 3.3 on p. 578 -- 579 of \cite{Het}), where $F_S(x) \coloneq f_S \left( \frac{x - 1}{2} \right)$ and $T : \mathbb{R}[x] \longrightarrow \mathbb{R}[x]$ sends $x^m \mapsto T_m(x)$ (Definition 1.1 on p. 572 of \cite{Het}). \\
		
		Substituting $S = T(\Delta)$, we have \[ [x^\ell] F_{T(\Delta)}(x) = (-1)^\ell \sum_{k = \ell}^d (-1)^k 2^{-k} \binom{k}{\ell} f_{k - 1}(T(\Delta)). \] We can think about this as a weighted alternating sum of pairs of faces $(F, G)$ of $\Delta$ with $F \subset G$ and $|G| = k$. \\
		
		The monomer/dimer covers will come from the $T(F_\Delta(x))$ side. Since $b_k = 0$ for all $k \ge 0$ in the unitary orthogonal polynomials $\widehat{T}_m(x)$ corresponding to Chebyshev polynomials of the first kind $T_m(x)$, we will only consider dimer covers since monomers have weight 0 in the monomer/dimer cover expansion (I-9 and I-10 of \cite{V2}). For the degree $\ell$ coefficient of $\widehat{T}_m(x)$, the number of missing elements $\ell$ in the associated interval $[0, m - 1]$ will have the same parity as $m$ since we are only using dimers. \\ 
		
		Let \[ \mathcal{D}_{m, \ell} \coloneq (-1)^{ \frac{m - \ell}{2} } (A_{m, \ell} + 2 B_{m, \ell}), \] where \[ A_{m, \ell} \coloneq \# \{ \text{dimer covers of $[0, m - 1]$ with $\ell$ missing elements that do \emph{not} use $\{ 0, 1 \}$}  \} \] and \[ B_{m, \ell} \coloneq  \# \{ \text{dimer covers of $[0, m - 1] \setminus \{ 0, 1 \} $ with $\ell$ missing elements}  \}. \]
		
		If $m \ge 2$, the reasoning in the proof of Part 2 of Proposition \ref{gamorthpov} implies that $2^m [x^\ell] \widehat{T}_m(x) = 2^\ell \mathcal{D}_{m, \ell}$. In the $m = 1$ case, there are no ``dimer covers'' apart from the empty one (not using any dimers) missing the one element $\{ 0 \}$ and $B_{m, \ell} = 0$ (i.e. only have a linear term with $\widehat{T}_1(x) = x$). Finally, we have $\widehat{T}_m(0) = 0$ if $m$ is odd and $\widetilde{T}_m(0) = (-1)^{ \frac{m}{2} } 2^{-(m - 1)}$ if $m$ is even and $m \ge 2$. When $\ell = 0$, we have $B_{m, 0} = 0$. Suppose that $\ell \ge 1$. Then, we have
		
		\begin{align*}
			[x^\ell] T(F_\Delta(x)) &= \sum_{m = \ell}^d ([x^m] F_\Delta(x)) ([x^\ell] T_m(x)) \\
			&= \sum_{m = \ell}^d (-1)^m \left( \sum_{k = m}^d (-1)^k 2^{-k} \binom{k}{m} f_{k - 1}(\Delta) \right) \cdot (2^{m - 1} [x^\ell] \widehat{T}_m(x)) \\
			&= \sum_{ \substack{ \ell \le m \le d \\ m \equiv \ell \pmod{2} } }  (-1)^m \left( \sum_{k = m}^d (-1)^k 2^{-k} \binom{k}{m} f_{k - 1}(\Delta) \right) \cdot (2^{\ell - 1} \mathcal{D}_{m, \ell}) \\
			&= (-1)^\ell  \sum_{ \substack{ \ell \le m \le d \\ m \equiv \ell \pmod{2} } } (-1)^m \cdot (-1)^m \left( \sum_{k = m}^d (-1)^k 2^{-k} \binom{k}{m} f_{k - 1}(\Delta) \right) \cdot (2^{\ell - 1} \mathcal{D}_{m, \ell}) \\
			&= (-1)^\ell \sum_{ \substack{ \ell \le m \le d \\ m \equiv \ell \pmod{2} } } \left( \sum_{k = m}^d (-1)^k 2^{-k} \binom{k}{m} f_{k - 1}(\Delta) \right) \cdot (2^{\ell - 1} \mathcal{D}_{m, \ell}) \\
			&= (-1)^\ell \sum_{k = \ell}^d (-1)^k 2^{-k} f_{k - 1}(\Delta) \left( \sum_{ \substack{ \ell \le m \le k \\ m \equiv \ell \pmod{2} } } \binom{k}{m} (2^{\ell - 1} \mathcal{D}_{m, \ell}) \right) 
		\end{align*}
		
		since $(-1)^\ell \cdot (-1)^m = 1$ if $m \equiv \ell \pmod{2}$. \\
		
		We can compare this to the terms $\binom{k}{\ell} f_{k - 1}(T(\Delta))$ in \[ [x^\ell] F_{T(\Delta)}(x) = (-1)^\ell \sum_{k = \ell}^d (-1)^k 2^{-k} \binom{k}{\ell} f_{k - 1}(T(\Delta)), \]

		If $\ell = 0$, we have \[ [x^0] F_{T(\Delta)}(x) = \sum_{k = \ell}^d (-1)^k 2^{-k} f_{k - 1}(T(\Delta)) \] and 
		
		\begin{align*}
			[x^0] T(F_\Delta(x)) &= \sum_{m = 0}^d ([x^m] F_\Delta(x)) ([x^0] T_m(x)) \\
			&= \sum_{m = 0}^d (-1)^m  \left( \sum_{k = m}^d (-1)^k 2^{-k} \binom{k}{m} f_{k - 1}(\Delta) \right) ([x^0] T_m(x)) \\
			&= \sum_{ \substack{ 0 \le m \le d \\ m \equiv 0 \pmod{2} } } (-1)^m  \left( \sum_{k = m}^d (-1)^k 2^{-k} \binom{k}{m} f_{k - 1}(\Delta) \right) (-1)^{ \frac{m}{2} } \\
			&=  \sum_{ \substack{ 0 \le m \le d \\ m \equiv 0 \pmod{2} } }   \left( \sum_{k = m}^d (-1)^k 2^{-k} \binom{k}{m} f_{k - 1}(\Delta) \right) (-1)^{ \frac{m}{2} } \\
			&= \sum_{k = 0}^d (-1)^k 2^{-k} f_{k - 1}(\Delta) \sum_{ \substack{ 0 \le m \le k \\ m \equiv 0 \pmod{2} } }  \binom{k}{m} (-1)^{ \frac{m}{2} } \\
			&= \sum_{k = 0}^d (-1)^k 2^{-k} f_{k - 1}(\Delta) \sum_{ \substack{ 0 \le m \le k \\ m \equiv 0 \pmod{2} } }  \binom{k}{m}  \mathcal{D}_{m, 0}.
		\end{align*}
		
	\end{rem}
	
	\color{black}

	Returning to the inverse map \[ P^{-1} : ( \gamma_{\frac{d}{2}}, 2 \gamma_{\frac{d}{2} - 1}, \ldots, 2^{\frac{d}{2} - 1} \gamma_1, \gamma_0 ) \mapsto ( h_{\frac{d}{2}}, 2 h_{\frac{d}{2} - 1}, \ldots, 2^{\frac{d}{2} - 1} h_1, 2^{\frac{d}{2}} h_0 ), \]
	
	from Proposition \ref{gamorthpov} and the discussion preceding it, we can find that nonnegativity of the entries of the analogous input vector $z$ in $z \mapsto P^{-1}$ implies that the output vectors have increasing entries. We are considering the truncation to $N \times N$ matrices. For example, $N = \frac{d}{2}$ in the setting of ``normalized gamma vectors''. The main recursion we will use is the following: \\
	
	\begin{lem} \label{pathwtrec}
		Suppose that $N \ge k$. The entries $\mu_{N, k}$ of $P^{-1}$ satisfy the recursion \[ \mu_{N, k} = \mu_{N - 1, k - 1} + b_k \mu_{N - 1, k} + \lambda_{k + 1} \mu_{N - 1, k + 1}, \] where we have $\mu_{r, s} = 0$ if $r < s$ or some index lies outside of $[0, N]$. If $b_i > 0$ and $\lambda_j > 0$ for all $0 \le i \le N$ and $1 \le j \le N$, this implies that $\mu_{N, k} \ge \mu_{N - 1, k}$ for all $N \ge 1$. Similarly, we have $\frac{\mu_{N, k}}{\mu_{N - 1, k}} \ge b_k$ if $\mu_{N - 1, k} \ne 0$ and $\frac{\mu_{N, k}}{\mu_{N - 1, k + 1}} \ge \lambda_{k + 1}$ if $\mu_{N - 1, k + 1} \ne 0$ in this setting. \\
	\end{lem} 
	
	\begin{proof}
		The recursion follows from partionining the Motzkin paths of length $N$ starting at 0 and ending at level $k$ according to the type of the last step taken. The first term is the case where the last step is the NE step. Similarly, the other two terms are from the cases where the last steps taken are E and SE respectively. Since the level reached after starting at level 0 and taking $r$ steps of a Motzkin path cannot be larger than $r$, we have $\mu_{r, s} = 0$ if $r < s$. \\
	\end{proof}
	
	This recursion implies some growth properties of these parameters that will be used in a later discussion of approximations. \\
	
	\begin{cor} \label{wtgrowth}
		Fix an initial collection of parameters $b_m$ and $\lambda_m$ for $0 \le m \le N$.
		\begin{enumerate}
			\item $\mu_{N, N - 1} = b_0 + b_1 + \ldots + b_{N - 1}$. This implies that $\mu_{N, N - 1}$ grows linearly with respect to $N$. 
			
			\item Given $N \ge k$, the Motzkin path weight sum $\mu_{N, k}$ is a polynomial of degree $N - k$ in the weights $b_i$ and $\lambda_j$ of the $E$ and $SE$ steps with nonnegative coefficients.  \\
		\end{enumerate}
	\end{cor}
	
	\begin{rem}
		Differences of indices also play an important role in sizes of matrix entries $a_{r, s}$ and $b_{r, s}$ from Proposition \ref{gamtohgvect} (see Lemma \ref{gamtgrowth}). \\
	\end{rem}
	
	\begin{proof}
		\begin{enumerate}
			\item We use induction on $N$. If $N = 2$, we have $\mu_{2, 1} = b_0 + b_1$ since we don't have enough steps to ``cancel out'' a decrease in height. Suppose that $\mu_{N - 1, N - 2} = b_0 + b_1 + \ldots + b_{N - 2}$. Then, Lemma \ref{pathwtrec} implies that
			\begin{align*}
				\mu_{N, N - 1} &= \mu_{N - 1, N - 2} + b_{N - 1} \mu_{N - 1, N - 1} + \lambda_N \mu_{N - 1, N} \\
				&= \mu_{N - 1, N - 2} + b_{N - 1} \\
				&= b_0 + b_1 + \ldots + b_{N - 2} + b_{N - 1}
			\end{align*}
			
			since $\mu_{N - 1, N - 1} = 1$ and $\mu_{N - 1, N} = 0$. \\
			
			\item The only term in the $N = 0$ case is $\mu_{0, 0} = 1$. For the $N = 1$ case, we have $\mu_{1, 0} = b_0$ and $\mu_{1, 1} = 1$. If $N = 2$, we have the following:
			
			\begin{itemize}
				\item $\mu_{2, 0} = b_0^2 + \lambda_1$
				
				\item $\mu_{2, 1} = b_0 + b_1$
				
				\item $\mu_{2, 2} = 1$
			\end{itemize}
			
			We use induction on $N$ to show that the statement holds for $N \ge 3$. Suppose that it holds when the first index $N$ replaced by $N - 1$. By Lemma \ref{wtgrowth}, we have \[ \mu_{N, k} = \mu_{N - 1, k - 1} + b_k \mu_{N - 1, k} + \lambda_{k + 1} \mu_{N - 1, k + 1}. \] The induction assumption implies that $\mu_{N - 1, k - 1}$ is a polynomial in the $b_i$ and $\lambda_j$ of degree $N - k$, $\mu_{N - 1, k}$ is one of degree $N - k - 1$, and $\mu_{N - 1, k + 1}$ is one of degree $N - k - 2$ if they are nonzero. If $k \ge 1$, then $\mu_{N, k} \ne 0$ implies that $\mu_{N - 1, k - 1} \ne 0$ since $(N - 1) - (k - 1) = N - k \ge 0$ and we assumed $N \ge 3$.  Since the coefficient of $\mu_{N - 1, k - 1}$ is nonzero, this implies that $\mu_{N, k}$ is a polynomial in the $b_i$ and $\lambda_j$ of degree $N - k$. The coefficients remain nonnegative since we are taking sums of polynomials in the $b_i$ and $\lambda_j$ with nonnegative coefficients. In the $k = 0$ case, we have $\mu_{N, 0} = b_0 \mu_{N - 1, 0} + \lambda_1 \mu_{N - 1, 1}$. The $N \ge 3$ assumption implies that $\mu_{N - 1, 0} \ne 0$. The induction assumption then implies that $\mu_{N, 0}$ is a polynomial of degree $N$ in the $b_m$ and $\lambda_m$. Since we are adding polynomials in these variables with nonnegative coefficients, the term $\mu_{N, k}$ has nonnegative coefficients as a polynomial in the $b_m$ and $\lambda_m$. \\
			
		\end{enumerate}
	\end{proof}

	In order to clarify our discussion, we will define the following: \\
	
	\begin{defn}
		Consider the input vector $z = (z_0, \ldots, z_N)$ of the inverse map $P^{-1}$ from an initial collection of $N + 1$ unitary orthogonal polynomials (indexed by $0, \ldots, N$) with parameters $b_k > 0$ and $\lambda_k > 0$ for all $1 \ge k \le N$.
		
		\begin{enumerate}
			\item The \textbf{formal $h$-vector} associated to $z$ is the vector $(h_0, \ldots, h_N)$ such that \[ P^{-1}(z) = ( h_N, 2 h_{N - 1}, \ldots, 2^{N - 1} h_1, 2^N h_0 ). \] 
			
			\item Given such an $h$, we will take $z$ to be the \textbf{generalized inverted gamma vector} associated to the collection of unitary orthogonal polynomials $P_0, \ldots, P_N$. \\
		\end{enumerate}

	\end{defn}
	
	\begin{rem}
		We consider $z$ to be the generalization of an inverted and rescaled gamma vector of $h$ since $P^{-1} : ( \gamma_{\frac{d}{2}}, 2 \gamma_{\frac{d}{2} - 1}, \ldots, 2^{\frac{d}{2} - 1} \gamma_1, \gamma_0 ) \mapsto ( h_{\frac{d}{2}}, 2 h_{\frac{d}{2} - 1}, \ldots, 2^{\frac{d}{2} - 1} h_1, 2^{\frac{d}{2}} h_0 )$ if the unitary orthogonal polynomials are the ones associated to the Chebyshev polynomials of the first kind translated by $-1$. \\
	\end{rem}
	
	Given the input vector $z$ of $z \mapsto z \cdot P^{-1}$, the output vectors are of the form 
	\begin{align*}
		q_\ell &= z_\ell \mu_{\ell, \ell} + z_{\ell + 1} \mu_{\ell + 1, \ell} + \ldots + z_N \mu_{N, \ell} \\
		&= z_\ell + z_{\ell + 1} \mu_{\ell + 1, \ell} + \ldots + z_N \mu_{N, \ell}
	\end{align*}

	for all $0 \le \ell \le N$ since $\mu_{\ell, \ell} = 1$. Note that the set of variables used increases as we decrease the index since we start with an earlier variable. \\
	
	As mentioned in the proof of Part 3 of Proposition \ref{gamorthpov}, we would have $q_\ell = 2^\ell h_{N - \ell}$. This would imply that $h_{N - \ell} = 2^{-\ell} q_\ell$ and $h_k = 2^{- (N - k) } q_{ N - k }$. In particular, $h_k \le h_{k + 1}$ if and only if $q_\ell \le 2 q_{\ell - 1}$ (unimodality analogue). Subtracting consecutive terms implies that \[ g_k = 2^{ - \left( N - k + 1 \right) } (2q_{N - k} - q_{N - k + 1}). \] 
	
	We will use these formal $h$-vectors and formal $g$-vectors to describe $z \ge 0$ vectors (analogous to $\gamma(\Delta) \ge 0$) realized by $h$-vectors of Cohen-Macaulay simplicial complexes and simplicial spheres using methods from Section \ref{gennonnegrealiz}. \\
	
	\begin{prop} \textbf{(Realizations of nonnegative generalized inverted gamma vectors) \\}
	
		Consider the realizations of nonnegative generalized inverted gamma vectors $z$ by formal $h$-vectors associated to simplicial complexes.
		\begin{enumerate}
			\item The conditions for realizations of $z \ge 0$ by Cohen--Macaulay simplicial complexes (equivalently of shellable ones) from Theorem \ref{hvectcmshell} are the following:
			
				\begin{itemize}
					\item \textbf{($h_0 = 1$ analogue)} \\
					$q_N = 2^N$
					
					\item \textbf{($M$-vector condition on formal $h$-vector)} \\
					$0 \le 2^{-(\ell - 1)} q_{\ell - 1} \le (2^{-\ell} q_\ell)^{\langle N - \ell \rangle}$  
				\end{itemize}
			
			\item The conditions for realizations of $z \ge 0$ by simplicial spheres from Theorem \ref{gconjsphere} are the following: 
			
				\begin{itemize}
					\item \textbf{($h_0 = 1$ analogue)} \\
					$q_N = 2^N$
					
					\item \textbf{($g \ge 0$ in formal $g$-vector)} \\
					$2q_{\ell - 1} \ge q_\ell$
					
					\item \textbf{($M$-vector condition on formal $g$-vector)} \\
					$2^{ -\ell } (2q_{ \ell - 1} - q_\ell) \le \left[ 2^{ -(\ell + 1) } (2q_\ell - q_{ \ell + 1 }) \right]^{\langle N - \ell \rangle}$ \\
				\end{itemize}
			 
		\end{enumerate}
	\end{prop}

	Fix an initial choice of weights $b_k > 0$ and $\lambda_k > 0$ for $1 \le k \le N$. Since \[ q_{\ell - 1} = z_{\ell - 1} + z_\ell \mu_{\ell, \ell - 1} + \ldots + z_N \mu_{N, \ell - 1} \]
	
	has exactly one variable $z_{\ell - 1}$ not used in \[ q_\ell = z_\ell + z_{\ell + 1} \mu_{\ell + 1, \ell} + \ldots + z_N \mu_{N, \ell}, \]
	
	the methods from Section \ref{gennonnegrealiz} on realizability of nonnegative gamma vectors apply in the same way to generalized inverted gamma vectors $z$. The only differences are that we start from the last index $N$ instead of first index 0 and work backwards with decreasing indices instead of increasing indices while constructing a recursion. \\

	As in Section \ref{gennonnegrealiz}, we start by considering simplicial spheres. The first condition of $g \ge 0$ implies unimodality when $(h_0, \ldots, h_{\frac{d}{2}})$ is the first half and middle term of a $(d + 1)$-element palindromic sequence. Since gamma positivity of reciprocal/palindromic polynomials implies unimodality, we can ask if nonnegativity the generalized inverted gamma vector implies that the $(h_0, \ldots, h_{\frac{d}{2}})$ form a monotonically increasing sequence. In our setting, this holds for sufficiently large parameters. \\

	\begin{prop} \textbf{(Unimodality analogue from nonnegative generalized inverted gamma vectors with large entries) \\} \label{formalgpos}
		Suppose that $z \ge 0$. The entries of the formal $h$-vector form a monotonically increasing sequence if and only if \[ z_{\ell - 1} \ge -\frac{1}{2} ((2 \mu_{\ell, \ell - 1} - 1) z_\ell + (2 \mu_{\ell + 1, \ell - 1} - \mu_{\ell + 1, \ell}) z_{\ell + 1} + \ldots + (2 \mu_{N, \ell - 1} - \mu_{N, \ell}) z_N) \] for all $1 \le \ell \le N$. This indicates that nonnegative entries that decrease sufficiently quickly are realized as nonnegative generalized inverted gamma vectors of monotonically increasing formal $h$-vector entries $(h_0, \ldots, h_N)$. If $2 \mu_{r - 1, s} \le \mu_{r, s}$ for all $0 \le r, s \le N$, the inequality above would hold for all $z \ge 0$. \\
	\end{prop}
	
	As for the remaining $M$-vector condition on the formal $g$-vectors, the situation is similar to nonnegative generalized inverted gamma vectors realized by Cohen--Macaulay simplicial complexes. We are making choices of $z \ge 0$ such that the differences above are large. \\

	\begin{prop} \textbf{(Nonnegative generalized inverted gamma vectors realized by simplicial spheres) \\} \label{geninvgamsphere}
		Suppose that $z_N = 2^N$ and the formal $g$-vector is nonnegative (i.e. $z$ satisfies the conditions of Proposition \ref{formalgpos}).
		
		\begin{enumerate}
			\item Given a choice of $z_k, \ldots, z_{N - 1} \ge 0$ compatible with the $M$-vector inequalities involving $2^{-\ell} (2q_{\ell - 1} - q_\ell)$ and $2^{-(\ell + 1)} (2q_\ell - q_{\ell + 1})$ for $k + 1 \le \ell \le N$, there is a choice of $z_{k - 1} \ge 0$ compatible with the $M$-vector inequalities involving $2^{-k}(2q_{k - 1} - q_k)$ and $2^{-(k + 1)} (2q_k - q_{k + 1})$ if and only if  
			
			\begin{align*}
				2^k [2^{-(k + 1)} (2 q_k - q_{k + 1}) ]^{ \langle N - k \rangle } &= 2^k \left[ 2^{-(k + 1)} \left( 2z_k + \sum_{m = k + 1}^N (2 \mu_{m, k} - \mu_{m, k + 1}) z_m \right) \right]^{ \langle N - k \rangle } \\
				&\ge \sum_{j = k}^N ( 2 \mu_{j, k - 1} - \mu_{j, k} ) z_j.
			\end{align*}
			
			\item If the condition in Part 1 is satisfied, the induced upper bound on $z_{k - 1}$ is \[ z_{k - 1} \le \frac{1}{2} \left( 2^k \left[ 2^{-(k + 1)} \left( 2z_k + \sum_{m = k + 1}^N (2 \mu_{m, k} - \mu_{m, k + 1}) z_m \right) \right]^{ \langle N - k \rangle } - \sum_{j = k}^N ( 2 \mu_{j, k - 1} - \mu_{j, k} ) z_j \right). \]
			
			\item Suppose that $z_N = 2^N$ and $z_k, \ldots, z_{N - 1} \ge 0$ compatible with the $M$-vector inequalities involving $2^{-\ell} (2q_{\ell - 1} - q_\ell)$ and $2^{-(\ell + 1)} (2q_\ell - q_{\ell + 1})$ for $k + 1 \le \ell \le N$. If $z_s \gg 0$ for all $k \le s \le N - 1$, there is a $z_{s - 1} \ge 0$ compatible with the nonnegativity of the difference between the upper bound and the lower bound in Part 1. This means that nonnegative entries $(z_k, \ldots, z_{N - 1})$ that decrease sufficiently quickly (relative to $d$) with entries $z_s$ satisfying designated order $\frac{N - s}{N - s - 1}$ upper bounds with respect to $z_{s + 1}, \ldots, z_N$ yield nonnegative upper bounds to choose from. Sufficiently large inputs would allow extension to the entire generalized inverted gamma vector. \\
		\end{enumerate}
		
		This is a direct analogue of Corollary \ref{sphereposgam}. \\
	\end{prop}
	
	\begin{rem} \textbf{(Comparison with Cohen--Macaulay simplicial complexes and varying $N$) \\} \label{cmorthgen}
		Observations analogous to those in Remark \ref{matcoeffsize} also hold for generalized inverted gamma vectors. Ratios between the $b_i$ and $\lambda_j$ and $N$ would be relevant parameters in versions of the bounds discussed. \\
	\end{rem}

	Before moving on, we will note that multiplication inside pseudopowers $a^{\langle k \rangle}$ (e.g. in simplicial sphere realization results) is similar to multiplication by a modified constant. When the starting index is large, this is close to multiplication by the original constant. \\

	\begin{prop} \textbf{(Multiplication by constants and pseudopower approximations) \\} \label{pseudopowrat}
		Let $a$ be a positive integer. Fix a constant $\beta > 0$ such that $\beta a$ is an integer and a positive integer $k \ge 1$. Then, we have \[ (\beta a)^{\langle k \rangle} \sim  \beta^{ \frac{k + 1}{k} } a^{\langle k \rangle} = \beta^{1 + \frac{1}{k}} a^{\langle k \rangle} \] as $a \to +\infty$. When $k \gg 0$, this is close to multiplication of $a^{\langle k \rangle}$ by $\beta$. \\
	\end{prop}
	
	\begin{proof}
		Recall from Corollary \ref{macpowasymp} that $a^{\langle k \rangle} \sim C_k a^{\frac{k + 1}{k} }$ as $a \to \infty$, where \[ C_m = \frac{(m!)^{ \frac{1}{m} }}{m + 1}. \] If we substitute $\beta a$ in place of $a$, this implies that $(\beta a)^{\langle k \rangle} \sim C_k (\beta a)^{\frac{k + 1}{k}} = \beta^{ \frac{k + 1}{k} } \cdot C_k a^{\frac{k + 1}{k} }$. Then, we have $(\beta a)^{\langle k \rangle} \sim  \beta^{ \frac{k + 1}{k} } a^{\langle k \rangle} = \beta^{1 + \frac{1}{k}} a^{\langle k \rangle} $ as $a \to +\infty$. \\
	\end{proof}
	
	\color{black} 
	
	The remainder of the analysis in Section \ref{gennonnegrealiz} carries over to our setting of more general orthogonal polynomials. \\
	
 	Finally, this framework shows that the analysis in Section \ref{linkcondimp} for simplicial spheres satisfying the link condition also carries over to the more general setting of unitary orthogonal polynomials associated to positive parameters $b_m > 0$ and $\lambda_m > 0$. We explain this in further detail below. \\
	
	\begin{rem} \textbf{(Applications to generalizations of approximations for simplicial spheres satisfying the link condition) \\} \label{linkapproxorth}
		
		\begin{enumerate}
			\item Recall that the analogue of the implication $\gamma(\Delta) \ge 0 \Longrightarrow h_i(\Delta) \ge \binom{d}{i}$ for generalized inverted gamma vectors $z \ge 0$ is $h_k \ge 2^k \mu_{N, N - k}$. Setting $k = 1$, we have the lower bound $h_1 \ge 2 \mu_{N, N - 1}$ for the formal $h_1$ component. In the context of the $g_1(\Delta) \ge d - 1$ assumptions made earlier from $\gamma(\Delta) \ge 0 \Longrightarrow h_1(\Delta) \ge d$ (among other cases), it is natural to look at lower bounds for formal $g_1$ components involving $\mu_{N, N - 1}$. Given constants $b_i > 0$, we see a similar behavior for the $\mu_{N, N - 1}$ since $\mu_{N, N - 1} = b_0 + \ldots + b_{N - 1}$ (Corollary \ref{wtgrowth}), which grows linearly with $N$. \\
			
			\item Suppose that $b_i > 0$ and $\lambda_j > 0$ for all $0 \le i \le N$ and $1 \le j \le N$. An important tool in asymptotic estimates involving simplicial spheres satsifying the link condition was to control behavior of the $g$-vector components after a single index change. This involved finding upper and lower bounds for ratios of the form $\frac{b_{r, s}}{b_{r - 1, s}}$ for terms $b_{i, j} = \binom{d - 2j}{i - j} - \binom{d - 2j}{i - j - 1}$ from Proposition \ref{gamtohgvect}. In our setting, the analogue would be a comparison of the weights $\mu_{r, s}$ with $\mu_{r, s - 1}$. Combining Part 2 of Lemma \ref{wtgrowth} with upper and lower bounds on ratios between the $b_i$ and $\lambda_j$ terms, we can use it and the number of terms involved to study bounds of ratios of $\mu_{r, s}$ with $\mu_{r, s - 1}$ since the difference in degrees as polynomials in the $b_i$ and $\lambda_j$ is 1. \\

		\end{enumerate}

	\end{rem}

\end{document}